%% file: paper.tex
\documentclass[a4paper,11pt,reqno]{amsart}
\usepackage[left=3cm,right=3cm,top=3cm,bottom=3cm]{geometry}
\usepackage[utf8]{inputenc}
\usepackage[english]{babel}
\usepackage{amsmath}
\usepackage{float}
\usepackage{amssymb}
\usepackage{amsfonts}
\usepackage{mathrsfs}
\usepackage{mathtools}
\usepackage{latexsym}
\usepackage{url}
\usepackage{amsthm}
\usepackage[ruled]{algorithm2e}
\usepackage{hyperref}
\usepackage{enumitem}
\usepackage[most]{tcolorbox}
\usepackage{caption}
\usepackage{subcaption}
\usepackage[section]{placeins}

\usepackage{svg}
\usepackage{tikz-cd} 

\theoremstyle{plain}
\newtheorem{theorem}{Theorem}[section]

\newtheorem{defn}[theorem]{Definition}
\newtheorem{cor}[theorem]{Corollary}
\newtheorem{prop}[theorem]{Proposition}

\newtheorem{coro}[theorem]{Corollary}

\theoremstyle{remark}
\newtheorem{remark}[theorem]{Remark}

\input{aux/macros}

\usepackage{cleveref}

\SetKwComment{Comment}{/* }{ */}

\begin{document}
\title[Numerical Invariant Ruelle Distributions]{
	Invariant Ruelle Distributions on Convex-Cocompact Hyperbolic Surfaces\\
	--\\
	A Numerical Algorithm via Weighted Zeta Functions
} 
\author{Philipp Schütte}
\email{pschuet2@math.uni-paderborn.de}
\address{Institute of Mathematics, Paderborn University, Paderborn, Germany}
\author{Tobias Weich}
\email{weich@math.uni-paderborn.de}
\address{Institute of Mathematic, Paderborn University, Paderborn, Germany}
\address{Institute for Photonic Quantum Systems, Paderborn University, Paderborn, Germany}


\begin{abstract}
	We present a numerical algorithm for the computation of invariant Ruelle distributions on convex co-compact hyperbolic surfaces.
	This is achieved by exploiting the connection between invariant Ruelle distributions and residues of meromorphically continued weighted zeta functions established by the authors together with Barkhofen~(2021).
	To make this applicable for numerics we express the weighted zeta as the logarithmic derivative of a suitable parameter dependent Fredholm determinant similar to Borthwick~(2014).
	As an additional difficulty our transfer operator has to include a contracting direction which we account for with techniques developed by Rugh~(1992).
	We achieve a further improvement in convergence speed for our algorithm in the case of surfaces with
	additional symmetries by proving and applying a symmetry reduction of weighted zeta functions.\\
	
\end{abstract}

\maketitle


\setcounter{tocdepth}{2}


\section*{Introduction}\label{intro}

An important notion that has significantly advanced the theory of chaotic, i.e.~hyperbolic, dynamical systems over the last couple of decades is that of
\emph{Pollicott-Ruelle resonances}~\cite{Rue76,Pol81,Liverani.2002,Dyatlov.2019}.
They constitute a discrete subset of the complex plane that provides a spectral invariant refining the ordinary $\mathrm{L}^2$-spectrum of the generator of the dynamics.
From a dynamical systems point of view the central relevance of Pollicott-Ruelle resonances is that they describe the mixing properties of the dynamical system.
Roughly speaking, if there is a simple leading resonance and a spectral gap then the system is exponentially mixing and the gap quantifies the exponential decay rate of the correlation function.
Furthermore if there is an asymptotic spectral gap the other resonances describe additional decay modes~\cite{Tsujii.2010,Zworski.2015}.

Apart from their dynamical importance for mixing properties the distribution of Pollicott-Ruelle resonances as well as the multiplicities of certain Pollicott-Ruelle resonances have proven to be intimately linked to geometrical \cite{CDDP22, FT23} and topological properties \cite{DZ17, KW20, DGRS20} of the underlying manifolds, respectively.
For sufficiently concrete hyperbolic flows such as geodesic flows on Schottky surfaces~\cite{Borthwick.2014, Weich.2016, Pohl.2020} or 3-disc obstacle scattering~\cite{GR89cl,Schuette.2021a,Schuette.2021b,Schuette.2022a} there are even efficient numerical algorithms that allow to calculate the spectrum of Pollicott-Ruelle resonances numerically.
These algorithms do not only enable testing of conjectures (see e.g.~\cite{Dya19}) but the numerical experiments regarding resonances also made possible the discovery of new and unexpected phenomena such as the alignment of resonance into chains subsequently leading to new mathematical theorems~\cite{Wei15, PV19}.

Beyond Pollicott-Ruelle resonances themselves the spectral approach also allows to associate with each Pollicott-Ruelle resonance a flow invariant distribution which has been termed \emph{invariant Ruelle distribution}~\cite{Weich.2021} (see Section~\ref{prelim} for a definition in our setting).
The invariant Ruelle distribution associated with the leading resonance is always an invariant measure which coincides with the Sinai-Ruelle-Bowen (SRB) measure for Anosov flows on compact manifolds and with the Bowen-Margulis-Sullivan measure (following the terminology of Roblin~\cite{Rob03}) for geodesic flows on non-compact manifolds in negative curvature.
This includes for example Schottky surfaces which will be studied in the present paper.
In the case of compact surfaces of constant negative curvature or more generally compact rank-1 locally symmetric spaces it has been proven that the invariant Ruelle distributions coincide with quantum phase space distributions~\cite{Weich.2021}.
These are the so-called Patterson-Sullivan distributions introduced by Anantharaman and Zelditch~\cite{Anantharaman.2007}.
Apart from these results the properties of invariant Ruelle distributions are widely unexplored territory.
Nevertheless motivated by the above results it is very likely that these invariant distributions are closely connected to the finer spectral and dynamical properties of the underlying flow.

The purpose of the present article is to develop a numerical algorithm that allows to concretely calculate the invariant Ruelle distributions and to provide some first numerical experiments that support the claim that invariant Ruelle distributions are an interesting spectral invariant that deserves further investigation.
As a concrete model we chose geodesic flows on Schottky surfaces which are a paradigmatic model for hyperbolic flows on non-compact manifolds.
In practice our numerical approach can and has been applied to $3$-disc scattering as well but the rigorous justification of its numerical convergence would require much more technical tools.
We therefore restrict to the setting of Schottky surfaces in this article.

\subsection*{Statement of results}\label{intro.1}

The primary concern of the present article are certain generalized densities $\mathcal{T}_{\lambda_0}$ called
\emph{invariant Ruelle distributions}.
For a given convex-cocompact hyperbolic surface $\mathbf{X}_\Gamma = \Gamma\backslash\mathbb{H}$ these can be
attached to any member $\lambda_0\in \mathrm{res}(\mathbf{X}_\Gamma)$ of the
discrete set of \emph{Pollicott-Ruelle resonances} $\mathrm{res}(\mathbf{X}_\Gamma) \subset \mathbb{C}$ of the
geodesic flow and they act on smooth test functions on the unit sphere bundle of the surface in a
flow-invariant manner:
\begin{equation*}
	\mathcal{T}_{\lambda_0}:~ f\longmapsto \mathcal{T}_{\lambda_0}(f)\in \mathbb{C} ~,
	\qquad f\in \mathrm{C}^\infty(S\mathbf{X}_\Gamma) ~.
\end{equation*}
These distributions encode non-trivial information about resonant states and
in the particular case of $\lambda_0 = \delta - 1$ with $\delta$ the Hausdorff dimension of the limit set of
$\Gamma$ being the first resonance they coincide with the Bowen-Margulis measures.
For the precise definition of Pollicott-Ruelle resonances and $\mathcal{T}_{\lambda_0}$ refer to
Section~\ref{prelim.1}.

Our means of investigating $\mathcal{T}_{\lambda_0}$ is the \emph{weighted zeta function} for $\mathbf{X}_\Gamma$
\begin{equation*}
	Z_f(\lambda) \defgr \sum_\gamma
	\frac{\mathrm{e}^{-\lambda T_\gamma}}{\vert \det(\mathrm{id} - \mathcal{P}_\gamma)\vert}
	\int_{\gamma^\#} f ~,
\end{equation*}
where the sum extends over all closed geodesics $\gamma$ of $\mathbf{X}_\Gamma$, $T_\gamma$ denotes
the length of $\gamma$, $\gamma^\#$ the primitive closed geodesic corresponding to $\gamma$, and
$\mathcal{P}_{\gamma}$ is the associated linearized Poincar\'{e} map.
It is known that $Z_f(\lambda)$ continues meromorphically to $\mathbb{C}$ and that this continuation is connected
to the invariant Ruelle distributions via the residue formula
\begin{equation*}
	\underset{\lambda = \lambda_0}{\mathrm{Res}}\left[ Z_f(\lambda) \right] = \mathcal{T}_{\lambda_0}(f) ~,
\end{equation*}
which is valid for any test function $f\in\mathrm{C}^\infty(S\mathbf{X}_\Gamma)$. An algorithm for the calculation of
weighted zeta functions therefore translates directly to an algorithm for the invariant Ruelle distributions.

Our main results are exactly such concrete formulae for $Z_f(\lambda)$ feasible for numerical evaluation.
To this end we introduce a \emph{dynamical determinant} for the given surface $\mathbf{X}_\Gamma$ as defined in Definition~\ref{def_dyn_det}:
\begin{equation*}
	d_f(\lambda, z, \beta) \defgr
	\exp\left( - \sum_{k = 1}^\infty \sum_{\gamma^\#}
	\frac{z^{k\cdot n(\gamma^\#)}}{k}
	\frac{\mathrm{e}^{-k\lambda T_\gamma^\# - k \beta \int_{\gamma^\#} f}}
	{\big\vert \mathrm{det}\big( \mathrm{id} - \mathcal{P}_{\gamma^\#}^k \big) \big\vert} \right) ~,
\end{equation*}
where now the sum extends over all primitive closed geodesics $\gamma^\#$ of $\mathbf{X}_\Gamma$ and the number
$n(\gamma^\#)\in\mathbb{N}$ is the so-called word length of $\gamma^\#$.

A priori $d_f(\lambda, z, \beta)$ converges locally uniformly for any fixed $(z, \beta)\in\mathbb{C}^2$ and $\mathrm{Re}(\lambda)$ sufficiently large by the exponential growth of the number of closed geodesics as a function of their maximal length.
This is considerably strengthened in Corollary~\ref{cor_fredholm_det}: $d_f(\lambda, z, \beta)$ continuous holomorphically to $\mathbb{C}$ in all of its three variables.
Furthermore $d_f$ can be used to calculate $Z_f$ due to Corollary~\ref{coro_log_deriv}
\begin{equation*}
	Z_f(\lambda) = \frac{\partial_\beta d_f(\lambda, 1, 0)}{d_f(\lambda, 1, 0)} ~,
\end{equation*}
which provides an independent proof of the meromorphic continuation of the weighted zeta function.\footnote{
	Note that this does not recover the whole strength of the general continuation theorem for $Z_f$ because it does not admit a straightforward interpretation of the poles as resonances and the residues as invariant Ruelle distributions.
}
This together with an everywhere absolutely convergent cycle expansion derived in Corollary~\ref{coro_cycle_exp}
\begin{equation*}
	d_f(\lambda, z, \beta) \defgr 1 + \sum_{n = 1}^\infty d_n(\lambda, \beta) z^n
\end{equation*}
almost puts us in a position to calculate concretely $Z_f(\lambda)$ for a given $\lambda\in\mathbb{C}$.
The missing ingredient to obtain numerically feasible formulae for $d_n(\lambda, \beta)$ is a concrete expression for the \emph{period integrals} $\int_{\gamma^\#} f$.
We present two approaches which are both reasonably cheap to evaluate and at the same time possess a straightforward geometrical interpretation.

In practice it turns out to be very favorable to exploit the symmetries of the underlying geometry as much as possible:
If $\mathbf{X}_\Gamma$ has an additional finite symmetry group $\mathbf{G}$ in a certain sense to be specified in Chapter~\ref{sym_red} then we can prove that $d_f$ factorizes which immediately implies a corresponding decomposition of $Z_f$:
\begin{equation*}
	d_f(\lambda, z,  \beta) = \prod_{\chi\in \widehat{\mathbf{G}}} d_{f, \chi}(\lambda, z,  \beta) ~. \qquad
	Z_f(\lambda) = \sum_{\chi\in \widehat{\mathbf{G}}}
	\frac{\partial d_{f, \chi}(\lambda, 1, 0)}{d_{f, \chi}(\lambda, 1, 0)} ~,
\end{equation*}
with both the product and the sum spanning the finitely many equivalence classes of unitary, irreducible representations of $\mathbf{G}$.
The individual terms $d_{f, \chi}$ exhibit far superior convergence properties compared to $d_f$ which is demonstrated as part of the numerical experiments which round off this article.

\subsection*{Paper organization}\label{intro.2}

The present paper is organized as follows:
In the introductory Chapter~\ref{prelim} the central objects of interest are defined namely Pollicott-Ruelle resonances, the invariant Ruelle distribution associated to such a resonances, and the weighted zeta function $Z_f(\lambda)$ (Chapter~\ref{prelim.1}).
While these definitions can be stated quite generally for the class of open hyperbolic systems our application is to the concrete subclass of geodesic flows on convex-cocompact hyperbolic surfaces (Chapter~\ref{prelim.2}).

As mentioned above the overall goal is the derivation of a numerical algorithm for the calculation of $Z_f(\lambda)$ on convex-cocompact surfaces.
The means to do just this is the dynamical determinant $d_f(\lambda, z, \beta)$ defined in Chapter~\ref{dyn_det}.
Just as $Z_f$ continues meromorphically to $\mathbb{C}$ so does $d_f$ continue holomorphically to $\mathbb{C}$ in all three of its arguments (Chapter~\ref{dyn_det.1}).
Furthermore a suitable logarithmic derivative of $d_f$ coincides with $Z_f$ (Chapter~\ref{dyn_det.2}) meaning that a numerical algorithm for the calculation of $d_f$ immediately transfers to one for $Z_f$.

Our numerical treatment of $d_f$ follows a philosophy termed cycle expansion which circumvents the difficulty that the original definition of $d_f$ in terms of an infinite sum does not converge in the domain where its zeros are located.
This culminates in the derivation of concrete formulae in Chapter~\ref{cycle_exp}.
The subsequent Chapter~\ref{algo} lays out two approaches for the last missing ingredient namely the (geometrically meaningful) numerical approximation of the period integrals $\int_{\gamma^\#} f$ appearing both in $d_f$ and $Z_f$.

While this completes a practically useful algorithm for the calculation of $d_f$ often one can do better in terms of runtime resource requirements by exploiting the inherent symmetries of the underlying surface.
Formalizing this symmetry reduction requires significant additional notation as well as effort and occupies the whole of Chapter~\ref{sym_red}.

The second-to-last Chapter~\ref{numerics} compiles a collection of example plots calculated with the machinery developed so far.
In particular this includes a case study supporting the claim of enhanced convergence properties of the symmetry reduced variants $d_{f, \chi}$ of the dynamical determinant $d_f$.
Finally Chapter~\ref{outlook} presents an outlook on open questions both in the realms of alternative numerical algorithms as well as of interesting numerical experiments to conduct in the future.


\subsection*{Acknowledgments}

This work has received funding from the Deutsche Forschungsgemeinschaft (DFG) (Grant No. WE 6173/1--1 Emmy Noether group “Microlocal Methods for Hyperbolic Dynamics”) as well as  SFB-TRR 358/1 2023 — 491392403 (CRC “Integral Structures in Geometry and Representation Theory”). P.S. was supported by an individual grant from the Studienstiftung des Deutschen Volkes.


\section{Analytical Preliminaries}\label{prelim}

We start this chapter off by giving a short introduction to the analytical theory of weighted zeta functions on open hyperbolic systems as presented in the companion paper \cite{Schuette.2021b} (Section \ref{prelim.1}).
Afterwards we describe the significantly more concrete dynamical setting which we will work in for the remainder of this article:
Geodesic flows on Schottky surfaces (Section~\ref{prelim.2}).

\subsection{Weighted Zeta Functions and Invariant Ruelle Distributions}\label{prelim.1}

Our presentation here follows closely a simplified version of \cite{Schuette.2021b}, see also \cite{Dyatlov.2016b}.
Let a smooth, possibly non-compact, manifold $\mathcal{M}$ and a smooth, possibly non-complete, flow $\varphi_t$ on $\mathcal{M}$ be given.
We make the following dynamical assumptions on $\varphi_t$:
\begin{enumerate}
	\item The generator $X$ of $\varphi_t$ vanishes nowhere,
	\item The trapped set of $\varphi_t$ defined by
	\begin{equation*}
	K \defgr \left\{x\in\mathcal{M} \,\big|\, \varphi_t(x) ~\text{exists}~ \forall t\in\mathbb{R} ~\text{and}~ \overline{\varphi_\mathbb{R}(x)} ~\text{compact} \right\}
	\end{equation*}
	is \emph{compact},
	\item $\varphi_t$ is \emph{hyperbolic on $K$} in the sense that for every $x\in K$ there exists a $\varphi_t$-invariant splitting of the tangent bundle
	\begin{equation*}
	T_x \mathcal{M} = \mathbb{R}\cdot X(x) \oplus E_s(x) \oplus E_u(x)
	\end{equation*}
	such that $E_s, E_u$ depend continuously on $x$ and the differential $\mathrm{d} \varphi_t$ acts in a contracting manner on $E_s$ and an expanding manner on $E_u$:
	\begin{equation*}
	\begin{split}
	\Arrowvert \mathrm{d}\varphi_t(x) v \Arrowvert_{T_{\varphi_t(x)} \mathcal{M}} &\leq C \exp(-c \vert t\vert) \Arrowvert v\Arrowvert_{T_x \mathcal{M}},\qquad t \geq 0,~ v\in E_s(x)\\
	\Arrowvert \mathrm{d}\varphi_t(x) v \Arrowvert_{T_{\varphi_t(x)} \mathcal{M}} &\leq C \exp(-c \vert t\vert) \Arrowvert v\Arrowvert_{T_x \mathcal{M}},\qquad t \leq 0,~ v\in E_u(x) ~.
	\end{split}
	\end{equation*}
\end{enumerate}
In this setting one can define a discrete subset $\mathrm{res}(X) \subseteq \mathbb{C}$ called \emph{Pollicott-Ruelle resonances of $X$} as follows:
Basic function analysis proves that the resolvent\footnote{
	Or rather its restriction to a suitable subset in a compact ambient manifold;
	we disregard this technical detail in the upcoming rather informal discussion.
}
$(X + \lambda)^{-1}$ yields a holomorphic family of bounded operators $\mathrm{L}^2 \rightarrow \mathrm{L}^2$ for sufficiently large $\mathrm{Re}(\lambda) \gg 0$.
Through the usage of anisotropic Sobolev spaces one can show \cite{Dyatlov.2016a,Dyatlov.2016b} that restricting the domain and enlarging the codomain enables meromorphic continuation to $\mathbb{C}$ of $(X + \lambda)^{-1}$ but now as a family of operators $\mathrm{C}^\infty_\mathrm{c} \rightarrow \mathcal{D}'$.
Our resonances are precisely the poles of this continuation.

Given a resonance $\lambda_0 \in \mathrm{res}(X)$ we can compute the residue $\Pi_{\lambda_0}$ of $(X + \lambda)^{-1}$ at $\lambda = \lambda_0$.
The meromorphic continuation outlined above uses the analytic Fredholm theorem as a central tool so the operator $\Pi_{\lambda_0}$ turns out to be of finite rank.
Furthermore, as a consequence of wavefront estimates for $\Pi_{\lambda_0}$, a certain generalization of the Hilbert space trace called a flat trace exists for the family of operators $\mathrm{Tr}^\flat\big( f \Pi_{\lambda_0} \big)$ where $f\in \mathrm{C}^\infty(\mathcal{M})$.
One calls the generalized density
\begin{equation*}
\mathcal{T}_{\lambda_0}: \mathrm{C}^\infty(\mathcal{M})\ni f\longmapsto \mathrm{Tr}^\flat\big( f\Pi_{\lambda_0} \big)\in \mathbb{C}
\end{equation*}
the \emph{invariant Ruelle distribution associated with $\lambda_0$} \cite{Weich.2021}.
We remark here that $\mathcal{T}_{\lambda_0}$ is supported on the trapped set $K$, so in particular it is compactly supported:
$\mathcal{T}_{\lambda_0}\in \mathcal{E}'(\mathcal{M})$.

In \cite{Schuette.2021b} the authors introduced a weighted zeta function which allows for a significantly more concrete approach to invariant Ruelle distributions.
They associated with the flow $\varphi_t$ and a weight $f\in \mathrm{C}^\infty(\mathcal{M})$ the complex function
\begin{equation} \label{eq_def_weighted_zeta}
Z_f(\lambda) \defgr \sum_\gamma \frac{\mathrm{e}^{-\lambda T_\gamma}}{\big\vert \det\big(\mathrm{id} - \mathcal{P}_\gamma \big) \big|} \int_\gamma f ~,
\end{equation}
with the sum extending over all closed trajectories of $\varphi_t$, $T_\gamma$ denoting the period length of the trajectory $\gamma$, and $\mathcal{P}_\gamma$ being the linearized Poincar\'{e} map associated with $\gamma$.
The latter is simply the differential of $\varphi_t$ restricted to stable and unstable directions:
\begin{equation*}
\mathcal{P}_{\gamma(t)} \defgr \mathrm{d} \varphi_{-t}(\gamma(t)) \big|_{E_s(\gamma(t)) \oplus E_u(\gamma(t))} ~,
\end{equation*}
where the dependence on the base point can be omitted when taking the determinant.
While $Z_f$ converges absolutely only for $\mathrm{Re}(\lambda) \gg 0$ it continues meromorphically to the whole complex plane \cite[Theorem~1.1]{Schuette.2021b}.
The circumstance that $Z_f$ is a useful function to consider stems from the following residue formula relating $Z_f$ to Ruelle distributions:
\begin{equation}
\mathcal{T}_{\lambda_0}(f) = \underset{\lambda = \lambda_0}{\mathrm{Res}}\left[ Z_f(\lambda) \right], \qquad f\in \mathrm{C}^\infty(\mathcal{M}) ~.
\end{equation}

\subsection{Introduction to Schottky Surfaces}\label{prelim.2}

In this section we provide a short introduction to Schottky, i.e. convex cocompact hyperbolic, surfaces.
The material presented here is quite classic and can be found in e.g. \cite{Borthwick.2016,Dalbo.2011}.

Schottky groups are discrete subgroups of the group $G\defgr \mathrm{PSL}(2, \mathbb{R})$ of orientation preserving isometries of the \emph{upper half plane}
\begin{equation*}
\mathbb{H} \defgr \{z\in \mathbb{C} \,|\, \mathrm{Im}(z) > 0\} ~,
\end{equation*}
equipped with the Riemannian metric
\begin{equation*}
\mathrm{g}_\mathbb{H}(x + \mathrm{i}y) \defgr \frac{\mathrm{d}x^2 + \mathrm{d}y^2}{y^2} ~.
\end{equation*}

The geodesics of $(\mathbb{H}, \mathrm{g}_\mathbb{H})$ are given by semicircles centered on the real line and by straight lines parallel to the imaginary axis.
We denote the associated geodesic flow on the unit tangent bundle $S\mathbb{H} = \mathbb{H}\times \{v\in\mathbb{C} \,|\, \vert v\vert = 1\}$ by $\varphi_t$.

We will introduce Schottky groups by their dynamics on $\mathbb{H}$
To this end let the reader be reminded that $G$ acts on the whole Riemann sphere $\mathbb{C}\cup \{\infty\}$ and therefore on $\mathbb{H}$ via Moebius transformations
\begin{equation*}
\begin{pmatrix}
a & b\\
c & d
\end{pmatrix} \cdot z \defgr \frac{az + b}{cz + d} ~.
\end{equation*}
This action extends to $S\mathbb{H}$ by acting on fiber coordinates via the derivative of a Moebius transformation.
Now recall that non-trivial isometries can be classified according to the absolute value of the trace of their matrix representation which determines a certain dynamical and fixed point behavior on the compactification $\overline{\mathbb{H}} \defgr \mathbb{H} \cup \mathbb{R} \cup \{\infty\}$:
\begin{enumerate}
	\item $\vert \mathrm{tr}\vert > 2$: hyperbolic isometry, two distinct fixed points in $\mathbb{R} \cup \{\infty\}$
	\item $\vert \mathrm{tr}\vert = 2$: parabolic isometry, one unique fixed point on $\mathbb{R} \cup \{ \infty \}$,
	\item $\vert \mathrm{tr}\vert < 2$: elliptic isometry, one unique fixed point in $\mathbb{H}$.
\end{enumerate}
We summarize shortly the dynamical properties of hyperbolic isometries because those will be particularly important for us:
Given a hyperbolic $g$ there exists a unique hyperbolic geodesic $\gamma(t)$ such that its endpoints at infinity $\lim_{t\rightarrow\pm\infty} \gamma(t)$ are the attracting and repelling fixed points of $g$.
One calls $\gamma(\mathbb{R})\subseteq \mathbb{H}$ the axis of $g$ and $g$ acts on $\gamma$ as a translation by a fixed hyperbolic distance $\ell(g) > 0$.
This distance is called the \emph{displacement length} of the isometry $g$.

With this in mind we can define \emph{Schottky groups} as those discrete, free subgroups $\Gamma < \mathrm{PSL}(2, \mathbb{R})$ which are finitely generated by hyperbolic isometries.
If $\{g_1, \hdots, g_r\}$ is a generating set for $\Gamma$ of minimal size then $r$ is called the \emph{rank} of $\Gamma$ and by a classical result of Maskit \cite{Maskit.1967} there exists a collection of open Euclidean discs $D_1, \hdots, D_{2r}$ with disjoint closures and centered on the real line such that
\begin{equation} \label{eq_schottky_dyn}
g_i(\partial D_i) = \partial D_{i + r}, \qquad g_i(D_i) = \mathbb{C} \setminus D_{i + r} ~.
\end{equation}
We call these circles \emph{fundamental circles} for the chosen generators and a fundamental domain for the action of $\Gamma$ on $\mathbb{H}$ is given by their complement $\mathbb{H} \setminus \bigcup_{i = 1}^{2r} D_i$.
We will refer to this particular fundamental domain as the \emph{canonical} one.
For an illustration see Figure~\ref{fig1}.
\begin{figure}[h]
	\includegraphics[scale=0.35]{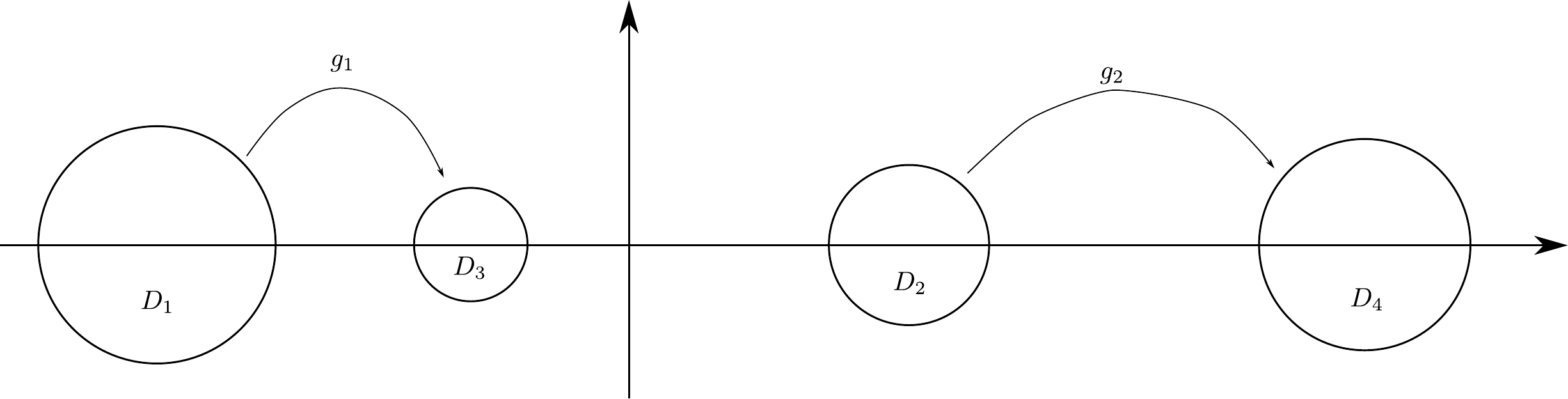}
	\caption{The fundamental circles and group element actions for a Schottky group of rank $r=2$ (three-funnel surface, see Figure~\ref{fig2} below).}
	\label{fig1}
\end{figure}

Conversely one can define a Schottky group of rank $r$ by fixing open discs $D_1, \hdots, D_{2r}$ with pairwise disjoint closures and centered on the real line and then taking the group generated by hyperbolic elements $g_1, \hdots, g_r$ satisfying \eqref{eq_schottky_dyn}.

For convenience of notation one usually defines $g_{i + r} \defgr g_i^{-1}$ for $1 \leq i\leq r$ and subsequently extends the indexing of generators to arbitrary $i > 2r$ by defining $g_i \defgr g_{(i\,\text{mod}\, 2r)}$.
Setting $D_i \defgr D_{(i\,\text{mod}\, 2r)}$ property \eqref{eq_schottky_dyn} continues to hold for these extended definitions.
We will also frequently use indices in the quotient ring $\mathbb{Z} \slash 2r\mathbb{Z}$.

To every Schottky group $\Gamma$ we associate a \emph{Schottky surface} obtained as the quotient space $\mathbf{X}_\Gamma \defgr \Gamma \setminus \mathbb{H}$.
By discreteness of $\Gamma$ the set $\mathbf{X}_\Gamma$ can be equipped with a canonical smooth structure.
Furthermore the metric $\mathrm{g}_\mathbb{H}$ is $\Gamma$-invariant and thus descends to $\mathbf{X}_\Gamma$ again making the quotient space a Riemannian manifold of constant negative curvature.
It is non-compact, of infinite volume, and the geodesic flow on its unit tangent bundle fits into the framework of open hyperbolic systems as presented in Section~\ref{prelim.1}, see~\cite[Section~6.3]{Dyatlov.2016b}.
In particular we can speak about Ruelle resonances, weighted zeta functions, and invariant Ruelle distributions on Schottky surfaces.

At first glance the non-compactness of Schottky surfaces might make them seem quite difficult in terms of numerical treatment.
Their suitability for our purposes follows from the particularly simply structure of free groups combined with the more general correspondence between group elements and closed geodesics \cite[Proposition~2.25]{Borthwick.2016}:
The closed oriented geodesics of a Schottky surface $\mathbf{X}_\Gamma$ are in bijection to the conjugacy classes of the group $\Gamma$ and the length of a geodesic $T_\gamma$ coincides with the displacement length $\ell(g)$ for any element $g$ in the associated conjugacy class.
We denote by $\gamma(g)$ the image of the conjugacy class containing the group element $g\in\Gamma$ under this bijection record the important relation \cite[Eq.~(15.3)]{Borthwick.2016}
\begin{equation}\label{eq_deriv_length}
g'(x_-) = \mathrm{e}^{-\ell(g)} = \mathrm{e}^{-T_{\gamma(g)}}
\end{equation}
between displacement length and attracting fixed point $x_+$ of $g$.

Because $\Gamma$ is finitely generated we can represent its elements as sequences $(i_1, \hdots, i_n)\in \mathcal{A}^*$ over the alphabet $\mathcal{A}\defgr (\mathbb{Z} \slash 2r\mathbb{Z})$ by defining
\begin{equation*}
g_{(i_1, \hdots, i_n)} \defgr g_{i_n}\cdots g_{i_1} ~.
\end{equation*}
Any such sequence $(i_1, \hdots, i_n)$ is called a \emph{word}, the $i_j$ its \emph{letters}, and $n$ its \emph{(word) length}.
From the fact that $\Gamma$ is free it follows immediately that the map from words to group elements becomes a bijection if we restrict to \emph{reduced words}, i.e. elements of $\{(i_1, \hdots, i_n)\in \mathcal{A}^* \,|\, i_j\neq i_{j + 1} + r\, \forall\, 1\leq j\leq n-1\}$.

Now the conjugacy classes of $\Gamma$ can be represented by words of minimal length and such a representation is unique modulo cyclic shifts of its letters.
We denote the set of all possible indices of such representatives of length $n$ by $\mathcal{W}_n$, i.e.:
\begin{equation*}
\mathcal{W}_n \defgr \big\{ (i_1, \hdots, i_n) \in \mathcal{A}^n \,\big|\, i_j\neq i_{j + 1} + r\, \forall\, 1\leq j\leq n-1 ~\text{and}~ i_n\neq i_1 + r \big\} ~.
\end{equation*}
We will also call $\mathcal{W}_n$ the \emph{set of closed words of length $n$}.
If the geodesic $\gamma\subseteq \mathbf{X}_\Gamma$ is represented by $g_{(i_1, \hdots, i_n)}$ with $(i_1, \hdots, i_n)\in \mathcal{W}_n$ then we denote by $n(\gamma) \defgr n$ the length of its minimal representation and call this the \emph{word length} of $\gamma$.

Finally, given a function $f: \bigcup_{i\neq j} D_i\times D_j\rightarrow \mathbb{C}$ we define its \emph{iteration along the group element} $g = g_{i_n}\cdots g_{i_1}$ to be the product
\begin{equation*} \label{eq_potential_iterate}
f_g \defgr f_{(i_1, \hdots, i_n)} \defgr f(x_-, x_+) \cdot f(g_{i_1} x_-, g_{i_1} x_+) \cdot \hdots \cdot f(g_{i_{n-1}}\cdots g_{i_1} x_-, g_{i_{n-1}}\cdots g_{i_1} x_+) ~,
\end{equation*}
where$(x_-, x_+)$ are the repelling and attracting fixed points of $g$.

In our numerics we will mostly be dealing with Schottky surfaces of rank $r = 2$.
From the topological standpoint there are only two possibilities for such surfaces corresponding to distinct combinatorics of their actions on the canonical fundamental domain.
These are depicted in Figure~\ref{fig2}.
\begin{figure}[h]
	\includegraphics[scale=0.3]{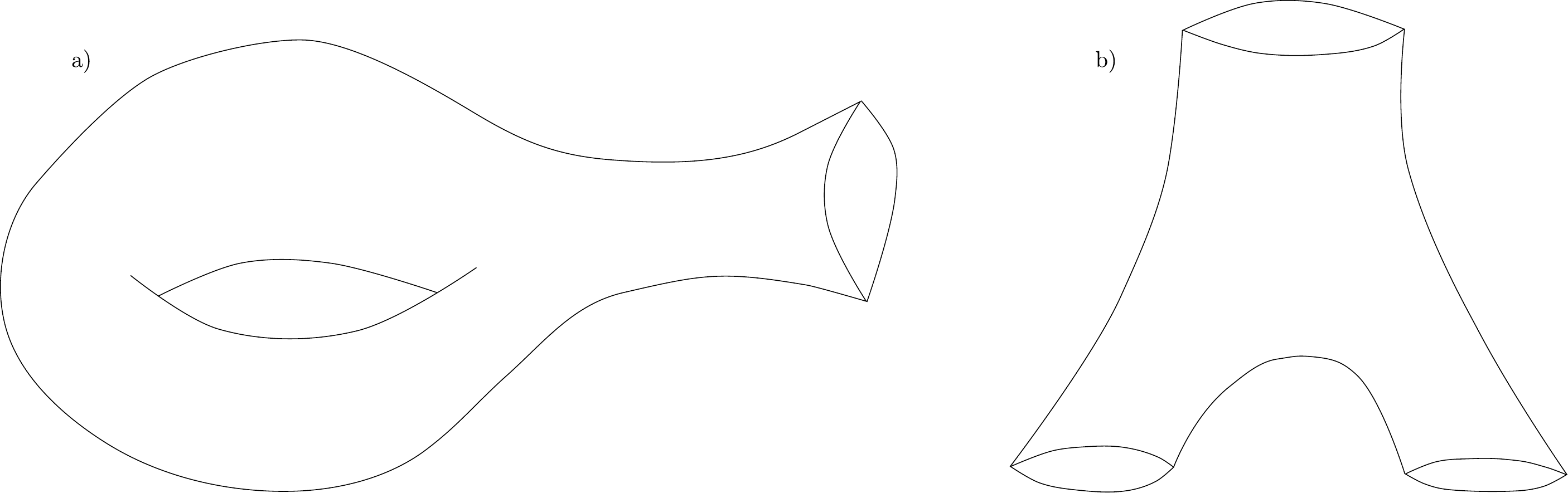}
	\caption{The two topological possibilities for Schottky surfaces of rank two: (a) Funneled torus. (b) Three-funnel surface.}
	\label{fig2}
\end{figure}


\section{Introducing Dynamical Determinants}\label{dyn_det}

In this section we introduce the central object from which our numerically feasible formula for weighted zeta functions $Z_f$ will be derived:
The dynamical determinant.

\begin{defn}\label{def_dyn_det}
	Let $\mathbf{X}_\Gamma$ be a Schottky surface and $f\in \mathrm{C}^\infty\left(S \mathbf{X}_\Gamma\right)$.
	Then the associated \emph{dynamical determinant} at $(\lambda, z, \beta)\in \mathbb{C}^3$ is formally defined as
	\begin{equation}\label{def_df}
		d_f(\lambda, z, \beta) \defgr
		\exp\left( - \sum_{k = 1}^\infty \sum_{\gamma^\#}
		\frac{z^{k\cdot n(\gamma^\#)}}{k}
		\frac{\mathrm{e}^{-k\lambda T_\gamma^\# - k \beta \int_{\gamma^\#} f}}
		{\big\vert \mathrm{det}\big( \mathrm{id} - \mathcal{P}_{\gamma^\#}^k \big) \big\vert} \right) ~,
	\end{equation}
	where the sum stretches over all primitive closed geodesics $\gamma^\#$.
\end{defn}

A numerical implementation of the individual summands appearing in this dynamical determinant is straightforward.
One simply uses the symbolic coding of closed geodesics and the fact that eigenvalues of the linearized Poincar\'{e} map are given by exponentials of lengths.
The concrete calculation of periods integrals $f_{\gamma^\#}$ is slightly less obvious but can be done quite efficiently after certain simplifications.
Two alternatives will be presented in Section~\ref{algo}.
Calculation of $d_f$ itself demands additional attention as the defining formula \eqref{def_dyn_det} does not converge on the whole complex plane but for sufficiently small $\vert z\vert$ only.
How to overcome this difficulty is the main content of Section~\ref{cycle_exp}.

The question remains how we can exploit $d_f$ for the calculation of $Z_f$, the latter being the actual object of interest for us.
We will answer this question in two steps in the upcoming sections:
First we prove that $d_f$ is analytic in its variables in Section~\ref{dyn_det.1}.
Its logarithmic derivative $\partial_\beta \log d_f(\lambda, 1, 0)$ is thus meromorphic and Section~\ref{dyn_det.2} demonstrates that it coincides with $Z_f(\lambda)$.

\subsection{Dynamical Determinants as Fredholm Determinants}\label{dyn_det.1}

In this section we prove that the dynamical determinant of Definition \ref{def_dyn_det} actually yields a well-defined holomorphic function of $(\lambda, z, \beta)$.
One possible way to do this would be applying microlocal techniques and anisotropic Sobolev spaces as for example presented in Baladi's book \cite{Baladi.2018}.
While generally feasible for the problem at hand, the methods presented in \cite{Baladi.2018} are specifically geared towards the setting of low regularity:
Applying it to Schottky surfaces would discard the additional information provided by the fact that Schottky groups are defined in terms of holomorphic functions.

We therefore turn towards the ideas developed by Rugh~\cite{Rugh.1992,Rugh.1996} in the analytic setting.
Instead of inferring analyticity of $d_f$ directly from~\cite[Theorem~1]{Rugh.1996} we provide a self-contained proof by adapting his techniques and notation to our concrete setting of Schottky surfaces.
Besides being self-contained this will later on offer a convenient entry point for symmetry reduction (Section~\ref{sym_red}) and should come in handy for the development of alternative numerical algorithms (Section~\ref{outlook}).

Before diving into the proof we give some definitions:
Let $\mathcal{D}\subseteq\mathbb{C}$ be an open disc in the complex plane.
Then we denote by
\begin{equation*}
\mathcal{H}^2(\mathcal{D}) \defgr \{f\in \mathrm{L}^2(\mathcal{D}) \,|\, f ~\text{is holomorphic} \}
\end{equation*}
the Bergman space of square-integrable, holomorphic functions on $\mathcal{D}$.
Furthermore we denote its dual space by $\mathcal{H}^{-2}(\mathcal{D})$ and identify it with the Bergman space $\mathcal{H}^2(\mathbb{C}\setminus\mathcal{D})$ via the bilinear pairing
\begin{equation*}
\langle u, v\rangle \defgr \int_{\partial \mathcal{D}} u(z) v(z) \frac{\mathrm{d}z}{2\pi\mathrm{i}}~, \qquad u\in \mathcal{H}^2(\mathcal{D}),\, v\in \mathcal{H}^{-2}(\mathcal{D}) ~.
\end{equation*}
If $\mathcal{D}$ is the unit disc $\mathbb{D} = \{z\in\mathbb{C} \,|\, \vert z\vert < 1\}$ then $\psi_n(z) \defgr \sqrt{(n + 1) / \pi} z^n$ defines an orthonormal basis for $\mathcal{H}^2(\mathcal{D})$ with dual basis given by $\psi_n^*(z) \defgr \sqrt{\pi / (n + 1)} z^{-n - 1}$, i.e. $\langle \psi_n, \psi_m^*\rangle = \delta_{n, m}$.
An arbitrary disc is easily reduced to this case by translation and scaling.

A last ingredient which will appear in the following proof is the so-called \emph{Bergman kernel} $K_\mathcal{D}(z, w)$ of $\mathcal{D}$.
This reproducing kernel satisfies the defining relation\footnote{
	Classically one expresses this relation as an integral over $\mathcal{D}$ instead of $\partial \mathcal{D}$.
	In our proof below we will need to deal with integrals over $\partial \mathcal{D}$, though, so we use this slightly less common definition.
}
\begin{equation*}
f(z) = \int_{\partial \mathcal{D}} K_\mathcal{D}(z, w) f(w) \mathrm{d}w  ~,
\end{equation*}
and can be expressed as a sum over an orthonormal basis.
Below we will employ an explicit expression for the unit disc and again reduce the case of a general disc by translation and scaling.

With these prerequisites at hand the main theorem now reads as follows:
\begin{theorem}[\cite{Rugh.1996},~Thm.~1] \label{thm_fredholm_det}
	Let $\Gamma$ be a rank-$r$ Schottky surface with generators $g_1, \hdots, g_r$ and fundamental circles $D_1, \hdots, D_{2r}$.
	Given a potential $V$ which is analytic in a neighborhood of $\bigcup_{i\neq j} D_i\times D_j$, the associated \emph{transfer operator} defined by the formula
	\begin{equation} \label{eq_def_transfer}
	\begin{split}
	&\quad \mathrm{L}_V: \bigoplus_{i\neq j} \mathcal{H}^{-2}(D_i) \otimes \mathcal{H}^2(D_j) \longrightarrow \bigoplus_{i\neq j} \mathcal{H}^{-2}(D_i) \otimes \mathcal{H}^2(D_j) ~,\\
	&\langle \mathrm{L}_V u(z_1, z_2), v(z_1)\rangle\bigg|_{\substack{v\in \mathcal{H}^2(D_i) \\ z_2\in D_j}} \defgr \int_{\partial D_i} V(z_1, z_2) v(z_1) u(g_i z_1, g_i z_2) \frac{\mathrm{d}z_1}{2\pi \mathrm{i}} ~, \quad i\neq j ~,
	\end{split}
	\end{equation}
	is a well-defined trace-class operator and its Fredholm determinant is an analytic function that satisfies the following identity for sufficiently small $\vert z \vert$:
	\begin{equation} \label{eq_transfer_det}
	\det\left( \mathrm{id} - z \mathrm{L}_V \right) = \exp\left( - \sum_{n = 1}^\infty \frac{z^n}{n} \sum_{w\in \mathcal{W}_n} \frac{V_w}{(1 - \mathrm{e}^{-T_{\gamma(g_w)}}) (\mathrm{e}^{T_{\gamma(g_w)}} - 1)} \right) ~.
	\end{equation}
\end{theorem}

\begin{proof}
	As in~\cite[Lemma~15.7]{Borthwick.2016} we may deduce the trace-class property of~\eqref{eq_def_transfer} from exponential bounds on the singular values $\mu_j(\mathrm{L}_V)$ of the transfer operator.
	To obtain such bounds it suffices to consider the norms of the images of an orthonormal basis of $\mathcal{H}^{-2}(D_i) \otimes \mathcal{H}^2(D_j)$ under the components of $\mathrm{L}_V$ by combining the additive Fan inequality \cite[A.25]{Borthwick.2016}, c.f. Appendix~\ref{app_fredholm_det}, with the basic estimate
	\begin{equation*}
	\mu_j(A) \leq \sum_{i = j}^\infty \Arrowvert A\phi_i\Arrowvert ~,
	\end{equation*}
	obtained via the min-max estimate \cite[A.23]{Borthwick.2016}, see also Appendix~\ref{app_fredholm_det}, and valid for any bounded operator $A: \mathcal{H}\rightarrow \mathcal{H}'$ between Hilbert spaces $\mathcal{H}, \mathcal{H}'$ and an orthonormal basis $\{\phi_i\}$ of $\mathcal{H}$.
	
	Now the potential $V$ is bounded on the closure of the poly-discs $D_i\times D_j$ and therefore acts as a bounded operator.
	As in \cite[Eq.~(15.15)]{Borthwick.2016} are thus left with the task of estimating the pullback action
	\begin{equation*}
	\Arrowvert (\psi_n^{i, *}\circ g_i) \otimes (\psi_n^j\circ g_i) \Arrowvert_{\mathrm{L}^2((\mathbb{C} \setminus D_i)\times D_j)} ~, \quad i\neq j ~,
	\end{equation*}
	where the orthonormal basis elements $\psi_n^i$ and $\psi_n^{i, *}$ are obtained from $\psi_n$ and $\psi_n^*$ by a suitable scaling and translating to $D_i$ and $\mathbb{C}\setminus D_j$, respectively.
	Now we observe that $g_i$, $i\neq j$, acts in contracting fashion on $D_j$ and in expanding fashion on $D_i$.
	It thus follows that for some constants $C, c > 0$ we have
	\begin{equation*}
	\Arrowvert (\psi_n^{i, *}\circ g_i) \otimes (\psi_n^j\circ g_i) \Arrowvert_{\mathrm{L}^2((\mathbb{C} \setminus D_i)\times D_j)} \leq C \mathrm{e}^{-c n} ~, \quad i\neq j ~,
	\end{equation*}
	where the constants dependent on the rate of contraction/expansion and the derivative of $g_i$. Now this in turn implies the estimate
	\begin{equation} \label{eq_sing_value_estimate}
	\mu_j(\mathrm{L}_V) \leq C' \sum_{n = j}^\infty \mathrm{e}^{-c n} = C' \frac{\mathrm{e}^{-c j}}{1 - \mathrm{e}^{-c}} ~,
	\end{equation}
	for some constant $C' > 0$ which additionally depends on the potential $V$ and finally proving the trace-class claim.
	
	To demonstrate the determinant formula \eqref{eq_transfer_det} we proceed similarly to \cite[Theorem~15.10]{Borthwick.2016} by first rewriting
	\begin{equation} \label{eq_log_det}
	\det\left( \mathrm{id} - z \mathrm{L}_V \right) = \exp\left( \mathrm{Tr} \left( \log(\mathrm{id} - z\mathrm{L}_V) \right) \right) = \exp\left( - \sum_{n = 1}^\infty \frac{z^n}{n} \mathrm{Tr}\left( \mathrm{L}_V^n \right) \right) ~.
	\end{equation}
	for values of $\vert z\vert$ small enough that the logarithm converges.
	The traces of iterates $\mathrm{L}_V^n$ decompose further in terms of components of $\mathrm{L}_V$, i.e. its restrictions in domain and codomain to the individual spaces $\mathcal{H}^{-2}(D_i)\otimes\mathcal{H}^2(D_j)$ for $i\neq j$.
	Multiplying out the defining formula \eqref{eq_def_transfer} and collecting all resulting components we see that only the	\emph{diagonal entries} of the form
	\begin{equation*}
	\mathrm{L}^{i_1, \hdots, i_n}_V: \mathcal{H}^{-2}(D_{i_n + r})\otimes \mathcal{H}^2(D_{i_1}) \longrightarrow \mathcal{H}^{-2}(D_{i_n + r})\otimes \mathcal{H}^2(D_{i_1})
	\end{equation*}
	can contribute to the trace. As a result we obtain
	\begin{equation*}
	\mathrm{Tr}\left( \mathrm{L}^n_V \right) = \sum_{i_1, \hdots, i_n} \mathrm{Tr}\left( \mathrm{L}^{i_1, \hdots, i_n}_V \right) ~,
	\end{equation*}
	where $i_j + r \neq i_{j + 1}$ for $1\leq j\leq n - 1$, $i_n + r \neq i_1$, and the diagonal components are explicitly given by
	\begin{equation} \label{eq_transfer_comp}
	\begin{split}
	&\qquad \left\langle \mathrm{L}^{i_1, \hdots, i_n}_V u(z_1, z_2), v(z_1) \right\rangle \bigg|_{\substack{v\in \mathcal{H}^2(D_{{i_n} + r}) \\ z_2\in D_{i_1}}} \\
	&= \int_{\partial D_{{i_n} + r}} V_{g_{(i_1, \hdots, i_n)}^{-1}}(z_1, z_2) v(z_1) u\big( g_{(i_1, \hdots, i_n)}^{-1}(z_1), g_{(i_1, \hdots, i_n)}^{-1}(z_2) \big) \frac{\mathrm{d} z_1}{2\pi\mathrm{i}} ~,
	\end{split}
	\end{equation}
	for any element $u\in \mathcal{H}^{-2}(D_{{i_n} + r}) \otimes \mathcal{H}^2(D_{i_1})$ and using the shorthand notations $g_{i_1}^{-1} \cdots g_{i_n}^{-1} = g_{(i_1, \hdots, i_n)}^{-1}$ as well as $V_{g_{(i_1, \hdots, i_n)}^{-1}}$ as defined in \eqref{eq_potential_iterate}.
	To see this observe that by \eqref{eq_def_transfer} the group element to apply is determined by the tensor product's first factor which in this particular instance yields $g_{i_n + r} = g_{i_n}^{-1}$ as the first element.
	A similar remark holds for the subsequent generators.

	Now the trace of an operator of the form given in \eqref{eq_transfer_comp} can be calculated as follows (with $g\defgr g_{(i_1, \hdots, i_n)}^{-1}$ for convenience):
	\begin{equation}
	\begin{split}
	&\qquad \mathrm{Tr}\left( \mathrm{L}^{i_1, \hdots, i_n}_V \right)\\
	&= \sum_{k, l \geq 0} \left\langle \mathrm{L}^{i_1, \hdots, i_n}_V \psi^{i_n + r, *}_k \otimes \psi^{i_1}_l, \psi^{i_n + r}_k \otimes \psi^{i_1, *}_l \right\rangle\\
	&= \sum_{k, l\geq 0} \int_{\partial D_{i_n + r}} \int_{\partial D_{i_1}} V_g(z_1, z_2) \psi^{i_n + r, *}_k(g(z_1)) \psi^{i_n + r}_k(z_1) \psi^{i_1}_l(z_2) \psi^{i_1, *}_l(g(z_2)) \frac{\mathrm{d} z_2}{2\pi\mathrm{i}} \frac{\mathrm{d} z_1}{2\pi\mathrm{i}}\\
	&= \int_{\partial D_{i_n + r}} \int_{\partial D_{i_1}} V_g(z_1, z_2) K_{i_n + r}(z_1, g(z_1)) K_{i_1}(g(z_2), z_2) \frac{\mathrm{d} z_2}{2\pi\mathrm{i}} \frac{\mathrm{d} z_1}{2\pi\mathrm{i}} ~,
	\end{split}
	\end{equation}
	where $K_{i_n + r}$ and $K_{i_1}$ denote the Bergman kernels of the domains $\mathbb{C} \setminus D_{i_n + r}$ and $D_{i_1}$, respectively.
	After translation and scaling we may assume that both discs coincide with the unit disc $\mathbb{D}$ such that $K_{i_n + r}(x, y) = K_{i_1}(x, y) = \sum_{k = 0}^\infty \frac{x^k}{y^{k + 1}} = \frac{1}{y - x}$.
	A simple application of the residue theorem then yields
	\begin{equation*}
	\begin{split}
	&~\quad \int_{\partial D_{i_n + r}} \int_{\partial D_{i_1}} V_g(z_1, z_2) K_{i_n + r}(z_1, g(z_1)) K_{i_1}(g(z_2), z_2) \frac{\mathrm{d} z_2}{2\pi\mathrm{i}} \frac{\mathrm{d} z_1}{2\pi\mathrm{i}} \\
	&= \int_{\partial \mathbb{D}} \int_{\partial \mathbb{D}} \widetilde{V}_g(z_1, z_2) \frac{1}{\widetilde{g}(z_1) - z_1} \frac{1}{z_2 - \widehat{g}(z_2)} \frac{\mathrm{d} z_2}{2\pi\mathrm{i}} \frac{\mathrm{d} z_1}{2\pi\mathrm{i}} \\
	&= \frac{V_g(x_-, x_+)}{(g'(x_-) - 1)(1 - g'(x_+))} ~,
	\end{split}
	\end{equation*}
	with a rescaled version $\widetilde{V}_g$ of $V_g$ and two different rescalings $\widetilde{g}$ and $\widehat{g}$ of $g$.
	In the last line the repelling and attracting fixed points $(x_-, x_+)$ of $g$ appear because they are the unique solutions of $g(z) = z$ on $D_{i_n + r}$ and $D_{i_1}$, respectively.
	
	In summary we obtain for the trace of diagonal entries of $n$-fold iterates of the transfer operator the concrete formula
	\begin{equation} \label{eq_comp_trace}
	\mathrm{Tr}\left( \mathrm{L}^{i_1, \hdots, i_n}_V \right) = \frac{V_g(x_-, x_+)}{(g'(x_-) - 1)(1 - g'(x_+))} = \frac{V_g(x_-, x_+)}{(\mathrm{e}^{\ell(g)} - 1)(1 - \mathrm{e}^{-\ell(g)})} ~,
	\end{equation}
	where second equality follows immediately from \eqref{eq_deriv_length}.
	If we plug \eqref{eq_comp_trace} into \eqref{eq_log_det} we obtain \eqref{eq_transfer_det} because the constraints on the finite sequence $(i_1, \hdots, i_n)$ mentioned above guarantees that the sum runs exactly over the set of closed words $\mathcal{W}_n$.
\end{proof}

Theorem~\ref{thm_fredholm_det} immediately yields a representation of $d_f$ as a Fredholm determinant by choosing an appropriate potential.
It must obviously include period integrals $\int_\gamma f$ over geodesics $\gamma\subseteq S\mathbf{X}_\Gamma$ which we encode in a fashion similar to \cite{Anantharaman.2007}:
Given a weight $f\in \mathrm{C}^\omega\left(S\mathbf{X}_\Gamma \right)$ we can consider its lift to $S\mathbb{H}$ which in turn can be expressed as a function on $\mathbb{R} \times (\partial\mathbb{H}^2\setminus \Delta)$, where $\Delta$ denotes the diagonal of $\partial\mathbb{H}^2$.
In these so-called \emph{Hopf coordinates} a point $(t, x_1, x_2)$ maps to $\gamma_{(x_1, x_2)}(t)$ with $\gamma_{(x_1, x_2)}$ a geodesic with $x_1$ and $x_2$ as its endpoints at infinity and a suitably chosen starting point $\gamma_{(x_1, x_2)}(0)$.\footnote{
	One advantage of these coordinates is the fact that the geodesic flow simply acts by translation in the first component.
	We will come back to these coordinates in Section~\ref{algo} where we use them to derive approximations for period integrals practical for numerical implementation.
}
Keeping in mind that closed geodesics of $\mathbf{X}_\Gamma$ possess endpoints at infinity in the intersections $I_i \defgr \overline{D_i}\cap \partial \mathbb{H}$, we define the following function which is real-analytic, c.f.~\cite[Section~7.2]{Anantharaman.2007}:
\begin{equation} \label{eq_period_integral_coord}
\bigcup_{i\neq j} I_i\times I_j\ni (x_1, x_2) \longmapsto f_{\gamma_{(x_1, x_2)}} \defgr \int_{t(x_1, x_2)}^{t(x_1, x_2) + \tau(x_1, x_2)} f(t, x_1, x_2) \mathrm{d}t ~,
\end{equation}
where $\tau(x_1, x_2)$ denotes the length of the segment of $\gamma_{(x_1, x_2)}$ that intersects the canonical fundamental domain $\mathcal{F}$ of $\mathbf{X}_\Gamma$ and $\gamma_{x_1, x_2}(t(x_1, x_2))$ is the starting point of this intersecting arc.
By analytic continuation it extends to a neighborhood of $\bigcup_{i\neq j} I_i\times I_j$ in $\mathbb{C}\times \mathbb{C}$ and we denote this extension by $f_{\gamma_{(z_1, z_2)}}$.

\begin{remark}
	Note that the analytic function $f_{\gamma(z_1, z_2)}$ will in general not extend to the entire poly-discs $D_i\times D_j$.
	This does not pose a significant problem as one can simply re-do the proof of Theorem~\ref{thm_fredholm_det} with suitable smaller poly-discs on which $f_{\gamma_{(z_1, z_2)}}$ actually is analytic.
	The theorem then continues to hold for any potential $V$ analytic on an open neighborhood of $\bigcup_{i\neq j} I_i\times I_j$.
\end{remark}

\begin{cor}\label{cor_fredholm_det}
	Let an analytic weight $f\in \mathrm{C}^\omega\left(S\mathbf{X}_\Gamma \right)$ be given. The Fredholm determinant (of the scaling by $z$) of $\mathrm{L}^f_{\lambda, \beta} \defgr \mathrm{L}_{V_{\lambda, \beta}}$ for the parameter-dependent choice of potential
	\begin{equation*}
	V_{\lambda, \beta}(z_1, z_2) \defgr [g_i'(z_2)]^{-\lambda} \cdot \exp\left( -\beta\cdot f_{\gamma_{(z_1, z_2)}} \right)\, , \quad (z_1, z_2)\in D_i\times D_j,\, i\neq j ~,
	\end{equation*}
	coincides with the dynamical determinant of weight $f$ evaluated at $(\lambda, z, \beta)$, i.e. for sufficiently small $\vert z\vert < C(\lambda, \beta)$ we have:
	\begin{equation*}
	d_f(\lambda, z, \beta) = \mathrm{det}\left(\mathrm{id} - z \mathrm{L}_{\lambda, \beta}^f \right) ~.
	\end{equation*}
	In particular the dynamical determinant $d_f(\lambda, z, \beta)$ continues to a holomorphic function in all three variables.
\end{cor}

\begin{proof}
	We begin by observing that given an element $g_w^m$, $w \in \mathcal{W}_n$, with fixed points $(x_-, x_+)$ we have
	\begin{equation*}
	\left( V_{\lambda, \beta} \right)_{g_w^m}(x_-, x_+) = \left( V_{\lambda, \beta} \right)_{g_w}^m (x_-, x_+)
	\end{equation*}
	immediately by definition of $V_{\lambda, \beta}$ and the iterate along a group element.
	If we denote by $\mathcal{W}_n^\mathrm{p}$ the primitive elements of $\mathcal{W}_n$, i.e. words which cannot be written as non-trivial iterations of shorter words, we observe that members of $\mathcal{W}_n^\mathrm{p}$ correspond to primitive geodesics under the correspondence between geodesics and members of $\mathcal{W}_n$.
	Now we calculate using \eqref{eq_transfer_det} for sufficiently small $\vert z\vert$ (with constant dependent on the particular $\lambda$ and $\beta$) that
	\begin{equation*}
	\begin{split}
	\det\left( \mathrm{id} - z\mathrm{L}^f_{\lambda, \beta} \right) &= \exp\left( -\sum_{n = 1}^\infty \sum_{m = 1}^\infty \frac{z^{mn}}{mn} \sum_{w \in\mathcal{W}_n^\mathrm{p}} \frac{\exp\big( -m\lambda T_{\gamma(g_w)} - m\beta \int_{\gamma(g_w)} f \big)}{\vert \det\big( \mathrm{id} - \mathcal{P}^m_{\gamma(g_w)} \big) \vert} \right)\\
	&= \exp\left( -\sum_{m = 1}^\infty \sum_{n = 1}^\infty \frac{z^{mn}}{mn} \sum_{\gamma^\#:\, n(\gamma^\#) = n} n\cdot \frac{\exp\big( -m\lambda T_{\gamma^\#} - m\beta \int_{\gamma^\#} f \big)}{\vert \det\big( \mathrm{id} - \mathcal{P}^m_{\gamma^\#} \big) \vert} \right)\\
	&= \exp\left( -\sum_{m = 1}^\infty \sum_{\gamma^\#} \frac{z^{m\cdot n(\gamma^\#)}}{m} \frac{\exp\big( -m\lambda T_{\gamma^\#} - m\beta \int_{\gamma^\#} f \big)}{\vert \det\big( \mathrm{id} - \mathcal{P}^m_{\gamma^\#} \big) \vert} \right)\\
	&= d_f(\lambda, z, \beta) ~,
	\end{split}
	\end{equation*}
	where in the first equality we used the fact that the contracting and expanding eigenvalues of the linearized Poincar\'{e} map $\mathcal{P}_\gamma$ of the geodesic flow on $\mathbf{X}_\Gamma$ are $\exp(\pm T_\gamma)$.
	By analyticity of Fredholm determinants we obtain an analytic continuation of $d_f$ in the $z$-variable and for fixed $(\lambda, \beta)$ to the complex plane $\mathbb{C}$.
	
	Lastly, we discuss regularity of $d_f$ in its three variables.
	Analyticity with respect to $z$ is standard in the theory of Fredholm determinants.
	For completeness we sketch the proof in Appendix~\ref{app_fredholm_det}.
	Analyticity in $\lambda$ and $\beta$ can also be readily deduced from standard arguments in the theory of Fredholm determinants.
	For an explicit and self-contained proof we refer the reader to Corollary~\ref{coro_cycle_exp} which is independent of the calculations in the remainder of this section.
\end{proof}

\subsection{Dynamical Determinants and Weighted Zeta Functions}\label{dyn_det.2}

We are finally in the position to prove the connection between the weighted zeta function $Z_f$ and the dynamical (Fredholm) determinant $d_f$:
\begin{cor}\label{coro_log_deriv}
	Given an analytic weight function $f\in \mathrm{C}^\omega\left(S\mathbf{X}_\gamma \right)$ the weighted zeta function at $\lambda\in\mathbb{C}$ coincides with the logarithmic derivative of the dynamical determinant w.r.t. $\beta$ and evaluated at $(z, \beta) = (1, 0)$:
	\begin{equation}
	Z_f(\lambda) = \frac{\partial_\beta d_f(\lambda, 1, 0)}{d_f(\lambda, 1, 0)} ~.
	\end{equation}
\end{cor}

\begin{proof}
	If we assume $\mathrm{Re}(\lambda) > 1$ and $\vert\beta\vert$ sufficiently small then plugging the potential defined in Corollary~\ref{cor_fredholm_det} into \eqref{eq_transfer_det} yields an absolutely convergent expression for $d_f$ at $z = 1$ and by an application of Corollary \ref{cor_fredholm_det} we may calculate:
	\begin{equation*}
	\begin{split}
	\frac{\partial_\beta d_f(\lambda, 1, 0)}{d_f(\lambda, 1, 0)} &= \partial_\beta \log\left( d_f(\lambda, 1, \beta) \right)\big|_{\beta = 0}\\
	&= \partial_\beta \left( -\sum_{m = 1}^\infty \sum_{\gamma^\#} \frac{1}{m} \frac{\exp\big( -m\lambda T_{\gamma^\#} - m\beta \int_{\gamma^\#} f \big)}{\vert \det\big( \mathrm{id} - \mathcal{P}^m_{\gamma^\#} \big) \vert} \right) \bigg|_{\beta = 0}\\
	&= \sum_{m = 1}^\infty \sum_{\gamma^\#} \frac{\exp\big( -m\lambda T_{\gamma^\#}\big)}{\vert \det\big( \mathrm{id} - \mathcal{P}^m_{\gamma^\#} \big) \vert} \int_{\gamma^\#} f\\
	&= \sum_{\gamma} \frac{\exp\big( -\lambda T_{\gamma}\big)}{\vert \det\big( \mathrm{id} - \mathcal{P}_{\gamma} \big) \vert} \int_{\gamma^\#} f ~.
	\end{split}
	\end{equation*}
	But then $Z_f$ must coincide with the logarithmic derivative of $d_f$ for all $\lambda\in \mathbb{C}$ by uniqueness of meromorphic continuation.
\end{proof}

Note that (the proof of) Corollary \ref{coro_log_deriv} can also be interpreted as given an alternative argument for meromorphic continuation of weighted zeta functions in the special case of Schottky surfaces and analytic weights:
The proof shows equality between the logarithmic derivative and the defining formula for weighted zeta functions in the halfplane $\mathrm{Re}(\lambda) > 0$ where the latter converges uniformly on compact subsets.
But now the Fredholm determinant defines an analytic function in $\lambda\in\mathbb{C}$ making its logarithmic derivative meromorphic.


\section{Cycle Expansion of Dynamical Determinants}\label{cycle_exp}

Up to this point we actually only ever dealt with the $\lambda$- and $\beta$-variables of our dynamical determinant $d_f$.
The former coincides with the parameter of the weighted zeta function $Z_f$ while the latter was used in the central logarithmic derivative argument in Section~\ref{dyn_det.1}.
This section will now exploit the remaining $z$-variable introduced in $d_f$ to derive formulae for $d_f$ which provide convergent expressions everywhere and are much better suited for actual computation than \eqref{def_df}.
This is done by considering the Taylor expansion around $z = 0$ before plugging in $z = 1$.
In the physics literature this procedure is known under the name \emph{cycle expansion} and has previously been used to great effect in both the physical \cite{Eckhardt.1989} as well as the mathematical \cite{Jenkinson.2002,Borthwick.2014} communities.

\begin{coro}\label{coro_cycle_exp}
	Given an analytic weight $f$ the dynamical determinant can be written as an absolutely convergent Taylor series
	\begin{equation*}
	d_f(\lambda, z, \beta) = 1 + \sum_{n = 1}^\infty d_n(\lambda, \beta) z^n ~,
	\end{equation*}
	where the coefficients are holomorphic in $(\lambda, \beta)$ and explicitly given by the following recursive formula:
	\begin{equation*}
	\begin{split}
	d_n(\lambda, \beta) &= \sum_{k = 1}^n \frac{k}{n} d_{n - k}(\lambda, \beta) a_k(\lambda, \beta) ~,\\
	d_0(\lambda, \beta) &\equiv 1, \quad a_k(\lambda, \beta) = -\sum_{w \in \mathcal{W}_k} \frac{1}{k} \frac{\exp\big(-(\lambda - 1)\ell(g_w) - \beta \int_{\gamma(g_w)} f \big)}{(\mathrm{e}^{\ell(g_w)} - 1)^2} ~.
	\end{split}
	\end{equation*}
	Furthermore they satisfy the following super-exponential bounds for some positive constants $C, c_1, c_2, c_3 > 0$:
	\begin{equation*}
	\vert d_n(\lambda, \beta) \vert \leq C\cdot \mathrm{e}^{-c_1 n^2 + c_2 n \vert \lambda\vert - c_3 n \vert \beta\vert} ~.
	\end{equation*}
\end{coro}

\begin{proof}
	First we derive the given recursion by starting from the following expression obtained by plugging the potential defined in Corollary \ref{cor_fredholm_det} into \eqref{eq_transfer_det} and valid for small $\vert z\vert$:
	\begin{equation*}
	d_f(\lambda, z, \beta) = \exp\left( \sum_{n = 1}^\infty a_n(\lambda, \beta) z^n \right), \quad a_n(\lambda, \beta) \defgr -\frac{1}{n} \sum_{w \in \mathcal{W}_n} \frac{\exp\big(-\lambda \ell(g_w) - \beta \int_{\gamma(g_w)} f \big)}{(\mathrm{e}^{\ell(g_w)} - 1) (1 - \mathrm{e}^{-\ell(g_w)})} ~.
	\end{equation*}
	One now arrives at the claimed recursion by expanding the exponential function in terms of Cauchy products and collecting common powers of $z$.
	As $d_f(\lambda, z, \beta)$ is analytic in $z$ on the whole complex plane the resulting power series must converge absolutely for any $z\in\mathbb{C}$.
	
	As doing this expansion explicitly is a common combinatorial problem there exists a well-known solution given in terms of so-called (complete) Bell polynomials $B_n$ (see e.g. \cite[Section~3.3]{Comtet.1974} or \cite[Section~16.1]{Borthwick.2016}).
	These are defined by
	\begin{equation*}
	\exp\left( \sum_{n = 1}^\infty \frac{a_n}{n!} z^n \right) = \sum_{n = 0}^\infty \frac{B_n(a_1, \hdots, a_n)}{n!} z^n
	\end{equation*}
	and can be shown to satisfy the recursion relation
	\begin{equation*}
	B_{n + 1}(a_1, \hdots, a_{n+1}) = \sum_{i = 0}^n \binom{n}{i} B_{n - i}(a_1, \hdots, a_{n - i}) a_{i + 1}, \quad B_0 = 1 ~.
	\end{equation*}
	For the straight-forward proof of this recursion we refer to the literature mentioned above. Using this relation it is elementary to calculate
	\begin{equation*}
	\begin{split}
	&\quad d_n(\lambda, \beta)\\
	&= \frac{1}{n!} B_n(1! a_1(\lambda, \beta), \hdots, n! a_n(\lambda, \beta))\\
	&= \frac{1}{n!} \sum_{i = 0}^{n - 1} (i+1)! \binom{n - 1}{i} B_{n - 1 - i}(1! a_1(\lambda, \beta), \hdots, (n - 1 - i)! a_{n - 1 - i}(\lambda, \beta)) a_{i+1}(\lambda, \beta)\\
	&= \sum_{i = 0}^{n-1} \frac{i + 1}{n} d_{n - 1 - i}(\lambda, \beta) a_{i+1}(\lambda, \beta) = \sum_{k = 1}^{n} \frac{k}{n} d_{n - k}(\lambda, \beta) a_k(\lambda, \beta) ~,
	\end{split}
	\end{equation*}
	and $d_0(\lambda,  \beta) \equiv B_0 = 1$, proving the claimed formula.

	To prove the estimates on the coefficients $d_n$ we proceed in a similar fashion as in Corollary \ref{cor_fredholm_det} but refine the arguments made there slightly, c.f. \cite[Lemma~16.1]{Borthwick.2016}:
	The Fredholm determinant can alternative be expressed as (see Appendix~\ref{app_fredholm_det})
	\begin{equation}
	\det\left( \mathrm{id} - \mathrm{L}^f_{\lambda, \beta} \right) = \sum_{n = 0}^\infty (-1)^n \mathrm{Tr}\big( \bigwedge^n \mathrm{L}^f_{\lambda, \beta} \big) ~,
	\end{equation}
	where $\bigwedge^n \mathrm{L}^f_{\lambda, \beta}$ denotes the $n$-fold exterior power of $\mathrm{L}^f_{\lambda, \beta}$ acting on the $n$-fold exterior power of its original domain.
	It immediately follows that we can re-write the coefficients $d_n(\lambda, \beta)$ as traces
	\begin{equation*}
	d_n(\lambda, \beta) = (-1)^n \mathrm{Tr}\big( \bigwedge^n \mathrm{L}^f_{\lambda, \beta} \big) ~.
	\end{equation*}
	These traces can be estimates by the same technique employed in the proof of Theorem~\ref{thm_fredholm_det}, but this time we explicitly keep the exponential dependency on the parameters $\lambda$ and $\beta$ instead of bounding them on some compact subset:
	\begin{equation*}
	\begin{split}
	\big| (-1)^n \mathrm{Tr}\big( \bigwedge^n \mathrm{L}^f_{\lambda, \beta} \big) \big|
	&\leq \sum_{j_1 < \hdots < j_n} \mu_{j_1}\left( \mathrm{L}_{\lambda, \beta}^f \right) \cdots \mu_{j_n}\left( \mathrm{L}_{\lambda, \beta}^f \right)\\
	&\leq C^n \mathrm{e}^{n c_2 \vert\lambda\vert - n c_3 \vert\beta\vert} \sum_{j_1 < \hdots < j_n} \mathrm{e}^{-C \left( j_1 + \hdots + j_n \right)}\\
	&\leq C^n \mathrm{e}^{n c_2 \vert\lambda\vert - n c_3 \vert\beta\vert} \mathrm{e}^{-C n^2} ~,
	\end{split}
	\end{equation*}
	where the first inequality combines the standard estimate of the trace norm in terms of singular values with an explicit expression for singular values of tensor powers.
	Absorbing the polynomial factor into the exponential term proves the claim.
\end{proof}

\begin{remark}
	The arguments made in the proof above could be refined further to obtain a result resembling \cite[Lemma~16.1]{Borthwick.2016} even more closely.
	We refrain from going into that much detail here as our concrete numerical calculations will rely on symmetry reduction (see Section~\ref{sym_red}) -- with this reduction the empirical convergence rate is far better than the analytically obtained estimates.
	Thus we do not see a great benefit in optimizing the theoretical bounds.
\end{remark}

\begin{remark}
	From the appearance of the coefficients $a_k(\lambda, \beta)$ alone one immediately notices an invariance property under the action of $\mathbb{Z}$ generated by shifts of words
	\begin{equation*}
	\mathcal{W}_n\ni (i_1, \hdots, i_n) \longmapsto (i_n, i_1, \hdots, i_{n-1}) ~.
	\end{equation*}
	Furthermore, it is straight forward to reduce the sum over $\mathcal{W}_k$ appearing in $a_k(\lambda, \beta)$ to a sum over primitive words which reduces the number of words one has to calculate in practice even further.
	
	We refrain from formalizing these simplifications here because they will be discussed in detail in Section~\ref{sym_red} where they are combined with a reduction w.r.t. additional symmetries of the underlying Schottky surface.
	For numerical experiments one would resort to the symmetry reduced variant anyways.
\end{remark}


\section{The Numerical Algorithm}\label{algo}

We are missing one further ingredient before we can really use a (cutoff version of) the formulae derived in Section~\ref{cycle_exp} for numerics.
This ingredient is a computationally feasible approach for the calculation of the period integrals $f_{\gamma^\#}$ which figure prominently in the dynamical determinant $d_f$.

We begin with a short sketch of what is forthcoming in this section.
To this end let the reader be reminded that being able to calculate the weighted zeta function $Z_f$ via the dynamical determinant $d_f$ was actually just a means for calculating invariant Ruelle distributions $\mathcal{T}_{\lambda_0}$ via the formula
\begin{equation} \label{eq_ruelle_residue}
\mathcal{T}_{\lambda_0}(f) = \underset{\lambda = \lambda_0}{\mathrm{Res}} \left[ Z_f(\lambda) \right] ~.
\end{equation}
Now visualizing the \emph{distribution} $\mathcal{T}_{\lambda_0}$ amounts to visualizing a suitable smooth approximation.
The latter should come with a parameter that controls the accuracy of the approximation.
We take the straightforward approach of choosing Gaussian test functions $f_\sigma$ of width $\sigma > 0$ and considering (roughly) the distributional convolution\footnote{
	What we are actually using are smooth approximations inspired by but not identical to convolution because only $S\mathbb{H} \cong \mathrm{PSL}(2, \mathbb{R})$ and $S\mathbb{H} \cong \mathrm{PSL}(2, \mathbb{R}) \slash \mathrm{PSO}(2)$ carry group structures but the quotients $\mathbf{X}_\Gamma \cong \Gamma\setminus \mathrm{PSL}(2, \mathbb{R}) \slash \mathrm{PSO}(2)$ and $S\mathbf{X}_\Gamma \cong \Gamma \setminus \mathrm{PSL}(2, \mathbb{R})$ do not.
}
\begin{equation*}
\mathrm{C}^\infty(S\mathbf{X}_\Gamma)\ni \mathrm{t}_\sigma \defgr \mathcal{T}_{\lambda_0} \ast f_\sigma \underset{\sigma\rightarrow 0}{\longrightarrow} \mathcal{T}_{\lambda_0} \in \mathcal{E}'(S\mathbf{X}_\Gamma) ~,
\end{equation*}
that converges to the original distribution $\mathcal{T}_{\lambda_0}$ in the limit $\sigma\rightarrow 0$.
Details on these approximating families will be provided in the upcoming sections.

Having restricted the class of weights to the family $f_\sigma$ we are still faced with the problem that $\mathrm{t}_\sigma$ is a function on the three-dimensional space $S\mathbf{X}_\Gamma$ and therefore still difficult to visualize.
We remedy this situation by considering not $\mathcal{T}_{\lambda_0}$ itself but two reductions obtained by either pushforward (projection) to the base manifold $\mathbf{X}_\Gamma$ or pullback (restriction) to certain hypersurfaces $\Sigma\subseteq S\mathbf{X}_\Gamma$:
\begin{equation*}
\begin{split}
\mathrm{C}^\infty(\mathbf{X}_\Gamma)\ni \pi_*\big( \mathrm{t}_\sigma \big) \longrightarrow \pi_* \mathcal{T}_{\lambda_0}, &\qquad \pi: S\mathbf{X}_\Gamma \longrightarrow \mathbf{X}_\Gamma ~,\\
\mathrm{C}^\infty(\Sigma)\ni \iota^*_\Sigma\big( \mathrm{t}_\sigma \big) \longrightarrow \iota^*_\Sigma \mathcal{T}_{\lambda_0}, &\qquad \iota_\Sigma: \Sigma \,\hookrightarrow S\mathbf{X}_\Gamma ~.
\end{split}
\end{equation*}

In the following two sections we will give precise operational prescriptions for both approaches.
This encompasses suitable choices of parametrization for the respective domains, concrete test functions $f_\sigma$ adapted to the specific application, and finally a numerically tractable approach to calculate the associated period integrals $\int_\gamma f_\sigma$.
The last ingredient missing for an actual algorithm is a means of calculating the residue in \eqref{eq_ruelle_residue}.
Different approaches to this problem are discussed in Remark~\ref{sym_red}.

\subsection{Ruelle Distributions on the Base Manifold}\label{algo.1}

The \emph{pushforward} of the distribution $\mathcal{T}_{\lambda_0}\in \mathcal{E}'(S \mathbf{X}_\Gamma)$ under the projection
$\pi: S\mathbf{X}_\Gamma\rightarrow \mathbf{X}_\Gamma$ is defined as a distribution on $\mathbf{X}_\Gamma$ by the following formula:
\begin{equation*}
\left\langle\pi_*\big( \mathcal{T}_{\lambda_0} \big), f \right\rangle \defgr \left\langle \mathcal{T}_{\lambda_0}, f\circ \pi \right\rangle, \qquad f\in \mathrm{C}^\infty(\mathbf{X}_\Gamma) ~.
\end{equation*}
Intuitively this distribution encodes the dependency of $\mathcal{T}_{\lambda_0}$ on the base point and averages over the directions in $S\mathbf{X}_\Gamma$.
It should therefore relate to those features of resonant states that are independent of direction.

As mentioned above we will use a variation of convolution as a means to obtain quantities that can actually be plotted.
As test functions we take a family of (hyperbolic) Gaussians constructed as follows:
Considering the transitive group action of $G$ on $\mathbb{H}$ one calculates the stabilizer of $\mathrm{i}\in\mathbb{H}$ to be the subgroup of rotations $\mathrm{PSO}(2)$.
This yields a diffeomorphism $G\slash K \cong \mathbb{H}$, $g\mapsto g\cdot \mathrm{i}$, and we may define a family of Gaussians $f_\sigma$, $\sigma > 0$, on the quotient $G\slash K$ by the formula
\begin{equation}\label{eq_hyperbolic_gauss}
	\widetilde{f}_\sigma(gK) = \exp\bigg( -\frac{\mathrm{d}_{G\slash K}(gK, eK)}{\sigma^2}^2 \bigg) ~,
\end{equation}
where the $G$-invariant metric $\mathrm{d}_{G\slash K}$ is defined in terms of the hyperbolic distance $\mathrm{d}_\mathbb{H}$, derived from the metric $\mathrm{g}_\mathbb{H}$, by $\mathrm{d}_{G\slash K}(gK, g'K) \defgr \mathrm{d}_\mathbb{H}(g\cdot\mathrm{i}, g'\cdot \mathrm{i})$.
We denoted the identity element by $e\in G$.

If $\mathcal{T}_{\lambda_0}$ were a distribution on $G\slash K$ we could define the operation of convolution in a straight forward manner, well-known from harmonic analysis, by sampling our distribution against the family of shifted Gaussians $\widetilde{f}_{\sigma, gK}$ defined as $\widetilde{f}_{\sigma, gK}(g'K) \defgr \widetilde{f}_{\sigma}(g'^{-1}gK)$.
It should be obvious that this is the correct notion of convolution by duality with respect to ordinary convolution of functions.

Even though the bi-quotient $\Gamma\setminus G\slash K$ no longer carries a group structure we can use the following family of smooth functions approximating $\pi_* \mathcal{T}_{\lambda_0}$ as a natural analogue of genuine convolution:\footnote{
	For notational convenience and because context removes any ambiguity we do not differentiate between the projections $S\mathbb{H}\rightarrow \mathbb{H}$ and $G\rightarrow G\slash K$.
	A similar comment applies to the projections $S\mathbf{X}_\Gamma\rightarrow \mathbf{X}_\Gamma$ and $\Gamma\setminus G \rightarrow \Gamma\setminus G\slash K$.
}

\begin{equation}\label{eq_invariant_hyperbolic_gauss}
\begin{split}
\mathrm{t}^{G\slash K}_{\lambda_0, \sigma} (gK) &\defgr \left\langle \mathcal{T}_{\lambda_0}, f_{\sigma, gK} \circ \pi \right\rangle ~, \qquad\qquad\qquad\qquad\quad~\, gK\in G\slash K \\
f_{\sigma, gK}(g'K) &\defgr \frac{1}{\mathcal{N}_\sigma} \sum_{h\in \Gamma} \exp\bigg( -\frac{\mathrm{d}_{G\slash K}(hgK, g'K)^2}{\sigma^2} \bigg) ~, \quad g'K \in G\slash K ~.
\end{split}
\end{equation}
The normalization factor $\mathcal{N}_\sigma$ certifies the condition $\int_{\mathbf{X}_\Gamma} f_\sigma = 1$, and the sum over $\Gamma$ converges absolutely for any $gK\in G\slash K$ by \cite[Eq.~(2.22)]{Borthwick.2016}.
Note that $f_{\sigma, gK}(g'K)$ is $\Gamma$-invariant both in $gK$ and $g'K$.

We can make the previous paragraphs even more specific as follows: calculating our approximation $\mathrm{t}^{G\slash K}_{\lambda_0} (gK)$ basically amounts to evaluating (the residues of) the dynamical determinant $d_{f_{\sigma, gK} \circ \pi}$.
This in turn boils down to an implementation of the following integrals over closed geodesics $\gamma \subseteq S\mathbf{X}_\Gamma$:
\begin{equation} \label{eq_integral_fund_domain}
\int_{\pi\circ \gamma} f_{\sigma, gK} = \int_0^{T_\gamma} f_{\sigma, gK}(g_0^{-1} a_t K) \mathrm{d}t ~,
\end{equation}
where $\{a_t\}_{t\in\mathbb{R}}\subseteq G$ is the standard geodesic through $\mathrm{i}\in \mathbb{H}$, i.e. $a_t\cdot \mathrm{i} = \mathrm{i} \mathrm{e}^t$, and the symmetry $g_0\in G$ was chosen to satisfy $g_0^{-1}(0) = x_-$ as well as $g_0^{-1}(\infty) = x_+$ for the endpoints at infinity $(x_-, x_+)$ of $\gamma$.

At this point we make a couple of simplifying assumptions to reach a computationally feasible expression.
The error of our upcoming approximations will depend on $\sigma$ in such a way that the difference between our final expression and \eqref{eq_integral_fund_domain} converges to $0$ in the limit $\sigma\rightarrow 0$.
This justifies using the former over the latter for numerical purposes.
A graphical illustration of the upcoming discussion can be found in Figure~\ref{fig3}.

\begin{figure}[h]
	\includegraphics[scale=0.55]{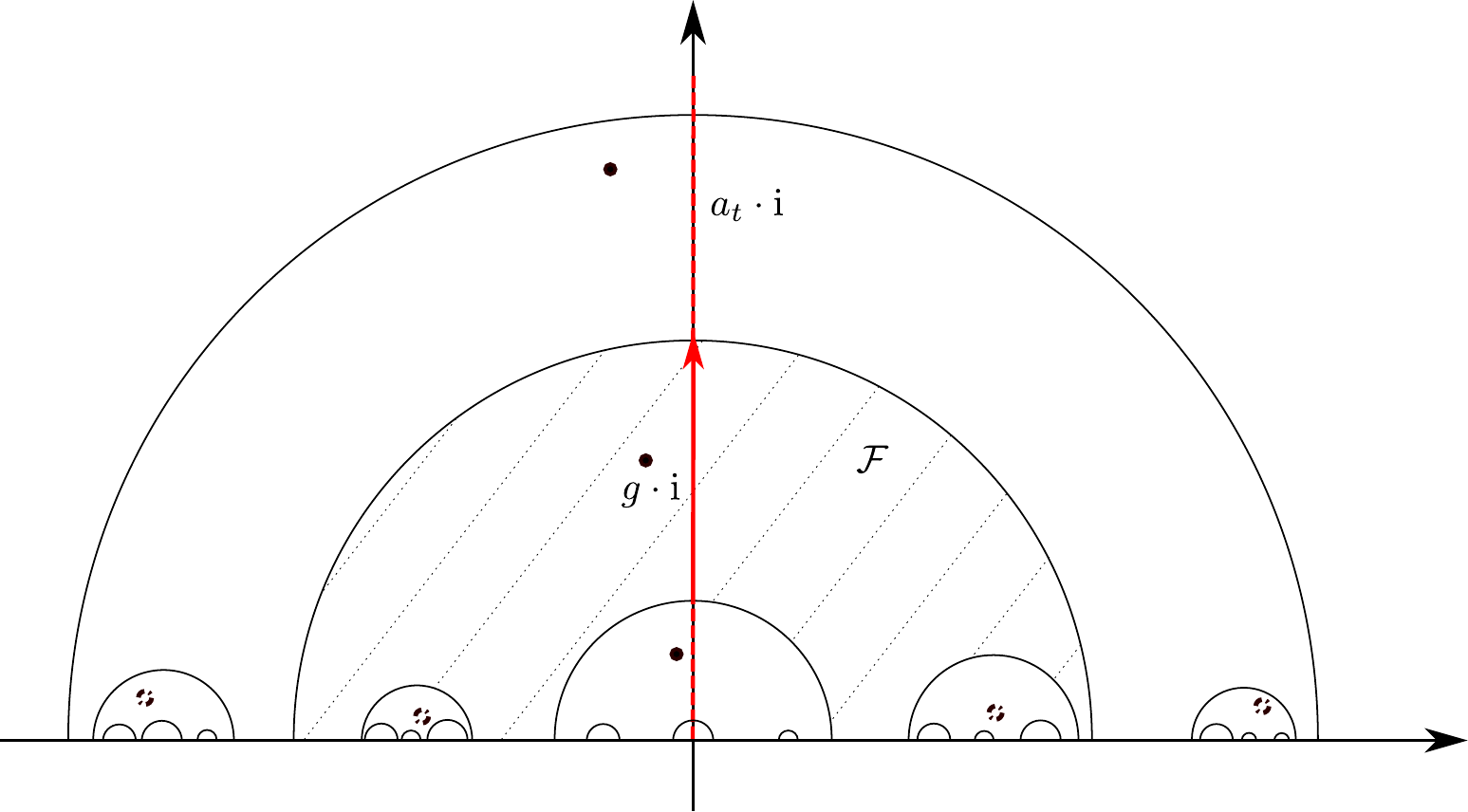}
	\caption{Illustration of the fundamental circles (in rank $r = 2$) and some of their translates as used in the approximation of Gaussian period integrals on the fundamental domain of a Schottky surface.
	The dashed points correspond to translates of $g\cdot\mathrm{i}$ along elements of $\Gamma_0$, with vanishing contribution in the limit $\sigma\rightarrow 0$, whereas the solid points correspond to translates along $\widetilde{\Gamma}$.}
	\label{fig3}
\end{figure}

By the $\Gamma$-invariance of $f_{\sigma, gK}$ we may restrict attention to centers $g$ inside the canonical fundamental domain, i.e. $g\cdot \mathrm{i}\in \mathcal{F}$.
It is then practical to consider the following decomposition of the group $\Gamma$ into two disjoint subsets:
\begin{equation*}
\begin{split}
\Gamma = \Gamma_0 \sqcup \widetilde{\Gamma} ~, \qquad
\Gamma_0 &\defgr \left\{ h\in\Gamma \,|\, h\mathcal{F}\cap (g_0^{-1} a_t\cdot \mathrm{i}) = \emptyset\, \forall t\in\mathbb{R} \right\} ~,\\
\widetilde{\Gamma} &\defgr \left\{ h\in\Gamma \,|\, h\mathcal{F}\cap (g_0^{-1} a_t\cdot \mathrm{i}) \neq \emptyset\, \forall t\in\mathbb{R} \right\} ~.
\end{split}
\end{equation*}
The sum over $\Gamma$ appearing in \eqref{eq_invariant_hyperbolic_gauss} splits accordingly.
We begin our analysis with the sum over $\Gamma_0$ by noting that there exists a constant $c > 0$ such that every $h\in \Gamma_0$ satisfies
\begin{equation*}
\mathrm{d}_{G\slash K}(g_0 h g K, a_t K) \geq c > 0 \quad \forall t\in\mathbb{R} ~.
\end{equation*}
Combining this estimate with the more general and well-known exponential growth bound \cite[Eq.~(2.22)]{Borthwick.2016} we obtain
\begin{equation*}
\# \left\{h\in \Gamma \,\bigg|\, \inf_{t\in [0, T_\gamma]} \mathrm{d}_{G\slash K}(g_0 h g K, a_t K) \leq s \right\} = \mathcal{O}(\mathrm{e}^s) ~,
\end{equation*}
which lets us conclude that
\begin{equation*}
\begin{split}
&\quad ~\mathcal{N}_\sigma^{-1} \sum_{h\in \Gamma_0} \int_0^{T_ \gamma} \exp\bigg( -\frac{\mathrm{d}_{G\slash K}(g_0 h g K, a_t K)^2}{\sigma^2} \bigg) ~\mathrm{d}t \\
&\leq C \mathcal{N}_\sigma^{-1} \sum_{n\in\mathbb{N}}
\sum_{\substack{h\in\Gamma_0:\, \forall t\in [0, T_\gamma]\\ \mathrm{d}_{G\slash K}(g_0 h g K, a_t K) \\ \leq c + 1}}
\mathrm{e}^{-\frac{c^2 + n^2}{\sigma^2}} ~,
\end{split}
\end{equation*}
i.e. the sum over $\Gamma_0$ is of the order $\mathcal{O}(\mathcal{N}_\sigma^{-1} \mathrm{e}^{-\frac{C}{\sigma^2}})$ as $\sigma \rightarrow 0$.
Our discussion of the normalization factor below will reveal $\mathcal{N}_\sigma^{-1} = \mathcal{O}(\sigma^{-2})$ so this does indeed vanish in the limit $\sigma\rightarrow 0$.

The significant contribution to the total period integral is given by the second sum over $\widetilde{\Gamma}$.
We treat this term by first fixing a group element $g_w$, $w = (i_1, ..., i_n)\in \mathcal{W}_n$, representing $\gamma$ and observing that the action of $g_w$ restricted to $\gamma$ is simply translation.
We may therefore absorb the cyclic subgroup generated by $g_w$ into the integral and re-write
\begin{equation*}
\begin{split}
&\quad \sum_{h\in \widetilde{\Gamma}} \int_0^{T_ \gamma} \exp\bigg( -\frac{\mathrm{d}_{G\slash K}(g_0 h g K, a_t K)^2}{\sigma^2} \bigg) ~\mathrm{d}t \\
&= \sum_{j = 1}^n \int_{-\infty}^\infty \exp\bigg( -\frac{\mathrm{d}_{G\slash K}(g_0 g_{i_j}\cdots g_{i_1} gK, a_t K)^2}{\sigma^2} \bigg) \mathrm{d}t ~.
\end{split}
\end{equation*}

As we approach our final expression we change perspective from the quotient $G\slash K$ back to the upper half plane.
To simplify the resulting Gaussian integral we make the following approximation
\begin{equation*}
\mathrm{d}_\mathbb{H}(x + \mathrm{i}y, \mathrm{i}s)^2 = \mathrm{d}_\mathbb{H}(x + \mathrm{i}y, \mathrm{i}y)^2 + \mathrm{d}_\mathbb{H}(\mathrm{i}y, \mathrm{i}s)^2 + \mathcal{O}(x^4) ~,
\end{equation*}
and substituting this into \eqref{eq_integral_fund_domain} we arrive at the following approximate expression for the period integrals:
\begin{equation*}
\int_{\pi\circ \gamma} f_{\sigma, gK} \approx \mathcal{N}_\sigma^{-1} \sum_{j = 1}^n \mathrm{e}^{-\frac{x_j^2}{\sigma^2 y_j^2}} \int_{-\infty}^\infty \mathrm{e}^{-\frac{(y_j - t)^2}{\sigma^2}} ~\mathrm{d}t = \mathcal{N}_\sigma^{-1} \sqrt{\pi} \sigma \sum_{j = 1}^n \mathrm{e}^{-\frac{x_j^2}{\sigma^2 y_j^2}} ~,
\end{equation*}
where we used the definition $x_j + \mathrm{i}y_j \defgr g_0 g_{i_j} \cdots g_{i_1} g \cdot \mathrm{i}$ for the complex coordinates of the points $g_0 g_{i_j} \cdots g_{i_1} gK \in G\slash K$.

Finally, we use the approximation $\cosh^{-1}(1 + z)^2 = 2z + \mathcal{O}(z^3)$ to simplify hyperbolic distance.
This lets us calculate the normalization factor $\mathcal{N}_\sigma$ in a straight forward manner:
\begin{equation*}
\mathcal{N}_\sigma \approx \int_0^\infty \int_{-\infty}^\infty \exp\bigg( -\frac{x^2 + (y - 1)^2}{\sigma^2 y} \bigg) \frac{\mathrm{d}x \mathrm{d}y}{y^2} = \sqrt{\pi} \sigma \int_0^\infty \exp\bigg( -\frac{(y - 1)^2}{\sigma^2 y} \bigg) \frac{\mathrm{d}y}{y^{3/2}} = \pi \sigma^2 ~,
\end{equation*}
thus concluding our discussion on how to approximate $\pi_* \mathcal{T}_{\lambda_0}$ for practical implementation purposes with the final expression
\begin{equation} \label{eq_calcInt_fundamental}
\int_{\pi\circ \gamma} f_{\sigma, gK} \approx \frac{1}{\sqrt{\pi}\sigma} \sum_{j = 1}^n \mathrm{e}^{-\frac{x_j^2}{\sigma^2 y_j^2}} ~.
\end{equation}

For convenience we summarize the steps necessary for the calculation of $Z_{f_{\sigma, gK}\circ \pi}$ in the pseudo-code of the following Snippet~\ref{alg1}.
Taking the residue of this weighted zeta then yields the approximation $\mathrm{t}^{G\slash K}_{\lambda_0, \sigma}$.
For remarks on how to calculate residues in practice we refer to Section~\ref{sym_red.2}.

\begin{algorithm}
	\caption{%
		Pseudo-code for the calculation of the weighted zeta function $Z_{f_{\sigma, gK}\circ \pi}$ used to approximate the distribution $\pi_* \mathcal{T}_{\lambda_0}$.
	}\label{alg1}
	\KwIn{width $\sigma > 0$, center $gK\in G/K$, resonance $\lambda_0\in\mathbb{C}$, cut-off $N\in\mathbb{N}$}
	
	$d_f \gets 1$ \;
	$\partial_\beta d_f \gets 0$\;
	$d[0] \gets 1$\;
	$\partial_\beta d[0] \gets 0$\;
	
	\Comment{calculate Bell recursion, for initial terms $a[k]$ see below}
	\For{$n := 1$ to $N$}{
		$d[n] \gets 0$\;
		\For{$k := 1$ to $n$}{
			$d[n] \gets d[n] + \frac{k}{n} d[n-k] a[k]$\;
			$\partial_\beta d[n] \gets \partial_\beta d[n] + \partial_\beta d[n-k] a[k] + d[n-k] \partial_\beta a[k]$\;
		}
		$d_f \gets d_f + d[n]$\;
		$\partial_\beta d_f \gets \partial_\beta d_f + \partial_\beta d[n]$;
	}
	$\mathrm{return}~ \frac{\partial_\beta d_f}{d_f}$\;
	
	\Comment{calculate initial terms $a[k]$ in Bell recursion}
	$a[k] \gets 0$\;
	$\partial a[k] \gets 0$\;
	\For{$w\in\mathcal{W}_k$}{
		$a[k] \gets a[k] - \frac{1}{k} \frac{\exp(-(\lambda_0 - 1) \ell(g_w))}{(\mathrm{e}^{\ell(g_w)} - 1)^2}$\;
		$\partial_\beta a[k] \gets \partial_\beta a[k] + \frac{\mathrm{calcInt}(w)}{k} \frac{\exp(-(\lambda_0 - 1) \ell(g_w))}{(\mathrm{e}^{\ell(g_w)} - 1)^2}$
		\Comment*[r]{$\mathrm{calcInt}$ implements \eqref{eq_calcInt_fundamental}}
	}
	
	\KwResult{approximation of $Z_{f_{\sigma, gK}\circ \pi}(\lambda_0)$}
\end{algorithm}

\subsection{Restricted Ruelle Distributions}\label{algo.2}

As announced in the introduction we consider the pullback $\iota_\Sigma^* \mathcal{T}_{\lambda_0}$ along the inclusion $\iota_\Sigma: \Sigma\hookrightarrow S\mathbf{X}_\Gamma$ of a hypersurface $\Sigma$ as a second approach of reducing the complexity of $\mathcal{T}_{\lambda_0}$ from the full three-dimensional space $S\mathbf{X}_\Gamma$ to a two-dimensional subspace.
This distributional operation is well-defined by a classical theorem of Hörmander~\cite[Thm.~8.2.4]{Hormander.2003} as long as $\Sigma$ is transversal to the geodesic flow \cite[Lemma~2.3]{Schuette.2021b}.
If this is satisfied we call $\Sigma$ a \emph{Poincar\'{e} section} and the first part of the upcoming discussion applies to any such submanifold.
Only once we require an implementation-level prescription for the calculation of period integrals will we introduce a specific choice of $\Sigma$.
This $\Sigma$ will be used throughout the numerical examples but could very well be replaced by a number of alternative choices (see Remark \ref{remark1}).

The operational meaning of the pullback is not provided by an explicit formula as was the case for the pushforward --
instead it can only be defined as a limit in $\mathcal{D}'(\Sigma)$. Concretely if any family of smooth functions converges to $\mathcal{T}_{\lambda_0}$ in the space $\mathcal{D}'_{E_u^*\oplus E_s^*}(S\mathbf{X}_\Gamma)$ then their restrictions to
$\Sigma$ converge to the restriction of $\mathcal{T}_{\lambda_0}$. The latter space consists of distributions with wavefront set
contained in $E_u^*\oplus E_s^*$:\footnote{
	The direct sum $E_u^*\oplus E_s^*\subseteq T^*(S\mathbf{X}_\Gamma)$ is closely related to the hyperbolicity of the geodesic flow on $S\mathbf{X}_\Gamma$.
	The technical details can be found in \cite{Schuette.2021b}.
}
\begin{equation*}
\lim_{\sigma\rightarrow 0} \mathrm{t}_{\lambda_0, \sigma}^\Sigma = \mathcal{T}_{\lambda_0} ~\text{in}~ \mathcal{D}'(S\mathbf{X}_\Gamma)
~\quad\Longrightarrow\quad
\iota_\Sigma^* \mathcal{T}_{\lambda_0} = \lim_{\sigma\rightarrow 0} \mathrm{t}_{\lambda_0, \sigma}^\Sigma \big|_\Sigma ~\text{in}~ \mathcal{D}'(\Sigma) ~.
\end{equation*}

We construct the approximations $\mathrm{t}^\Sigma_{\lambda_0, \sigma}$ by a similar approach as the one that resulted in the functions $\mathrm{t}^{G / K}_{\lambda_0, \sigma}$ in the previous paragraph.
Here we work on the whole group $G \cong S\mathbb{H}$ instead of $G / K$, though:
\begin{equation}
\begin{split}
\mathrm{t}^\Sigma_{\lambda_0, \sigma}(g) &\defgr \left\langle \mathcal{T}_{\lambda_0}, f_{\sigma, g} \right\rangle ~, \qquad\qquad\qquad\qquad\qquad ~ g\in G\\
f_{\sigma, g}(g') &\defgr \frac{1}{\mathcal{M}_\sigma} \sum_{h\in \Gamma} \exp\bigg( -\frac{\mathrm{d}_G(g, hg')^2}{\sigma^2} \bigg) ~, \qquad g'\in G ~, 
\end{split}
\end{equation}
where $\mathcal{M}_\sigma$ again denotes an appropriate normalization factor and $\mathrm{d}_G$ is a smooth distance function on $G\times G$ to be specified later.
By definition $f_{\sigma, g}(g')$ is left $\Gamma$-invariant in $g'$ such that $\mathrm{t}^\Sigma_{\lambda_0, \sigma}$ is well-defined on $S\mathbf{X}_\Gamma$.

In terms of concrete calculation we need a way to evaluate integrals of $f_{\sigma, g}$ over closed geodesics for elements $g\in G$ such that $\Gamma g\cdot (\mathrm{i}, \mathrm{i})\in \Sigma$.
To render these quantities as computationally inexpensive as possible we will proceed to specifying a suitable combination of surface $\Sigma$ and adapted distance $\mathrm{d}_G$.

As mentioned and exploited in Section~\ref{dyn_det.1} there exists a particular set of coordinates well adapted to the action of the geodesic flow called \emph{Hopf coordinates} \cite{Thirion.2007,Dang.2021}.
Both $\Sigma$ and $\mathrm{d}_G$ will be described in these coordinates but we start by giving a more structure theoretic definition:
The following map is a $G$-equivariant\footnote{
	We will not go into details about the $G$-action on the first component of the codomain of $\mathbb{H}$, c.f. \cite[Proposition~2.9]{Dang.2021}.
}
diffeomorphism
\begin{equation*}
\begin{split}
\mathcal{H}: G &\longrightarrow \mathbb{R}\times (G\slash P)_\Delta\\
g &\mapsto (\beta_{g_+}(g^{-1}\cdot o, o), g_+, g_-) ~,
\end{split}
\end{equation*}
where $g_+$ and $g_-$ are the fixed points on the boundary of hyperbolic space $G \slash P \defgr \mathrm{PSO}(2) \cong \partial\mathbb{H}$ of the isometry $g$ and $\beta_{g_+}(g^{-1}\cdot o, o)$ denotes the hyperbolic distance between the point of intersection of $\gamma(g)$ with the horocycle\footnote{
	In the upper halfplane model the horocycles centered at a boundary point $\xi\in \partial\mathbb{H}$ are the circles tangent to $\mathbb{R}$ at $\xi$ if $\xi\in\mathbb{R}$ and the lines parallel to $\mathbb{R}$ if $\xi = \infty$.
}
through $o \defgr (\mathrm{i}, \mathrm{i})$ centered at $g_+$ and the point $g^{-1}\cdot o$.
The subscript $\Delta$ indicates removal of the diagonal.

Geometrically, $\mathcal{H}$ describes a correspondence between the group $G\cong S\mathbb{H}$ and the space of parameterized geodesics of hyperbolic space.
We note again that in these coordinates the geodesic flow acts as $(t, g_+, g_-) \mapsto (t + s, g_+, g_-)$.

Our surface $\Sigma\subseteq S\mathbf{X}_\Gamma$ is best described dynamically in terms of the geodesic flow:
On the cover $\pi_\Gamma: S\mathbb{H} \rightarrow S\mathbf{X}_\Gamma$ its intersection $\widetilde{\Sigma} \defgr \pi_\Gamma^{-1}(\Sigma)\cap S\overline{\mathcal{F}}$ with the canonical fundamental domain consists of the points of intersection between the boundary $\partial D_i$ of fundamental circles and geodesics $\gamma$ with points at infinity $\gamma_\pm$ in distinct fundamental intervals $I_i, I_k$ (see Figure~\ref{fig4} for an illustration):
\begin{equation*}
\widetilde{\Sigma} = \bigg\{ (z, v)\in S\mathbb{H} \,\big|\, \exists \,\text{geodesic}\,\gamma\subseteq \mathbb{H}: (\gamma_+, \gamma_-)\in \bigcup_{i \neq j} I_i\times I_j,\, (z, v)\in \gamma \cap S\big( \bigcup_i \partial D_i \big) \bigg\} ~.
\end{equation*}
With respect to Hopf coordinates we can describe $\widetilde{\Sigma}$ by means of a smooth function $t = t(g_+, g_-)$, where $(g_+, g_-)\in \bigcup_{i \neq j} I_i\times I_j\subseteq (G \slash P)_\Delta$, via its graph
\begin{equation*}
\mathcal{H}(\widetilde{\Sigma}) = \bigg\{(t(g_+, g_-), g_+, g_-) \,\big|\, (g_+, g_-)\in \bigcup_{i \neq j} I_i\times I_j \bigg\} ~,
\end{equation*}
i.e. $\widetilde{\Sigma}$ is essentially parameterized by pairs of boundary points in distinct fundamental intervals.
Note that this definition determines a well-defined, unique hypersurface $\Sigma\subseteq S\mathbf{X}_\Gamma$ that satisfies the claimed relation $\widetilde{\Sigma} = \pi_\Gamma^{-1}(\Sigma)\cap S\overline{\mathcal{F}}$.

\begin{figure}[h]
	\includegraphics[scale=0.42]{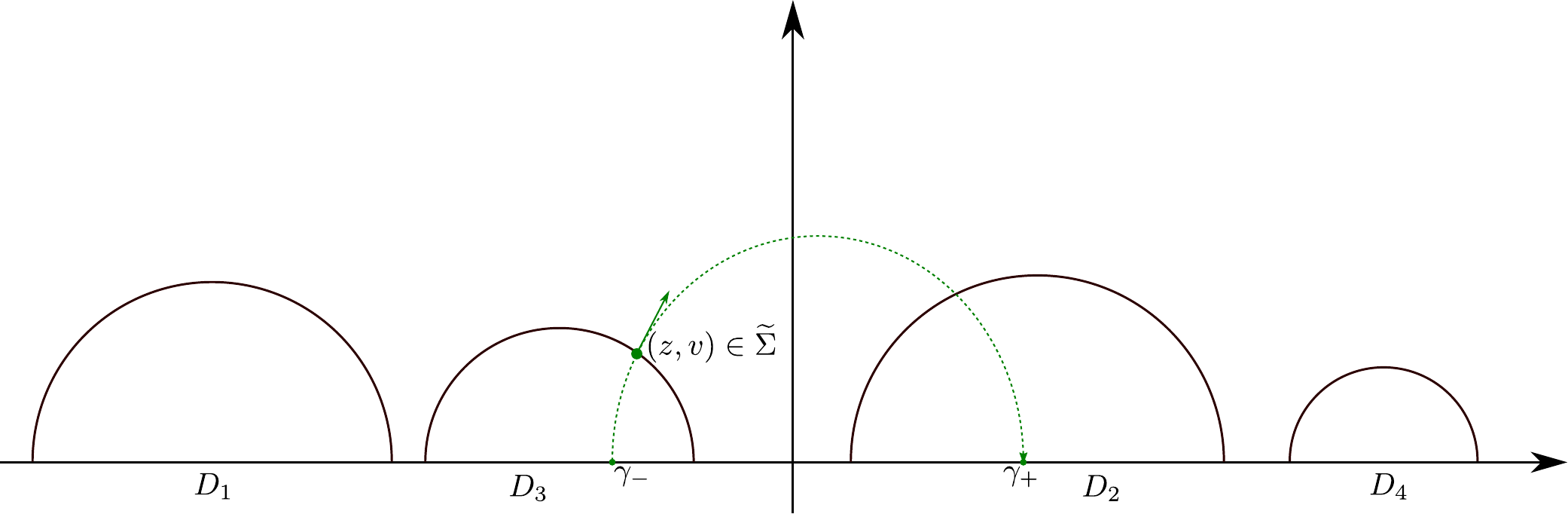}
	\caption{Sketch of a point $(z, v)$ in the fundamental domain $\widetilde{\Sigma}$ for the preimage of the Poincar\'{e} section $\Sigma$ together with its Hopf coordinates $(\gamma_-, \gamma_+)$ in the coordinate space $\bigcup_{i \neq j} I_i\times I_j$.}
	\label{fig4}
\end{figure}

To leverage the simplicity of the geodesic flow in Hopf coordinates we choose an adapted distance function $\mathrm{d}_{G}$ in the following fashion:
Denoting by $\mathrm{d}_{G / P}$ some $K$-invariant metric on the boundary of hyperbolic space let
\begin{equation*}
\mathrm{d}_G(g, g')^2 \defgr \mathrm{d}_{G / P}(g_+, g'_+)^2 + \mathrm{d}_{G \slash P}(g_-, g'_-)^2 + \vert t - t'\vert^2
\end{equation*}
in the respective Hopf coordinates of $g$ and $g'$.
In terms of dynamics the resulting Gaussian amounts to a product of Gaussians in the contracting, expanding, and neutral (flow-) directions.

\begin{remark}
	Note that the resulting distance function is not analytic on the whole domain $\mathbb{R}\times (G\slash P)_\Delta$ if we make the obvious choice
	of angle coordinates on the unit circle $\mathrm{S}^1 \cong G\slash P$ such that $\mathrm{d}_{G \slash P}^2$ becomes absolute value
	square on $[0, 2\pi]$ with the endpoints identified.
	While this is certainly true we do not require this analyticity for an application of Corollary~\ref{cor_fredholm_det} but only the
	analyticity of the \emph{resulting potential} as a function on the fundamental intervals $I_i$, c.f.~the discussion surrounding~\eqref{eq_period_integral_coord}.
	The final expression below will satisfy this analyticity making the theory of dynamical determinants developed above applicable here.
	From a theoretical standpoint the discussions involving the Poincar\'{e} section $\Sigma$ should be viewed more as an interpretation of
	what the potential given below means geometrically.
\end{remark}

Having made the choices above we can now calculate
$\mathrm{t}^\Sigma_{\lambda_0, \sigma}\circ \mathcal{H}^{-1}|_{\mathcal{H}(\widetilde{\Sigma})}$ using dynamical determinants if we have suitable expressions for the following period integrals over closed geodesics $\gamma \subseteq S\mathbf{X}_\Gamma$:
\begin{equation*}
\bigcup_{i \neq j} I_i\times I_j\ni (g_+, g_-) \mapsto \int_\gamma f_{\sigma, \mathcal{H}^{-1}(t(g_+, g_-), g_+, g_-)} ~.
\end{equation*}
First we may exploit $\Gamma$-invariance to re-write the integral as a sum over intersections with the fundamental domain.
The argument here is quite similar to the previous Section~\ref{algo.1}.
If the geodesic $\gamma$ is again represented by an isometry $g_w$, $w = (i_1, \hdots, i_n) \in \mathcal{W}_n$ then these intersections are determined by fixed points $(\gamma^i_+, \gamma^i_-)$ of cyclic permutations
\begin{equation*}
g_{i_n} g_{i_{n-1}}\cdots g_{i_1},~ g_{i_1} g_{i_n} \cdots g_{i_2},~ g_{i_2} g_{i_1} g_{i_n} \cdots g_{i_3},~ \text{etc.} ~,
\end{equation*}
and the period integrals become
\begin{equation*}
\begin{split}
&\quad ~ \int_\gamma f_{\sigma, \mathcal{H}^{-1}(t(g_+, g_-), g_+, g_-)}\\
&= \mathcal{M}_\sigma^{-1} \sum_{i = 1}^n \mathrm{e}^{-\frac{\mathrm{d}_{G\slash P}(g_+, \gamma^i_+)^2 + \mathrm{d}_{G\slash P}(g_-, \gamma^i_-)^2}{\sigma^2}}
\int_{-\infty}^{\infty} \mathrm{e}^{-\frac{\vert t(g_+, g_-) - t\vert^2}{\sigma^2}} \mathrm{d}t ~.
\end{split}
\end{equation*}
Here the second integral comes from taking the $t$-entry of Hopf coordinates for the lift $\widetilde{\gamma}^i$ of $\gamma$ with endpoints at infinity $(\gamma^i_+, \gamma^i_-)$ and integrating its distance from $t(g_+, g_-)$ over the whole geodesic $\widetilde{\gamma}^i$ of $\gamma$.
If we combine this with an evaluation of the constant $\mathcal{M}_\sigma$, which reduces to iterated Gaussian integrals, our period integrals becomes the rather handy expression
\begin{equation}\label{eq_poincare_integrals}
\int_\gamma f_{\sigma, \mathcal{H}^{-1}(t(g_+, g_-), g_+, g_-)} = \frac{1}{\pi\sigma^2} \sum_{i = 1}^n \mathrm{e}^{-\frac{\mathrm{d}_{G\slash P}(g_+, \gamma^i_+)^2 + \mathrm{d}_{G\slash P}(g_+, \gamma^i_-)^2}{\sigma^2}} ~.
\end{equation}

Again the approximation $\mathrm{t}^{\Sigma}_{\lambda_0, \sigma}$ restricted to $\Sigma$ can be calculated via the residues of a weighted zeta function, concretely $Z_{f_{\sigma, \mathcal{H}^{-1}(t(g_+, g_-), g_+, g_-)}}$.
The latter is straightforward to implement using Snippet~\ref{alg1} but substituting the routine $\verb|calcInt|$ with an implementation of \eqref{eq_poincare_integrals}. Furthermore if we take Equation~\eqref{eq_poincare_integrals} as a definition it does indeed yield an analytic potential which makes Corollary~\ref{cor_fredholm_det} immediately applicable.

\begin{remark}\label{remark1}
	Note that it is conceptually straight forward to replace the specific hypersurface $\Sigma$ with another choice $\Sigma'$.
	One has to make sure that $\Sigma'$ admits a family of test functions for which period integrals can be calculated efficiently.
	In most applications this should come down to finding an appropriate parametrization for $\Sigma'$, adapting the test functions to this parametrization, and finally calculating the period integrals in this parametrization (using suitable approximations).
\end{remark}


\section{Symmetry Reduction of Weighted Zeta Functions}\label{sym_red}

If we use the theoretical and practical tools developed up to this point it turns out that we require a large amount of closed geodesics and corresponding period integrals to compute the dynamical determinant with sufficient accuracy.
In this section we will therefore develop a method that allows us to exploit inherent symmetries of different classes of Schottky surfaces to significantly reduce the required computational resources.
Our approach to symmetry reduction essentially is an adaptation and generalization of similar work done by Borthwick and Weich \cite{Weich.2016} in the context of \emph{iterated function schemes}.
Even though these systems only incorporate an expanding direction it is quite straight forward to include the contracting direction present in the determinants constructed in Section~\ref{dyn_det}.

\begin{remark}
	It should be rather straight forward to formulate and prove a version of the upcoming symmetry reduced dynamical determinant in the full setting of~\cite{Rugh.1992}.
	We refrain from explicitly treating this greater generality here to keep the discussion aligned with the previous sections, in particular Section~\ref{dyn_det}.
	A practically important generalization will instead be included in the first author's PhD thesis, see also Section~\ref{numerics}.
\end{remark}

This section is organized as follows:
In Section~\ref{sym_red.1} we give some basic definitions and present the main theorem stating how our dynamical determinant $d_f$ decomposes as a product of symmetry reduced dynamical determinants with the product indexed by irreducible, unitary representations of some suitable symmetry group.
While this theorem describes the theoretical situation completely it is not directly accessible to practical implementation.
Section~\ref{sym_red.2} remedies this by providing a detailed description how the symmetry reduction can be implemented.
The final Section~\ref{sym_red.3} contains a short comparison of the computational effort needed to calculate dynamical determinants with and without symmetry reduction.

\subsection{Main Theorem}\label{sym_red.1}

We begin by stating the fundamental definition of what the symmetry group of a particular representation of some Schottky surface should be \cite[Theorem~3.1]{Weich.2016}:
\begin{defn}\label{def_sym_group}
	Let $\Gamma\subseteq \mathrm{SL}(2, \mathbb{R})$ be a rank-$r$ Schottky group with generators $g_1, \hdots, g_r$ and fundamental discs $D_1, \hdots, D_{2r}$.
	Let $\mathbf{G}$ be a finite group acting on $\bigcup_i D_i$ by holomorphic functions extending continuously to the boundary and
	define a $\mathbf{G}$-action on $\{1, \ldots, 2r\}$ by the relation $\mathbf{g}(D_i) = D_{\mathbf{g}\cdot i}$.	
	
	Then $\mathbf{G}$ is called a \emph{symmetry group of the generating set $\langle g_1, ..., g_r\rangle = \Gamma$} if for any $\mathbf{g}\in \mathbf{G}$ and $i\neq j\in\{1, \hdots, 2r\}$ there exists an index $k\in\{1, \hdots, 2r\}$ such that $k\neq \mathbf{g}\cdot i$ and
	\begin{equation*}
	\mathbf{g}\cdot (g_j z) = g_k (\mathbf{g}\cdot z), \qquad \forall z\in D_i ~,
	\end{equation*}
	
\end{defn}
Because $\mathbf{g}\in\mathbf{G}$ acts on each disc as a biholomorphic map the image $\mathbf{g}(D_i)$ must again be some disc making $\mathbf{g}\cdot i$ well-defined. The index $k$ in the relation $\mathbf{g}\cdot (g_j z) = g_k (\mathbf{g}\cdot z)$ is unique and we observe that $k = \mathbf{g}\cdot j$.
Furthermore the group action of $\mathbf{G}$ on the elements of $\mathbb{Z} \slash 2r\mathbb{Z}$ takes distinct pairs of indices $i\neq j$ to distinct pairs: $\mathbf{g}\cdot i\neq \mathbf{g}\cdot j$.
The action of an element $\mathbf{g}\in\mathbf{G}$ can then be written concisely as
\begin{equation*}
\mathbf{g}\cdot (g_j z) = g_{\mathbf{g}\cdot j} (\mathbf{g}\cdot z), \qquad \forall z\in D_i,\, j\neq i ~.
\end{equation*}
We extend this action to $(\mathbb{Z} \slash 2r\mathbb{Z})^n$ (or $\mathcal{W}_n$) by acting on each element separately.

The desired composition of our dynamical (Fredholm) determinants now follows from the rather simple representation theory of finite groups by observing that $\mathbf{G}$ acts on the function spaces introduced in Section~\ref{dyn_det.1} via the \emph{left-regular representation}:
\begin{equation*}
\left\langle \mathbf{g}\cdot u(z_1, z_2), v(z_1) \right\rangle \defgr \left\langle u(\mathbf{g}^{-1}\cdot z_1, \mathbf{g}^{-1}\cdot z_2), v(z_1) \right\rangle, \qquad u\in \bigoplus_{i\neq j} \mathcal{H}^{-2}(D_i)\otimes \mathcal{H}^2(D_j) ~.
\end{equation*}
This representation will generally not be unitary because we defined the $\mathrm{L}^2$-scalar product via Lebesgue measure.
The standard trick of averaging the pushforward of Lebesgue measure over the finite group $\mathbf{G}$ guarantees unitary, though.
This yields a modified Bergman space which contains the same functions but whose scalar product differs from the standard one used so far by some smooth density factor.

Well known representation theory of finite groups \cite[Part~I]{Fulton.2004} now provides a direct sum decomposition of this modified Bergman space indexed by characters $\chi$ of (equivalence classes $\widehat{\mathbf{G}}$ of) irreducible, unitary representations of $\mathbf{G}$ with the projectors on the individual summands given by
\begin{equation*}
\mathrm{P}_\chi \defgr \frac{d_\chi}{\vert \mathbf{G}\vert} \sum_{\mathbf{g}\in\mathbf{G}} \overline{\chi(\mathbf{g})} \mathbf{g} ~.
\end{equation*}
In this equation $d_\chi$ refers to the dimension of the representation with character $\chi$ and $\vert \mathbf{G}\vert$ denotes the cardinality of $\mathbf{G}$.
As the projections do not involve the scalar product we immediately derive a corresponding non-orthogonal direct sum decomposition of our original Hilbert spaces:
\begin{equation*}
\bigoplus_{i\neq j} \mathcal{H}^{-2}(D_i)\otimes \mathcal{H}^2(D_j) = \bigoplus_{\chi\in \widehat{\mathbf{G}}} \mathcal{H}_\chi, \qquad \mathcal{H}_\chi\defgr \mathrm{P}_\chi\bigg( \bigoplus_{i\neq j} \mathcal{H}^{-2}(D_i)\otimes \mathcal{H}^2(D_j) \bigg)~.
\end{equation*}

The transfer operator $\mathrm{L}_V$ defined in Section~\ref{dyn_det.1} now commutes with the action of $\mathbf{G}$ if the potential $V$ is $\mathbf{G}$-invariant, i.e. $V(\mathbf{g}^{-1} \cdot z_1, \mathbf{g}^{-1}\cdot z_2) = V(z_1, z_2)$:
\begin{equation*}
\begin{split}
&\qquad \left\langle \mathrm{L}_V(\mathbf{g}\cdot u)(z_1, z_2), v(z_1) \right\rangle \bigg|_{\substack{v\in \mathcal{H}^2(D_i) \\ z_2\in D_j}}\\
&= \int_{\partial D_i} u(\mathbf{g}^{-1} \cdot g_i(z_1), \mathbf{g}^{-1} \cdot g_i(z_2)) V(z_1, z_2) v(z_1) \frac{\mathrm{d} z_1}{2\pi\mathrm{i}}\\
&= \int_{\partial D_i} u(g_{\mathbf{g}^{-1} \cdot i}(\mathbf{g}^{-1} \cdot z_1), g_{\mathbf{g}^{-1} \cdot i}(\mathbf{g}^{-1} \cdot z_2)) V(\mathbf{g}^{-1}\cdot z_1, \mathbf{g}^{-1}\cdot z_2) v(z_1) \frac{\mathrm{d} z_1}{2\pi\mathrm{i}}\\
&= \left\langle \mathbf{g}\cdot \big( \mathrm{L}_V u \big)(z_1, z_2), v(z_1) \right\rangle \bigg|_{\substack{v\in \mathcal{H}^2(D_i) \\ z_2\in D_j}} ~.
\end{split}
\end{equation*}
From this it is obvious that $\mathrm{L}_V$ commutes with the projections $\mathrm{P}_\chi$ which makes the subspaces $\mathcal{H}_\chi$ invariant under $\mathrm{L}_V$.
This simple observation is already the key to the main factorization theorem of this section.
Before we can state said theorem we need one additional definition:
Given an element $\mathbf{g}\in \mathbf{G}$ we define a symmetry adapted version of $\mathcal{W}_n$ as follows
\begin{equation*}
\mathcal{W}^\mathbf{g}_n \defgr \big\{ (i_1, \hdots, i_n) \in (\mathbb{Z} \slash 2r\mathbb{Z})^n \,\big|\, i_j\neq i_{j + 1} + r\, \forall\, j\in \{1, \hdots, n - 1\} ,\, \mathbf{g}\cdot i_n\neq i_1 + r \big\} ~.
\end{equation*}
In unison with \cite{Weich.2016} we refer to elements of $\mathcal{W}^\mathbf{g}_n$ as \emph{$\mathbf{g}$-closed words of length $n$}.

For the main theorem first note that $g_w\circ \mathbf{g}$ has a unique pair of repelling and attracting fixed points $x^{w, \mathbf{g}}_-\in D_{i_n + r}$ and $x^{w, \mathbf{g}}_+\in D_{\mathbf{g}^{-1}\cdot i_1}$ because $\mathbf{g}$ is biholomorphic and the same argument as in \cite[Lemma~2.6]{Weich.2016} can be applied with domains $D_{i_n + r}$ and $D_{\mathbf{g}^{-1}\cdot i_1}$.
It can be stated as follows:
\begin{theorem}[\cite{Weich.2016}, Prop.~3.3 and Thm.~4.1]\label{thm_sym_factorization}
	Let $\Gamma$ be a Schottky group with symmetry group $\mathbf{G}$ and $V$ a $\mathbf{G}$-invariant potential which is analytic in a neighborhood of the fundamental circles.
	Then the Fredholm determinant of the transfer operator $\mathrm{L}_V$ factorizes as
	\begin{equation*}
	\det\left( \mathrm{id} - z\mathrm{L}_V \right) = \prod_{\chi\in \widehat{\mathbf{G}}} d_{V, \chi}(z) ~,
	\end{equation*}
	where for sufficiently small $\vert z\vert$ the \emph{symmetry reduced determinants} $d_{V, \chi}$ are explicitly given by the following expressions:
	\begin{equation*}
	d_{V, \chi}(z) = \exp\left( -\sum_{n = 1}^{\infty} \frac{z^n}{n} \frac{d_\chi}{\vert \mathbf{G}\vert} \sum_{\mathbf{g}\in \mathbf{G}} \chi(\mathbf{g}) \sum_{w\in \mathcal{W}^\mathbf{g}_n} \frac{V_w(\mathbf{g}\cdot x^{w, \mathbf{g}}_-, \mathbf{g}\cdot x^{w, \mathbf{g}}_+)}{((g_w\circ \mathbf{g})'(x^{w, \mathbf{g}}_-) - 1)(1 - (g_w\circ \mathbf{g})'(x^{w, \mathbf{g}}_+))} \right) ~.
	\end{equation*}
\end{theorem}

\begin{proof}
	By the previously discussed decomposition of the domain of $\mathrm{L}_V$ into the direct sum $\bigoplus_{\chi\in \widehat{\mathbf{G}}} \mathcal{H}_\chi$ it is clear that
	\begin{equation*}
	\det\left( \mathrm{id} - z\mathrm{L}_V \right) = \prod_{\chi\in \widehat{\mathbf{G}}} \det\left( \mathrm{id} - z\mathrm{L}_V\big|_{\mathcal{H}_\chi} \right) ~.
	\end{equation*}
	We are therefore tasked with computing $d_{V, \chi}(z) \defgr \det\big( \mathrm{id} - z\mathrm{L}_V\big|_{\mathcal{H}_\chi} \big)$.
	This Fredholm determinant can be expressed in terms of traces of $n$-fold iterates just as in \eqref{eq_log_det}.
	The key is therefore evaluating
	\begin{equation*}
	\mathrm{Tr}\bigg( \mathrm{L}_V^n\big|_{\mathcal{H}_\chi} \bigg) = \mathrm{Tr}\big( \mathrm{P}_\chi \mathrm{L}_V^n \big) ~,
	\end{equation*}
	where the equality immediately follows from the fact that $\mathrm{L}_V$ commutes with the projections $\mathrm{P}_\chi$ and the trace-class property is inherited from $\mathrm{L}^n_V$.
	
	Now the (diagonal) components of $n$-fold iterates of $\mathrm{L}_V$ were already calculated in \eqref{eq_transfer_comp}.
	A very similar calculation yields, after replacing the sum over $\mathbf{g}$ by with a sum over $\mathbf{g}^{-1}$, the following expression:
	\begin{equation*}
	\mathrm{Tr}\bigg( \mathrm{L}_V^n\big|_{\mathcal{H}_\chi} \bigg) = \frac{d_\chi}{\vert \mathbf{G}\vert} \sum_{\mathbf{g}\in \mathbf{G}} \chi(\mathbf{g}) \sum_{w\in \mathcal{W}^\mathbf{g}_n} \mathrm{Tr}\big( \mathrm{L}_{V, \mathbf{g}}^w \big) ~,
	\end{equation*}
	where the modified diagonal components $\mathrm{L}_{V, \mathbf{g}}^w$, $w\in \mathcal{W}^\mathbf{g}_n$, are operators of the form
	\begin{equation*}
	\left\langle \mathrm{L}_{V, \mathbf{g}}^w u(z_1, z_2), v(z_1) \right\rangle \bigg|_{\substack{v\in \mathcal{H}^2(D_{i_n + r}) \\ z_2\in D_{\mathbf{g}^{-1} i_1}}} = \int_{\partial D_{i_n + r}} V_g(\mathbf{g}\cdot z_1, \mathbf{g}\cdot z_2) u(g_w(\mathbf{g}\cdot z_1), g_w(\mathbf{g} \cdot z_2)) v(z_1) \frac{\mathrm{d} z_1}{2\pi\mathrm{i}} ~.
	\end{equation*}
	Here the traces of these operators can be calculated in complete analogy to Theorem~\ref{thm_fredholm_det} but now the pair of fixed points
	$x^{w, \mathbf{g}}_-\in D_{i_n + r}$ and $x^{w, \mathbf{g}}_+\in D_{\mathbf{g}^{-1}\cdot i_1}$ of the holomorphic map $g_w\circ \mathbf{g}$ appears. We finally obtain
	\begin{equation*}
	\mathrm{Tr}\left( \mathrm{L}^w_{V, \mathbf{g}} \right) = \frac{V_w(\mathbf{g}\cdot x^{w, \mathbf{g}}_-, \mathbf{g}\cdot x^{w, \mathbf{g}}_+)}{((g_w\circ \mathbf{g})'(x^{w, \mathbf{g}}_-) - 1)(1 - (g_w\circ \mathbf{g})'(x^{w, \mathbf{g}}_+))} ~,
	\end{equation*}
	which finishes the proof.
\end{proof}

\subsection{Implementation Details}\label{sym_red.2}

One of the key observations in \cite{Weich.2016} is a clever grouping of terms in $d_{V, \chi}$.
On the one hand it provides convergence beyond small $\vert z\vert$ and on the other hand it speeds up the practical computation of symmetry reduced zeta functions tremendously.
An adaptation of this work allows us to achieve both these benefits for our weighted zeta functions as well.
The following presentation closely follows \cite[Sections~4--5]{Weich.2016}.

To meet our objective we first introduce some additional notation:
Observing that elements $w\in \mathcal{W}_n^\mathbf{g}$ always occur together with their closing group element $\mathbf{g}$ in Theorem~\ref{thm_sym_factorization} it makes sense to define
\begin{equation*}
\mathcal{W}^\mathbf{G} \defgr \left\{ (w, \mathbf{g})\in \bigg( \bigcup_{n = 1}^\infty \mathcal{W}_n \bigg) \times \mathbf{G} \,\bigg|\, \mathbf{g}\cdot w_{n_w} \neq w_1 + r \right\} ~,
\end{equation*}
where $n_w$ denotes the length of $w$, i.e. $w\in \mathcal{W}_{n_w}$, and we index words as $w = (w_1, \hdots, w_{n_w})$.
We will henceforth denote elements of $\mathcal{W}^\mathbf{G}$ by boldface letters and indicate their first and second components by the corresponding non-boldface letter and a suitable subscript:
$\mathbf{w} = (w, \mathbf{g}_\mathbf{w})\in \mathcal{W}^\mathbf{G}$.
This definition immediately allows us to shorten the notation for fixed points introduced in the previous section by setting $x^\mathbf{w}_\pm \defgr x^{w, \mathbf{g}_\mathbf{w}}_\pm$.

The actual regrouping of terms now happens due to a derived action of $\mathbf{G}\times \mathbb{Z}$ on $\mathcal{W}^\mathbf{G}$ which we will describe next.
First of all an element $\mathbf{h}\in \mathbf{G}$ acts on $\mathcal{W}^\mathbf{G}$ via
\begin{equation*}
\mathbf{h}\cdot \mathbf{w} = \mathbf{h}\cdot (w, \mathbf{g}_\mathbf{w}) \defgr (\mathbf{h}\cdot w, \mathbf{h} \mathbf{g}_\mathbf{w} \mathbf{h}^{-1}) ~.
\end{equation*}
This action is complemented by the $\mathbb{Z}$-action generated by the two (inverse) \emph{shifts} acting in the first component as
\begin{equation*}
\begin{split}
\sigma_R \mathbf{w} &\defgr \big( (\mathbf{g}_\mathbf{w} w_n, w_1, \ldots, w_{n - 1}), \mathbf{g}_\mathbf{w} \big) ~,\\
\sigma_L \mathbf{w} &\defgr \big( (w_{2}, \ldots, w_n, \mathbf{g}_\mathbf{w}^{-1} w_1), \mathbf{g}_\mathbf{w} \big) ~.
\end{split}
\end{equation*}
This $\mathbb{Z}$-action commutes with the action of $\mathbf{G}$.
We can therefore consider the space $\big[ \mathcal{W}^\mathbf{G} \big] \defgr (\mathbf{G}\times \mathbb{Z}) \setminus \mathcal{W}^\mathbf{G}$ of orbits under the product of these actions and we denote the equivalence class containing $\mathbf{w}\in \mathcal{W}^\mathbf{G}$ by $[\mathbf{w}]$.

Before we can re-organize the sums appearing in Theorem~\ref{thm_sym_factorization} it remains to identify an appropriate notion of \emph{composite} elements.
To this end we define the \emph{$k$-fold iteration} of $\mathbf{w}\in \mathcal{W}^\mathbf{G}$ by the formula
\begin{equation*}
\mathbf{w}^k \defgr \big( (\mathbf{g}_\mathbf{w}^{k-1} w_1, \ldots, \mathbf{g}_\mathbf{w}^{k-1} w_n, \mathbf{g}_\mathbf{w}^{k-2} w_1, \ldots, \mathbf{g}_\mathbf{w} w_1, \ldots, \mathbf{g}_\mathbf{w} w_n, w_1, \ldots w_n), \mathbf{g}_\mathbf{w}^k  \big) ~,
\end{equation*}
and we call an element of $\mathcal{W}^\mathbf{G}$ \emph{prime} if it cannot be represented as such an iteration for $k > 1$.
Otherwise we call the element composite.

Next we recall some features of the various notions just introduced.
Complete proofs can be found in~\cite[Lemma~4.3, Prop.~4.4]{Weich.2016}.
We continue to denote by $V(z_1, z_2)$ some $\mathbf{G}$-invariant potential.
\begin{enumerate}
	\item An orbit of the $\mathbf{G}\times \mathbb{Z}$-action consists either entirely of prime or entirely of composite elements making $\left[ \mathcal{W}^\mathbf{G}_\mathrm{p} \right] \defgr \{ [\mathbf{w}] \,|\, w ~\text{prime} \}$ well-defined;
	\item Given $\mathbf{v}\in [\mathbf{w}^k]$ one has the equalities $V_v(\mathbf{g}_\mathbf{v} x^\mathbf{v}_-, \mathbf{g}_\mathbf{v} x^\mathbf{v}_+) = V_w(\mathbf{g}_\mathbf{w} x^\mathbf{w}_-, \mathbf{g}_\mathbf{w} x^\mathbf{w}_+)^k$ and $(g_v\circ \mathbf{g}_\mathbf{v})'(x^\mathbf{v}_\pm) = (g_w\circ \mathbf{g}_\mathbf{w})'(x^\mathbf{w}_\pm)^k$;
	\item If $\mathbf{G}$ acts freely on $\mathcal{W}^\mathbf{G}$ then the number of elements in the equivalence class $[\mathbf{w}]\in \big[ \mathcal{W}^\mathbf{G}_\mathrm{p} \big]$ can be calculated as $\# [\mathbf{w}] = \vert\mathbf{G} \vert \cdot n_\mathbf{w}$;
	\item If $m_\mathbf{w}\defgr \mathrm{ord}(\mathbf{g})$ denotes the \emph{order}\footnote{
		I.e. the smallest integer $k > 0$ such that $\mathbf{g}^k = \mathrm{id}_\mathbf{G}$.
	}
	of $\mathbf{g}\in\mathbf{G}$ then the equalities $(g_w\circ \mathbf{g}_\mathbf{w})'(x^\mathbf{w}_\pm) = g_{w^{m_\mathbf{w}}}'(x^{w^{m_\mathbf{w}}}_\pm)^{1 / m_\mathbf{w}}$ and $V_w(\mathbf{g}_\mathbf{w} \cdot x^\mathbf{w}_\pm) = V_{w^{m_\mathbf{w}}}(x^{w^{m_\mathbf{w}}}_\pm)^{1 / m_\mathbf{w}}$ hold for any $[\mathbf{w}]\in \big[ \mathcal{W}^\mathbf{G}_\mathrm{p} \big]$ and denoting by $x^v_\pm$ the fixed points of $g_v$.\footnote{
		And if the respective right-hand sides are real-valued, which is always the case for our particular potentials.
	}
\end{enumerate}
With this we can now reformulate Theorem~\ref{thm_sym_factorization} in a form which is very reminiscent of our definition of the dynamical determinant $d_f$ in \eqref{def_df}.
From the preceding discussion combined with Theorem~\ref{thm_sym_factorization} and under the assumptions of $\mathbf{G}$-invariant potential and free $\mathbf{G}$-action on $\mathcal{W}^\mathbf{G}$ we immediately deduce the following for sufficiently small $\vert z\vert$
\begin{equation} \label{eq_symmetry_reduced_general}
\begin{split}
d_{V, \chi}(z) &= \exp\left( - d_\chi \sum_{[\mathbf{w}]\in \big[\mathcal{W}^\mathbf{G} \big]} \frac{\#[\mathbf{w}]\cdot z^{n_\mathbf{w}}}{n_\mathbf{w}\cdot \vert\mathbf{G} \vert} \frac{\chi(\mathbf{g}_\mathbf{w}) V_w(\mathbf{g}_\mathbf{w}\cdot x^\mathbf{w}_-, \mathbf{g}_\mathbf{w}\cdot x^\mathbf{w}_+)}{((g_w\circ \mathbf{g}_\mathbf{w})'(x^\mathbf{w}_-) - 1) (1 - (g_w\circ \mathbf{g}_\mathbf{w})'(x^\mathbf{w}_+))} \right)\\
&= \exp\left( - d_\chi \sum_{k = 1}^\infty \sum_{[\mathbf{w}] \in \big[\mathcal{W}^\mathbf{G}_\mathrm{p} \big]} \frac{\#[\mathbf{w}^k] z^{k n_\mathbf{w}}}{k n_\mathbf{w} \vert\mathbf{G}\vert} \frac{\chi(\mathbf{g}_\mathbf{w}^k) \cdot V_w(\mathbf{g}_\mathbf{w}\cdot x^\mathbf{w}_-, \mathbf{g}_\mathbf{w}\cdot x^\mathbf{w}_+)^k}{((g_w\circ \mathbf{g}_\mathbf{w})'(x^\mathbf{w}_-)^k - 1) (1 - (g_w\circ \mathbf{g}_\mathbf{w})'(x^\mathbf{w}_+)^k)} \right)\\
&= \exp\left( - d_\chi \sum_{k = 1}^\infty \sum_{[\mathbf{w}]\in \big[\mathcal{W}^\mathbf{G}_\mathrm{p} \big]} \frac{z^{k n_\mathbf{w}}}{k} \frac{\chi(\mathbf{g}_\mathbf{w}^k) \cdot V_{w^{m_\mathbf{w}}}(x^{w^{m_\mathbf{w}}}_-, x^{w^{m_\mathbf{w}}}_+)^{k/m_\mathbf{w}}}{\big( \mathrm{e}^{ k T_{\gamma(g_{w^{m_\mathbf{w}}})} / m_\mathbf{w} } - 1 \big) \big( 1 - \mathrm{e}^{ -k T_{\gamma(g_{w^{m_\mathbf{w}}})} / m_\mathbf{w} } \big)} \right) ~.
\end{split}
\end{equation}
From here we can proceed by applying the cycle expansion philosophy introduced in Section~\ref{cycle_exp}.
This results in the following Proposition which generalizes Corollaries~\ref{coro_log_deriv} as well as \ref{coro_cycle_exp} and serves as our primary tool for practical implementation.
We require a final definition before presenting the actual statement:
Given a weight function $f\in\mathrm{C}^\omega(S\mathbf{X}_\Gamma)$ we say that it has \emph{$\mathbf{G}$-invariant period integrals} if its integrals over closed geodesics satisfy
\begin{equation*}
\int_{\gamma(g_w)} f = \int_{\gamma(g_{\mathbf{g}\cdot w})} f ~, \qquad \forall \mathbf{g}\in\mathbf{G},\, \forall w\in \mathcal{W}_n ~.
\end{equation*}
The main example for this is given by the case where $\mathbf{G}$ even acts on the whole unit sphere bundle $S\mathbf{H}$.
It then acts on $\Gamma$-invariant elements of $\mathrm{C}^\infty(S\mathbb{H})$ by translation and if $f$ is invariant under this $\mathbf{G}$-action then it also has $\mathbf{G}$-invariant period integrals because $\gamma(g_{\mathbf{g}\cdot w}) = \mathbf{g}\cdot \gamma(g_w)$ in this setting.

The central proposition now reads as follows:
\begin{prop}\label{prop_sym_factorization}
	Let $\Gamma$ be a Schottky group with symmetry group $\mathbf{G}$ acting freely on $\mathcal{W}^\mathbf{G}$ and $f\in \mathrm{C}^\omega(S \mathbf{X}_\Gamma)$ an analytic weight with $\mathbf{G}$-invariant period integrals.
	Then the following hold:
	\begin{enumerate}
		\item The dynamical determinant decomposes as a product
		\begin{equation*}
		d_f(\lambda, z, \beta) = \prod_{\chi\in \widehat{\mathbf{G}}} d_{f, \chi}(\lambda, z, \beta)
		\end{equation*}
		of symmetry reduced dynamical determinants $d_{f, \chi}$.
		\item The $d_{f, \chi}$ are given by an everywhere convergent power series expansion
		\begin{equation*}
		d_{f, \chi}(\lambda, z, \beta) = 1 + \sum_{n = 1}^\infty d_n^\chi(\lambda, \beta) z^n ~,
		\end{equation*}
		with coefficients in this expansion being holomorphic functions in $(\lambda, \beta)$ and explicitly given by the recursion
		\begin{equation*}
		\begin{split}
		d_n^\chi(\lambda, \beta) &= \sum_{m = 1}^n \frac{m}{n} d_{n - m}^\chi(\lambda, \beta) a_m^\chi(\lambda, \beta) ~, \qquad d_0^\chi(\lambda, \beta) \equiv 1 ~,\\
		a_m^\chi(\lambda, \beta) &= - d_\chi \sum_{\substack{([\mathbf{w}], k) \in \big[\mathcal{W}^\mathbf{G}_\mathrm{p} \big]\times \mathbb{N}_{> 0} \\ n_\mathbf{w}\cdot k = m}} \frac{\chi(\mathbf{g}_\mathbf{w}^k)}{k} \frac{\exp\big( -\frac{k(\lambda - 1)}{m_\mathbf{w}} T_{\gamma(g_{w^{m_\mathbf{w}}})} - \frac{k\beta}{m_\mathbf{w}} \int_{\gamma(g_{w^{m_\mathbf{w}}})} f \big)}{\big( \mathrm{e}^{kT_{\gamma(g_{w^{m_\mathbf{w}}})} / m_\mathbf{w}} - 1 \big)^2} ~.
		\end{split}
		\end{equation*}
		The coefficients $d_n^\chi(\lambda, \beta)$ satisfies super-exponential bounds of the same kind as in Corollary~\ref{coro_cycle_exp}.
		\item The weighted zeta function $Z_f$ decomposes as a sum of meromorphic functions given by logarithmic derivatives of the $d_{f, \chi}$:
		\begin{equation*}
		Z_f(\lambda) = \sum_{\chi\in\widehat{\mathbf{G}}} \frac{\partial_\beta d_{f, \chi}(\lambda, 1, 0)}{d_{f, \chi}(\lambda, 1, 0)} ~.
		\end{equation*}
	\end{enumerate}
\end{prop}
\begin{proof}
	The idea of proof is very much aligned with the material presented in Sections~\ref{dyn_det} and \ref{cycle_exp}:
	To derive (1) we would like to plug the concrete potential of Corollary~\ref{cor_fredholm_det} into the product decomposition of Theorem~\ref{thm_sym_factorization}.
	
	A slight difficulty with this strategy is the fact that the potential $V_{\lambda, \beta}$ of Corollary~\ref{cor_fredholm_det} is not $\mathbf{G}$-invariant due to the presence of the derivative term.
	This can be remedied by substituting it with a $\mathbf{G}$-averaged version possessing built-in $\mathbf{G}$-invariant, c.f.~\cite[Lemma~5.6]{Weich.2016}:
	\begin{equation*}
	V_{\lambda, \beta}^\mathbf{G} \defgr \prod_{\mathbf{g}\in\mathbf{G}} V_{\lambda, \beta}(\mathbf{g} \cdot z_1, \mathbf{g}\cdot z_2)^{1 / \vert\mathbf{G}\vert} ~.
	\end{equation*}
	By \eqref{eq_transfer_det} we only need to verify $(V_{\lambda, \beta})_w = (V_{\lambda, \beta}^\mathbf{G})_w$ for any closed word $w\in\mathcal{W}_n$ to prove that the transfer operators associated with these potentials give rise to the same Fredholm determinant.
	Using the $\mathbf{G}$-invariance of the period integrals of $f$ and the proof of \cite[Lemma~5.6]{Weich.2016} (which is essentially an application of the ordinary chain rule) we can now calculate
	\begin{equation*}
	\begin{split}
	\big( V_{\lambda, \beta}^\mathbf{G} \big)_w &= \exp\left( -\frac{\beta}{\vert\mathbf{G}\vert} \sum_{\mathbf{g}\in\mathbf{G}} \int_{\gamma(g_{\mathbf{g}\cdot w})} f \right) \cdot \mathrm{e}^{-\lambda T_{\gamma(g_w)}}\\
	&= \exp\left( -\lambda T_{\gamma(g_w)} - \beta \int_{\gamma(g_w)} f \right) = \big( V_{\lambda, \beta} \big)_w ~.
	\end{split}
	\end{equation*}
	This proves (1) if we define $d_{f, \chi}(\lambda, z, \beta) \defgr d_{V_{\lambda, \beta}^\mathbf{G}, \chi}(z)$.
	
	To obtain the expression for $d_{f, \chi}(\lambda, z, \beta)$ claimed in (2), we start with the following formula derived for a general potential in \eqref{eq_symmetry_reduced_general} above and valid for sufficiently small $\vert z\vert$
	\begin{equation*}
	\begin{split}
	d_{f, \chi}(\lambda, z, \beta)
	&= \exp\left( - d_\chi \sum_{k = 1}^\infty \sum_{[\mathbf{w}]\in \big[\mathcal{W}^\mathbf{G}_\mathrm{p} \big]} \frac{z^{k n_\mathbf{w}}}{k} \frac{\chi(\mathbf{g}_\mathbf{w}^k) \cdot \exp\big( -\frac{k\lambda}{m_\mathbf{w}} T_{\gamma(g_{w^{m_\mathbf{w}}})} - \frac{k\beta}{m_\mathbf{w}} \int_{\gamma(g_{w^{m_\mathbf{w}}})} f \big)}{\big( \mathrm{e}^{ k T_{\gamma(g_{w^{m_\mathbf{w}}})} / m_\mathbf{w} } - 1 \big) \big( 1 - \mathrm{e}^{ -k T_{\gamma(g_{w^{m_\mathbf{w}}})} / m_\mathbf{w} } \big)} \right)\\
	&= \exp\left( \sum_{n = 1}^\infty a_n^\chi(\lambda, \beta) z^n \right) ~,
	\end{split}
	\end{equation*}
	with the definition
	\begin{equation*}
	a_n^\chi(\lambda, \beta) \defgr - d_\chi \sum_{\substack{([\mathbf{w}], k) \in \big[\mathcal{W}^\mathbf{G}_\mathrm{p} \big]\times \mathbb{N}_{> 0} \\ n_\mathbf{w}\cdot k = n}} \frac{\chi(\mathbf{g}_\mathbf{w}^k)}{k} \frac{\exp\big( -\frac{k\lambda}{m_\mathbf{w}} T_{\gamma(g_{w^{m_\mathbf{w}}})} - \frac{k\beta}{m_\mathbf{w}} \int_{\gamma(g_{w^{m_\mathbf{w}}})} f \big)}{\big( \mathrm{e}^{kT_{\gamma(g_{w^{m_\mathbf{w}}})} / m_\mathbf{w}} - 1 \big) \big( 1 - \mathrm{e}^{- kT_{\gamma(g_{w^{m_\mathbf{w}}})} / m_\mathbf{w} } \big)} ~.
	\end{equation*}
	The theory of Bell polynomials already used in the proof of Corollary~\ref{coro_cycle_exp} now yields the claimed recursive relation.
	
	The super-exponential bounds can be derived in the same manner as in Corollary~\ref{coro_cycle_exp} by simply observing that the singular values of the restrictions satisfy $\mu_i(\mathrm{P}_\chi \mathrm{L}_V) \leq \Arrowvert\mathrm{P}_\chi \Arrowvert \mu_i(\mathrm{L}_V)$.
	We can therefore recycle our previous calculations and arrive at the same bounds with possibly different constants.
	
	Lastly we derive (3) in the following straightforward manner:
	Exchanging the logarithm of the product over $\widehat{\mathbf{G}}$ with a sum over logarithms in Corollary~\ref{coro_log_deriv} lets us calculate
	\begin{equation*}
	\begin{split}
	Z_f(\lambda) &= \partial_\beta \log\left( \det\big( \mathrm{id} - z \mathrm{L}_{\lambda, \beta}^f \big) \right) \bigg|_{z=1, \beta=0}\\
	&= \partial_\beta \log\left( \det\big( \mathrm{id} - z \mathrm{L}_{V_{\lambda, \beta}^\mathbf{G}} \big) \right) \bigg|_{z=1, \beta=0}\\
	&= \sum_{\chi\in\widehat{\mathbf{G}}} \partial_\beta \log\left( d_{f, \chi}(\lambda, z, \beta) \right) \bigg|_{z=1, \beta=0} ~,
	\end{split}
	\end{equation*}
	finishing our proof.
\end{proof}

For practical purposes the following form for the coefficients occurring in the base of the recursion is more suitable:
\begin{equation*}
\begin{split}
a_m^\chi(\lambda, \beta) &= - \frac{d_\chi}{m} \sum_{\substack{[\mathbf{w}] \in \big[\mathcal{W}^\mathbf{G}_\mathrm{p} \big] \\ n_\mathbf{w} || m}} n_\mathbf{w}\cdot \chi(\mathbf{g}_\mathbf{w}^{m / n_\mathbf{w}}) \frac{\exp\big( -\frac{m(\lambda - 1)}{n_\mathbf{w} m_\mathbf{w}} T_{\gamma(g_{w^{m_\mathbf{w}}})} - \frac{m\beta}{n_\mathbf{w} m_\mathbf{w}} \int_{\gamma(g_{w^{m_\mathbf{w}}})} f \big)}{\big( \mathrm{e}^{mT_{\gamma(g_{w^{m_\mathbf{w}}})} / (n_\mathbf{w} m_\mathbf{w})} - 1 \big)^2} ~,
\end{split}
\end{equation*}
where $n || m$ for $n, m\in \mathbb{N}$ indicates that $n$ divides $m$.
Rescaling both lengths and period integrals by $1 / (n_\mathbf{w} m_\mathbf{w})$ like this yields a rather convenient way for vectorized evaluation of the coefficients $a_m^\chi$.

\begin{remark}
	The proposition does not reveal if and why any improvement in terms of convergence should be expected from the coefficients $d_n^\chi$ compared with their non-symmetry reduced counterparts $d_n$.
	The practical calculations below will reveal a significant improvement, though.
	For a theoretical discussion of this phenomenon we refer to \cite[Appendix~B]{Weich.2016}.
\end{remark}

\begin{remark}
	If we want to apply Proposition~\ref{prop_sym_factorization} we have to adjust the families of test functions presented in Section~\ref{algo} to satisfy $\mathbf{G}$-invariance of their period integrals.
	But this is fairly straightforward:
	We simply replace the expressions derived in \eqref{eq_calcInt_fundamental} and \eqref{eq_poincare_integrals} by similar sums but over points $x_j + \mathrm{i}y_j$ or $(\gamma^i_-, \gamma^i_+)$, respectively, derived not only from the original word $w$ but from all words $\mathbf{g}\cdot w$ for $\mathbf{g}\in \mathbf{G}$.
	A subsequent normalization by $\vert\mathbf{G}\vert$ then yields a suitable prescription.
\end{remark}

\begin{algorithm}
	\caption{%
		Pseudo-code for the calculation of symmetry reduced determinants $d_\chi$ for a given character $\chi$.
		Iterating this procedure over all characters $\chi\in\widehat{\mathbf{G}}$ and combining it with the
		first block of Snippet~\ref{alg1} yields a means for the symmetry reduced calculation of approximations
		to invariant Ruelle distributions.
	}\label{alg3}
	
	\KwIn{width $\sigma > 0$, center $gK\in G/K$, resonance $\lambda_0\in\mathbb{C}$, character $\chi\in \widehat{\mathbf{G}}$, cut-off $N\in\mathbb{N}$}
	
	\Comment{calculate Bell recursion as in Snippet~\ref{alg1}, but with $a[]$ replaced by modified initial terms $a^\chi[]$}
	
	\Comment{calculate modified initial terms $a^\chi[k]$ in Bell recursion}
	$a^\chi[k] \gets 0$\;
	$\partial a^\chi[k] \gets 0$\;
	\For{$[\mathbf{w}]\in\left[ \mathcal{W}_p^\mathbf{G} \right] ~\mathrm{where}~ n_\mathbf{w} || k$}{
		$a^\chi[k] \gets a^\chi[k] - n_\mathbf{w} \chi(\mathbf{g}_\mathbf{w}^{k / n_\mathbf{w}}) \frac{\exp(-(\lambda_0 - 1) k \ell(g_w) / (n_\mathbf{w} m_\mathbf{w}))}{(\mathrm{e}^{k \ell(g_w) / (n_\mathbf{w} m_\mathbf{w})} - 1)^2}$\;
		$\partial_\beta a[k] \gets \partial_\beta a[k] + n_\mathbf{w} \chi(\mathbf{g}_\mathbf{w}^{k / n_\mathbf{w}}) k \frac{\exp(-(\lambda_0 - 1) k \ell(g_w)) / (n_\mathbf{w} m_\mathbf{w})}{(\mathrm{e}^{k \ell(g_w) / (n_\mathbf{w} m_\mathbf{w})} - 1)^2}\cdot \mathrm{calcInt}(w)$ \;
		\Comment*[r]{$\mathrm{calcInt}$ implements \eqref{eq_calcInt_fundamental}}
		$a^\chi[k] \gets \frac{d_\chi}{k} a^\chi[k]$\;
		$\partial_\beta a[k] \gets \frac{d_\chi}{k} \partial_\beta a[k]$\;
	}
	
	\KwResult{approximation of one summand in the symmetry decomposition of $Z_{f_{\sigma, gK}\circ \pi}(\lambda_0)$ according to Proposition~\ref{prop_sym_factorization}}
\end{algorithm}

\subsection{Example Surfaces}\label{sym_red.3}

In order to make use of Proposition~\ref{prop_sym_factorization} in practice one requires example surfaces with sufficiently rich symmetry groups.
To meet this requirement we provide two classes of such examples covering both topological possibilities for Schottky surfaces of rank $r = 2$:
The three-funneled surfaces and the funneled tori.

\subsubsection{Three-funneled surface}\label{sym_red.three_funnel}

The family of three-funneled surfaces constitutes the main class of examples in the original paper of Borthwick \cite{Borthwick.2014} and many subsequent papers on the numerical calculation of resonances \cite{Weich.2016,Pohl.2020}.
Due to this prevalence will we refrain from giving too many details and simply state their generators:
\begin{equation*}
	g_1 \defgr \begin{pmatrix}
	\cosh(\ell_1 / 2) & \sinh(\ell_1 / 2) \\ \sinh(\ell_1 / 2) & \cosh(\ell_1 / 2)
	\end{pmatrix}
	, \qquad
	g_2 \defgr \begin{pmatrix}
	\cosh(\ell_2 / 2) & a \sinh(\ell_2 / 2) \\ a^{-1} \sinh(\ell_2 / 2) & \cosh(\ell_2 / 2)
	\end{pmatrix}
\end{equation*}
The three numbers $\ell_1, \ell_2, \ell_3 > 0$ parameterize the family and can be interpreted geometrically as the lengths of the closed geodesics winding around the funnels.
The parameter $a$ is not free but must be chosen such that the condition $\mathrm{tr}(g_1 g_2^{-1}) = -2\cosh(\ell_3 / 2)$ is fulfilled.
We follow the notation introduced in \cite{Borthwick.2014} and denote the generated surface by $X(\ell_1, \ell_2, \ell_3)$.

The following realization $\mathbf{G} = \{ \mathrm{e}, \sigma_1, \sigma_2, \sigma_1\sigma_2 \}$ of Klein's four-group is a symmetry group of the generators given above provided that $\ell_1 = \ell_2$:\footnote{
	For an elementary proof see~\cite[Example~3.2]{Weich.2016}.
}
\begin{equation*}
	\sigma_1 = \begin{pmatrix}
	-1 & 0 \\ 0 & 1
	\end{pmatrix}
	, \qquad
	\sigma_2 = \begin{pmatrix}
	0 & \sqrt{a} \\ 1/\sqrt{a} & 0
	\end{pmatrix} ~,
\end{equation*}
where both matrices act on $\mathbb{H}$ via Möbius transformations.
The action on the letters is then given by
\begin{equation}\label{eq_sym_action}
\begin{split}
&\sigma_1(1) = 3,~ \sigma_1(2) = 4,~ \sigma_1(3) = 1,~ \sigma_1(4) = 2\\
&\sigma_2(1) = 2,~ \sigma_2(2) = 1,~ \sigma_2(3) = 4,~ \sigma_2(4) = 3 ~.
\end{split}
\end{equation}

The final ingredient for the calculation of symmetry reduced zeta functions is the character table of $\mathbf{G}$.
This well-known data is given in Table~\ref{tab1}.

\begin{table}[h!]
	\centering
	\caption{Character table for the symmetry group of three-funneled surfaces.}
	\begin{tabular}{|| c | c | c | c | c ||}
		 \hline
		 & $\mathrm{e}$ & $\sigma_1$ & $\sigma_2$ & $\sigma_1\sigma_2$\\
		 \hline
		 \hline
		 A & $1$ & $1$ & $1$ & $1$\\
		 \hline
		 B & $1$ & $-1$ & $1$ & $-1$\\
		 \hline
		 C & $1$ & $1$ & $-1$ & $-1$\\
		 \hline
		 D & $1$ & $-1$ & $-1$ & $1$\\
		 \hline
	\end{tabular}
	\label{tab1}
\end{table}

In the less symmetric case $\ell_1 \neq \ell_2$ the generators lose their symmetry with respect to $\sigma_2$ but retain $\{ \mathrm{e}, \sigma_1 \}$ as their symmetry group.
This smaller group has only two irreducible representations, namely the trivial one and one that equals $-1$ on $\sigma_1$.
Both are one-dimensional.

\begin{figure}
	\centering
	\includegraphics[scale=0.55]{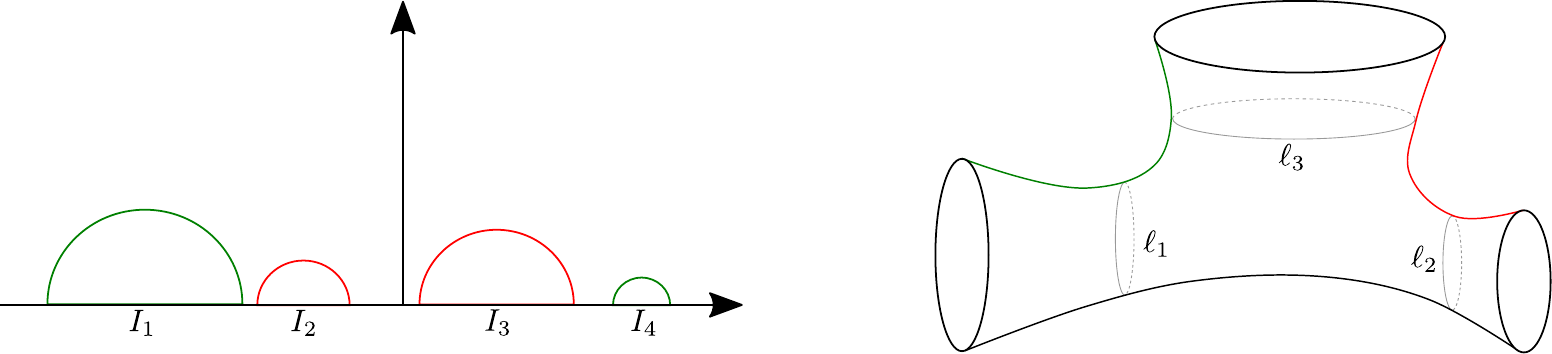}
	\caption{%
		The right-hand side illustrates the three-funneled surface with boundary lengths $\ell_1, \ell_2, \ell_3$
		embedded (non-isometrically) into three-dimensional Euclidean space. Highlighted in red and green is the
		Poincar\'{e} section (projected onto the surface) corresponding to the fundamental domain on the left.
	}\label{fig_threefunnel_sym}
\end{figure}

\subsubsection{Funneled torus}\label{sym_red.torus}

Our second family of surfaces is also well-known in the literature.
Generators are given by
\begin{equation*}
\begin{split}
	g_1 &\defgr \begin{pmatrix}
	\exp(\ell_1 / 2) & 0\\ 0 & \exp(-\ell_1 / 2)
	\end{pmatrix}~,\\
	g_2 &\defgr \begin{pmatrix}
	\cosh(\ell_2 / 2) - \cos(\varphi)\sinh(\ell_2 / 2) & \sin^2(\varphi)\sinh(\ell_2 / 2) \\
	\sinh(\ell_2 / 2) & \cosh(\ell_2 / 2) + \cos(\varphi)\sinh(\ell_2 / 2)
	\end{pmatrix} ~,
\end{split}
\end{equation*}
where again three parameters $\ell_1, \ell_2$, and $\varphi$ are needed to specify a concrete member.
Geometrically they describe the lengths of two closed geodesics on the surface and the angle between them.
Obviously the generator $g_1$ contains the boundary point $\infty$ in its fundamental interval which makes it inaccessible for our algorithm because we use the fundamental intervals directly as coordinates of the Poincar\'{e} section.
Conjugating the generators by a simple rotation yields a new pair $\widetilde{g}_1$ and $\widetilde{g}_2$ of generators which do not suffer from this problem.
It turns out that $\pi / 8$ is a particularly handy value in the maximally symmetric case $\ell_1 = \ell_2$ and $\varphi = \pi / 2$ as it leads to a very symmetric arrangement of fundamental circles (see also \cite[Section~5.3]{Pohl.2020} where the boundary point $\infty$ must be rotated outside of the fundamental intervals for somewhat similar reasons).
In this case the conjugated generators are of the explicit form
\begin{equation*}
\begin{split}
	\widetilde{g}_1 &\defgr \begin{pmatrix}
	\cosh(\ell / 2) + \sinh(\ell / 2) / \sqrt{2} & \sinh(\ell / 2) / \sqrt{2}\\
	\sinh(\ell / 2) / \sqrt{2} & \cosh(\ell / 2) - \sinh(\ell / 2) / \sqrt{2}
	\end{pmatrix}\\
	\widetilde{g}_2 &\defgr \begin{pmatrix}
	\cosh(\ell / 2) - \sinh(\ell / 2) / \sqrt{2} & \sinh(\ell / 2) / \sqrt{2}\\
	\sinh(\ell / 2) / \sqrt{2} & \cosh(\ell / 2) + \sinh(\ell / 2) / \sqrt{2}
	\end{pmatrix} ~.
\end{split}
\end{equation*}
Again we adopt the same notation as in \cite{Borthwick.2014} and denote the funneled tori generated by $\widetilde{g}_1$ and $\widetilde{g}_2$ by $Y(\ell_1, \ell_2, \varphi)$.

The generators $\widetilde{g}_1$ and $\widetilde{g}_2$ again admit a realization of Klein's four-group as their symmetry group.
Concrete symmetries are given by the Möbius transformations induced via
\begin{equation*}
\sigma_1 = \begin{pmatrix}
0 & 1\\ 1 & 0
\end{pmatrix}
, \qquad
\sigma_2 = \begin{pmatrix}
0 & 1 \\ -1 & 0
\end{pmatrix} ~,
\end{equation*}
with their action on the symbols being the same as in \eqref{eq_sym_action}.
One can easily verify these relations by elementary matrix calculations of the type $\sigma_1 \widetilde{g}_1 \sigma_1^{-1} = \widetilde{g}_2$.
The character table thus coincides with the one given above in Table~\ref{tab1}.

\begin{figure}
	\centering
	\includegraphics[scale=0.55]{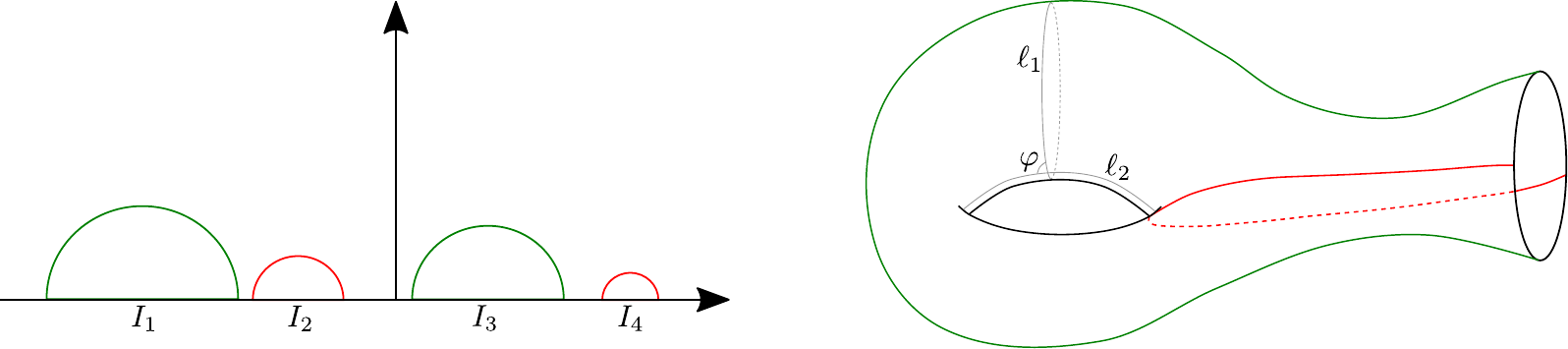}
	\caption{%
		Here the right-hand side illustrates the funneled torus with parameters $\ell_1, \ell_2, \phi$
		embedded (non-isometrically) into three-dimensional Euclidean space. Highlighted in red and green is the
		Poincar\'{e} section (projected onto the surface) corresponding to the fundamental domain on the left.
	}\label{fig_funneltorus_sym}
\end{figure}


\section{Numerical Results}\label{numerics}

This final section finishes the present paper by presenting some numerical calculations
which were obtained with the tools developed up to this point. We begin with a comparison of
convergence rates of invariant Ruelle distributions depending on the particular symmetry
group used for a given surface.

\begin{remark}\label{remark_residues}
	To do this we still require a practically feasible approach for the calculation of residues of $Z_f(\lambda)$.
	If one can calculate the function $Z_f(\lambda)$ efficiently\footnote{
		E.g. by vectorization or via a distributed computational scheme.
	}
	for a large number of support points $\lambda\in\mathbb{C}$, i.e. on large arrays, then it would be possible to calculate its residue at $\lambda_0$ very generically via the classical integral formula
	\begin{equation*}
	\underset{\lambda=\lambda_0}{\mathrm{Res}} \left[ Z_f(\lambda) \right] = \frac{1}{2\pi\mathrm{i}} \int_C Z_f(\lambda)~\mathrm{d}\lambda ~,
	\end{equation*}
	where a convenient choice for the contour $C$ could e.g. be a sufficiently small rectangle with $\lambda_0$ at its center.
	This integral can then be evaluated using numerical quadrature methods.
	
	If a large number of function evaluations is too computationally expensive then the following well-known formula offers an alternative:
	\begin{equation*}
	\underset{\lambda=\lambda_0}{\mathrm{Res}} \left[ Z_f(\lambda) \right] = \frac{1}{(N-1)!}
	\lim_{\lambda\rightarrow \lambda_0}\left( \frac{\mathrm{d}}{\mathrm{d}\lambda} \right)^{N-1}
	\left[ (\lambda - \lambda_0)^N Z_f(\lambda) \right] ~,
	\end{equation*}
	where $N$ denotes the order of the pole $\lambda_0$.
	If $N = 1$, i.e. $\lambda_0$ is a simple pole, then this general formula takes a particularly simple shape and
	plugging in the logarithmic derivative of the dynamical determinant $d_f$ yields for the simple case
	\begin{equation} \label{eq_residue_practical}
	\underset{\lambda=\lambda_0}{\mathrm{Res}} \left[ Z_f(\lambda) \right] = \frac{\partial_\beta d_f(\lambda_0, 1, 0)}{\partial_\lambda d_f(\lambda_0, 1, 0)} ~.
	\end{equation}
	This expression can be evaluated directly because calculation of $\partial_\lambda d_f$ requires only a straightforward modification of our algorithm for $d_f$ and $\partial_\beta d_f$.
	The numerics presented below feature only simple poles so we used \eqref{eq_residue_practical} throughout our implementations.\footnote{
		Our resonance calculations were obtained with a root finding algorithm that combines the argument principle
		from complex analysis with the classical Newton iteration. In particular our procedure always yields pairs
		of resonances and corresponding orders.
	}
\end{remark}

\subsection*{Funneled Torus Experiments}

The first system we consider is the funneled torus for the two cases of its full $4$-element Kleinian symmetry
group as described in Section~\ref{sym_red.torus} and without any symmetry reduction.
The quantum resonance spectrum for the concrete example $Y\left(10, 10, \frac{\pi}{2}\right)$ was already obtained
numerically by Borthwick~\cite[Figure~11]{Borthwick.2014} by means of Selberg's zeta function.
In Figure~\ref{fig_torus_resonances}
we recover this spectrum but now using the dynamical determinant $d_f$ with constant weight function $f$.
This yields the Pollicott-Ruelle resonances of $Y\left(10, 10, \frac{\pi}{2}\right)$ which by~\cite{Weich.2018}
coincide with the quantum spectrum after a shift by $-1$. Compared with previous calculations of quantum
resonances in the literature our numerics therefore illustrate this theorem about the relationship between
classical and quantum resonances.

\begin{figure}[h]
	\centering
	\makebox[\textwidth][c]{
		\includegraphics[width=1.25\textwidth,trim={0mm 20mm 0mm 0mm}]{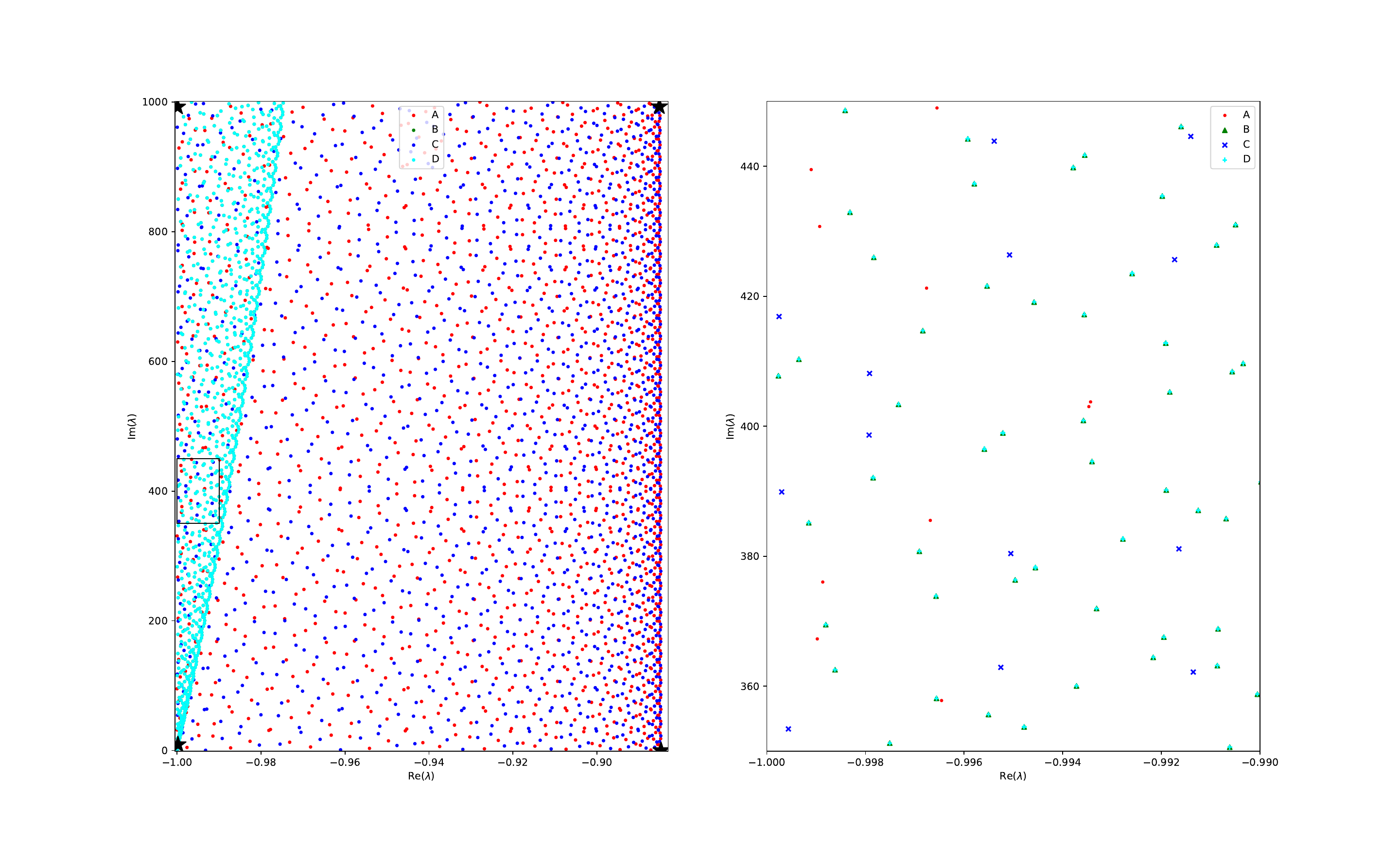}
	}
	\caption{%
		Resonances of the funneled torus $Y\left(10, 10, \frac{\pi}{2}\right)$ calculated in the symmetry reduction by the
		Klein four-group. The right-hand side shows a zoom into a small region of the left-hand side plot. Notice
		how the representations labeled B and D split resonances of multiplicity two into pairs of resonances of multiplicity
		one. In particular this simplifies the formula for the calculation of invariant Ruelle distributions
		via residues as described above.
		The four resonances marked by black stars in the left plot were used for investigations of Ruelle distributions below.
	}\label{fig_torus_resonances}
\end{figure}

We begin by investigating in more detail the first resonance of $Y(10, 10, \pi/2)$ which is located at
$\delta - 1 \approx -0.8847$ with $\delta$ the Hausdorff dimension of the limit set. The result of numerically
calculating $t^\Sigma_{\lambda_0, \sigma}$ is shown in Figure~\ref{fig_torus_first_distribution} as plots of
three different quantities: The left-most plot shows the real part of the distribution which gets complemented by
the imaginary part in the middle. The invariant
Ruelle distribution associated with the first resonance coincides with the Bowen-Margulis measure so it
should be expected that the numerical approximation is real-valued and positive which is exactly the case in the shown plot.
The right-most coordinate square features a combination of real and imaginary parts: There the complex argument
is indicated through the color of the peaks and the absolute value of $t^\Sigma_{\lambda_0, \sigma}$ was encoded
as the lightness of the particular color. The mapping of colors to complex arguments is simply given by the
angle on the standard color wheel in the HSB/HSL encoding of RGB shifted by $\pi$, i.e.~light blue corresponds
to an argument of $0$ while red corresponds to $\pi$ (see Figure~\ref{fig_colors}).
From this illustration it is immediately clear that the distribution is remarkably homogeneous within each square
$I_i\times I_j$ of the coordinate domain parameterizing the Poincar\'{e} section.

\begin{figure}
	\centering
	\includegraphics[scale=0.24]{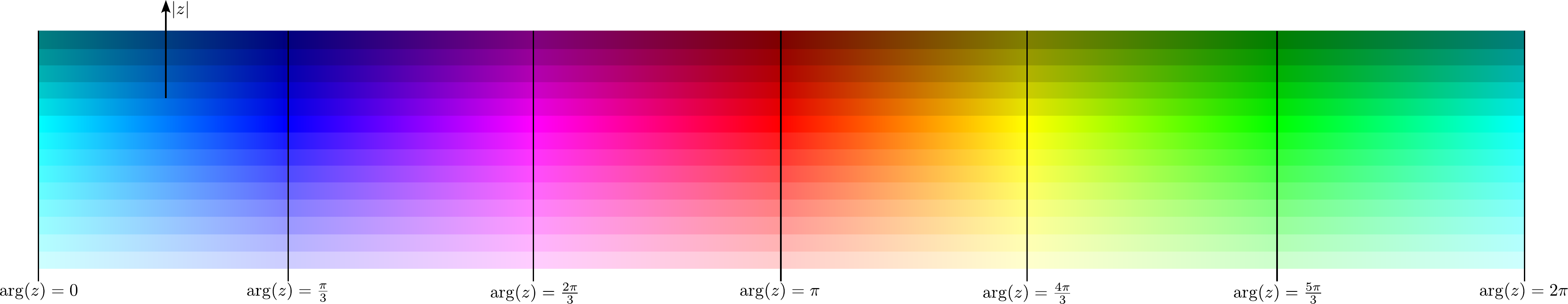}
	\caption{%
		Sketch of how complex arguments $\mathrm{arg}(z)$ map to different colors in the phase plots above. The absolute value
		$\vert z\vert$ maps to the lightness of the color as a second dimension with darker colors corresponding to higher
		values of $\vert z\vert$.
	}\label{fig_colors}
\end{figure}

\begin{figure}[h]
	\centering
	\makebox[\textwidth][c]{
		\includegraphics[width=1.10\textwidth,trim={85mm 185mm 65mm 5mm}]{%
			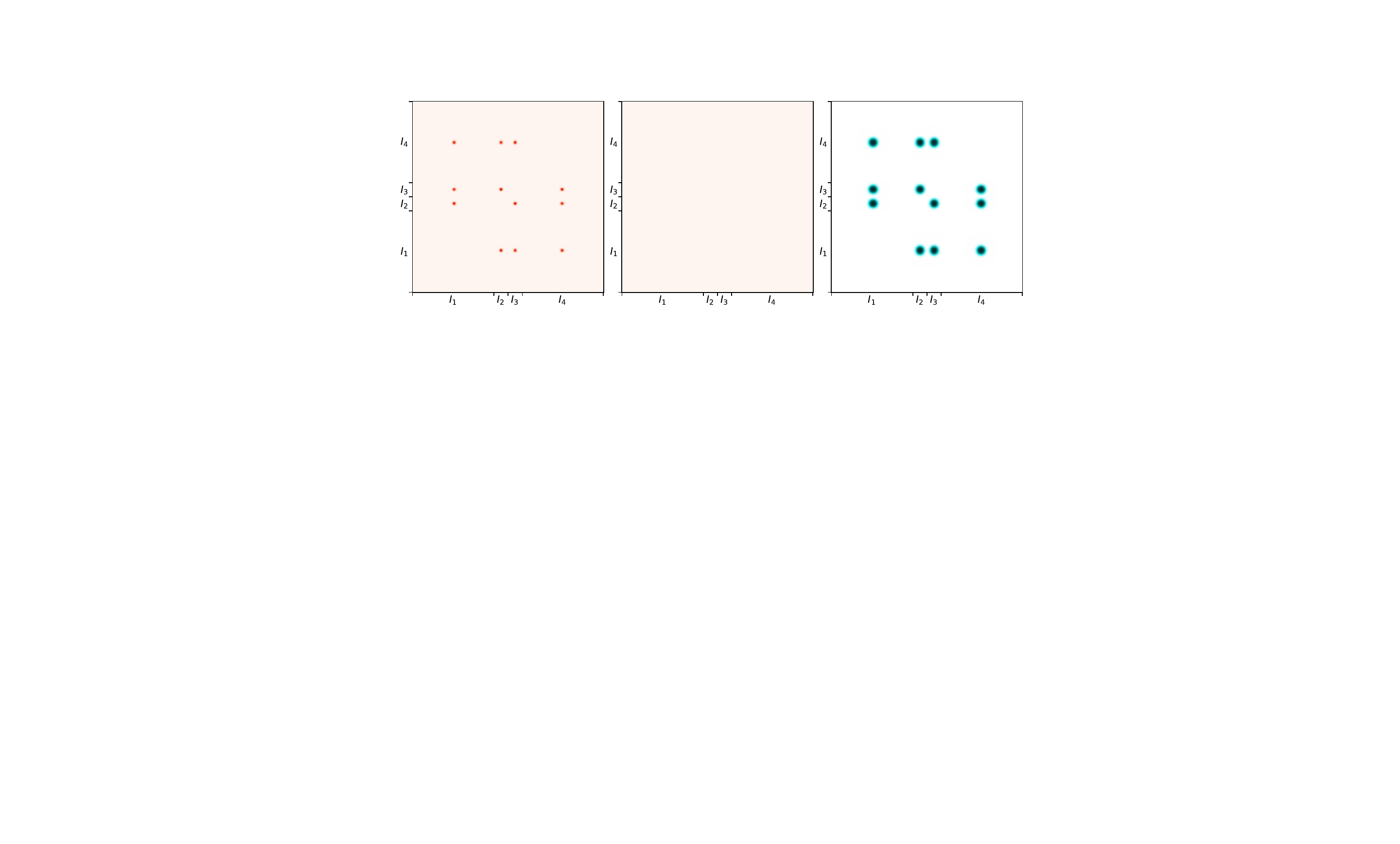}
	}
	\caption{%
		Invariant Ruelle distribution $t^\Sigma_{\lambda_0, \sigma}$ on the canonical Poincar\'{e} section
		$\Sigma$ of the funneled torus $Y(10, 10, \pi/2)$ associated with the first resonance
		$\lambda_0 \approx -0.8847$. This figure uses the width $\sigma = 10^{-3}$ and $n_\mathrm{max} = 5$
		summands in the cycle expansion of the dynamical determinant $d_f$. The fundamental intervals $I_i$
		which parameterize the section are ordered as qualitatively shown in Figure~\ref{fig_funneltorus_sym}
		The left and middle columns show the
		real and imaginary parts of the distribution. The right column encodes both real and imaginary parts by
		using the complex argument to determine the color (as an angle on the color wheel) and the absolute
		value to provide the lightness.
	}\label{fig_torus_first_distribution}
\end{figure}

Next we consider additional resonances from Figure~\ref{fig_torus_resonances} under the aspect of how well the
associated distributions converge in practice. It is reasonable to expect that the outer edges of the square
$[-1, \delta - 1] + [0, 1000]\mathrm{i}$
of the complex plane where resonances were calculated should correspond to the best respectively worst rates of
convergence. To increase the resolution of the distribution the plots were additionally restricted to
coordinates within strictly smaller intervals
\begin{equation*}
	(\widetilde{I}_2\cup \widetilde{I}_3\cup \widetilde{I}_3) \times (\widetilde{I}_1\cup \widetilde{I}_2) ~,
	\qquad \widetilde{I}_i\subset I_i ~,
\end{equation*}
where the location of the new intervals within the original fundamental domain is illustrated in
Figure~\ref{fig_domain_refinement}.
\begin{figure}[h]
	\centering
	\includegraphics[scale=0.42]{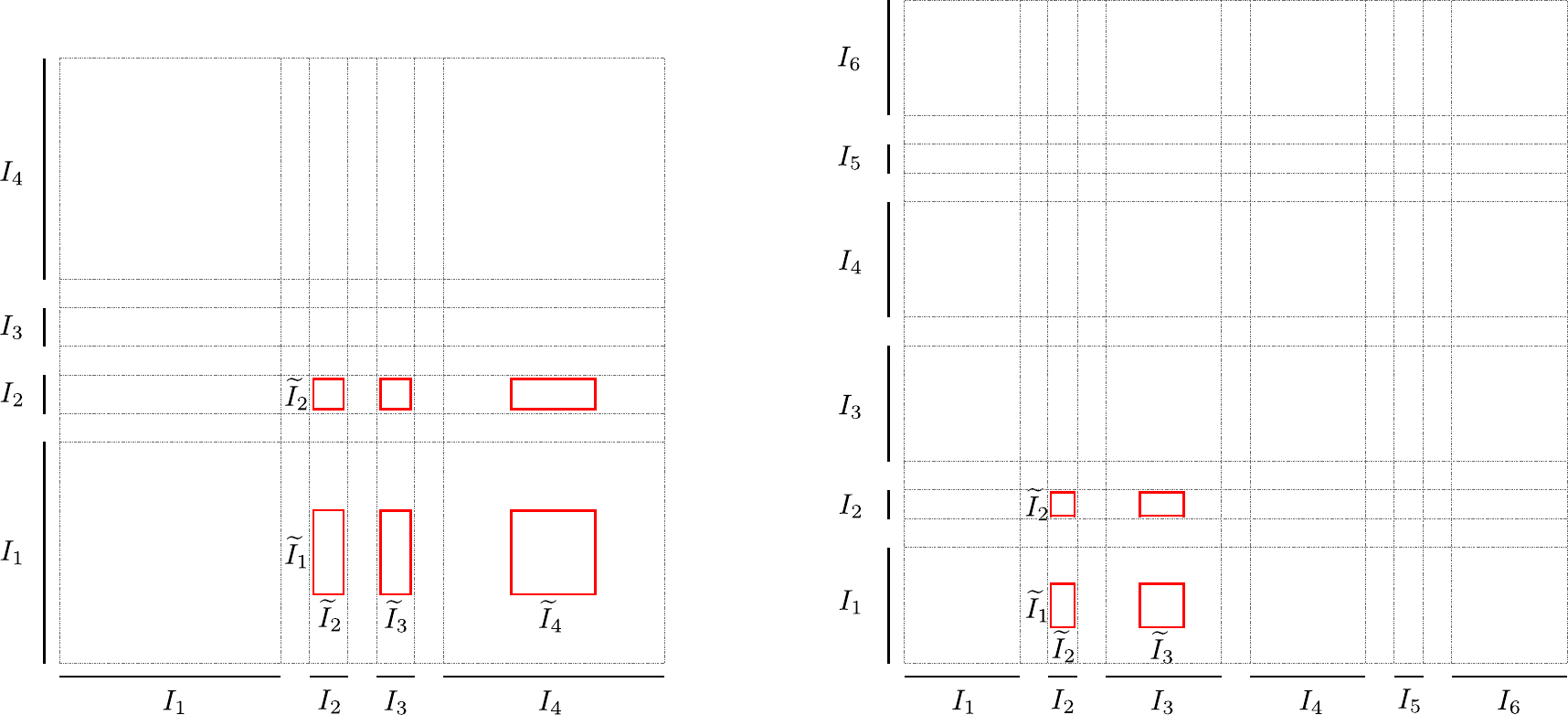}
	\caption{%
		Location and relative size of the refined coordinate domains $\widetilde{I}_i\subset I_i$ compared
		to the original fundamental intervals of the funneled torus $X(10, 10, \pi/2)$ (left) and the
		three-funneled surface $X(12, 12, 12)$ (right).
	}\label{fig_domain_refinement}
\end{figure}
This very basic \emph{domain refinement} already reduces the
amount of redundancy in the resulting plots significantly by excluding large areas where the distributions
vanish and exploiting the internal symmetries of the distributions as prominently visible in
Figure~\ref{fig_torus_first_distribution}.

The resulting plots for a collection of four resonances are depicted in Figure~\ref{fig_torus_reduced_distributions}:
Here the columns correspond to the same resonance whereas the rows share a common cutoff order $n_\mathrm{max}$,
i.e.~the number of summands used in the cycle expansion.
From this figure we see that even the distribution associated with the resonance
$\lambda_0\approx -0.9999 + 992.4\mathrm{i}$ in the upper left
of the considered resonance domain converges nicely already at $n_\mathrm{max} = 3$.

\begin{figure}[h]
	\centering
	\makebox[\textwidth][c]{
		\includegraphics[width=1.5\textwidth,trim={0mm 20mm 0mm 0mm}]{%
			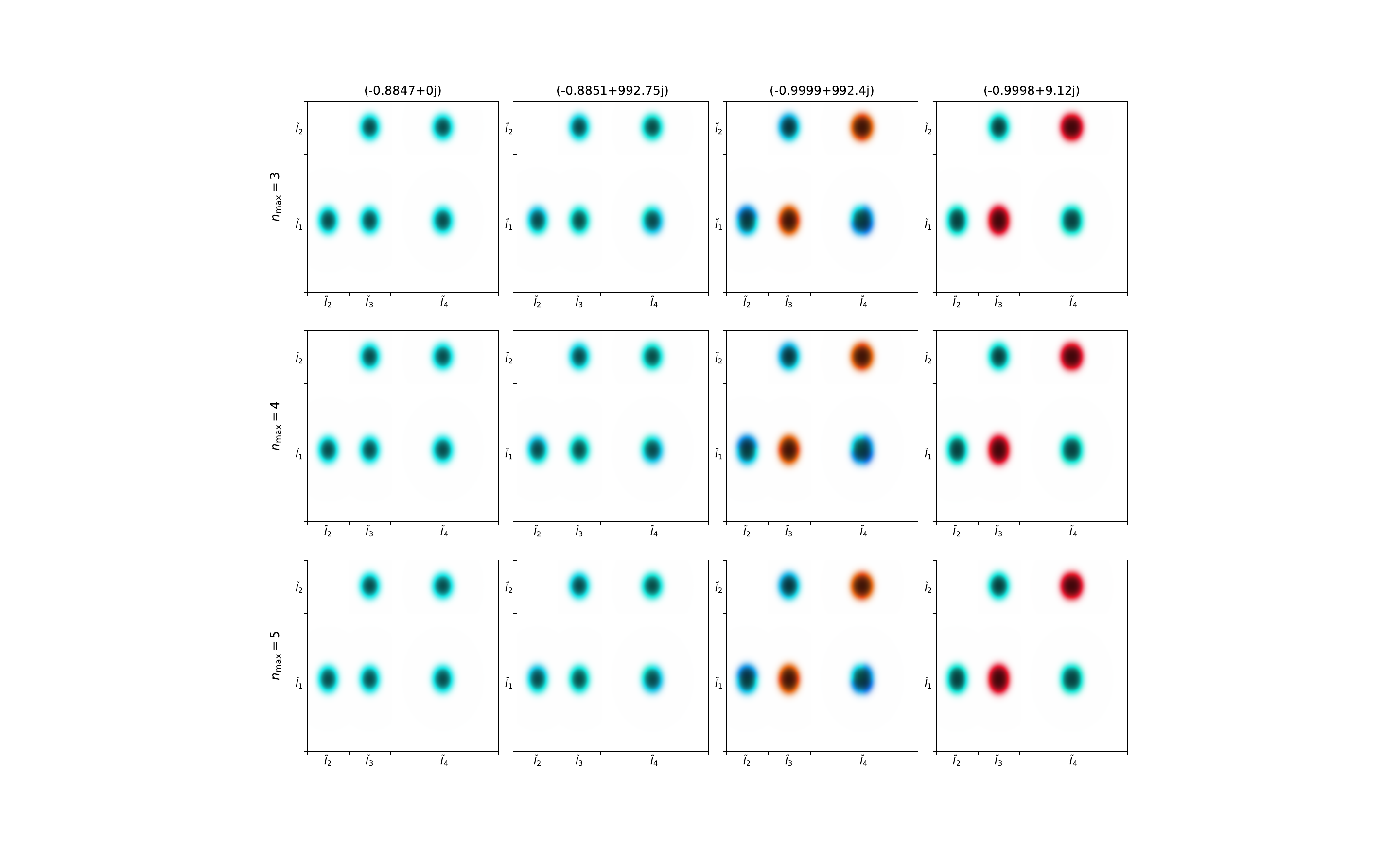}
	}
	\caption{%
		Fully symmetry reduced invariant Ruelle distributions $t^\Sigma_{\lambda_0, \sigma}$ on the canonical
		Poincar\'{e} section $\Sigma$ of the funneled torus $Y(10, 10, \pi/2)$ evaluated at four
		different choices of resonance $\lambda_0$ (marked in Figure~\ref{fig_torus_resonances})
		with $\sigma = 10^{-3}$. The three different rows
		show the numerical results of using (from top to bottom) $n_\mathrm{max}\in \{3, 4, 5\}$ summands in
		the cycle expansion. Following along any of the four columns shows that the presented plots have
		converged rather well already at $n_\mathrm{max} = 3$. Due to symmetries in the distributions as visible
		in Figure~\ref{fig_torus_first_distribution} it suffices to consider as the coordinate domain the refined subset
		$(\widetilde{I}_2\cup \widetilde{I}_3\cup \widetilde{I}_4) \times (\widetilde{I}_1\cup \widetilde{I}_2)$ of Figure~\ref{fig_domain_refinement}.
	}\label{fig_torus_reduced_distributions}
\end{figure}

One particularly noteworthy feature of these plots is the large degree of similarity between the first and second
columns both with respect to the absolute value as well as the complex argument.
An explanation for this qualitative agreement might be the fact that both associated resonances are located near the
global spectral gap at
\begin{equation*}
	\mathrm{Re}(\lambda) = \delta - 1
\end{equation*}
and a first conjecture could be that recurrence to this gap which was observed in previous investigations of quantum
resonances on convex-cocompact hyperbolic surfaces is related to properties of resonant states.

By analogy with the application of cycle expansion to resonances it is straightforward to conjecture that symmetry
reduction should improve the rate of convergence of invariant Ruelle distributions even for the (rather small)
Klein four-group used with $Y(10, 10, \pi/2)$. To support this claim the experiment above was repeated with the
trivial one-element symmetry group $\{e\}$.\footnote{
	The $\mathbb{Z}$-symmetry was still factored out, though, which in itself reduces the computational cost
	of dynamical determinant evaluation by quite a margin.
}
Once converged our numerical results should be independent from the symmetry group used: While the approximation
$t^\Sigma_{\lambda_0, \sigma}$ does contain a symmetric version of the Gaussian test functions this should not
matter due to the global invariant Ruelle distribution also being invariant with respect to the full symmetry
group of the surface.

The resulting plots in Figure~\ref{fig_torus_reduced_distributions_trivial} do indeed coincide with those of
Figure~\ref{fig_torus_reduced_distributions} while at the same time exhibiting worse convergence properties:
The distribution at $\lambda_0\approx -0.9999 + 992.4\mathrm{i}$ has not converged until $n_\mathrm{max} = 7$.
While not visible in the figure itself this worsened convergence rate also shows up for the other three resonances
as can be seen by considering the relative errors of order $N$
\begin{equation*}
	\frac{\big| d_{N}^\chi(\lambda_0, 0) \big|}{\big| 1 + \sum_{n = 1}^{N} d_n^\chi(\lambda_0, 0) \big|} ~.
\end{equation*}
Note that these quantities are still functions of the coordinates $(x_-, x_+)$ due to their dependence on the
particular weight function used.\footnote{
	At this point it is slightly inconvenient that in our notation for the individual summands $d_n^\chi$ of
	$d_{f, \chi}$ the weight $f$ is only implicit. Nevertheless we chose this notation to keep the formulae in
	Section~\ref{sym_red} as legible as possible and because the practically used $d_n^\chi$ always depend on
	Gaussian test functions anyways.
}
Two straightforward ways to obtain scalar measures of convergence quality are given by either averaging with
respect to $(x_-, x_+)$ over the whole coordinate domain of the Poincar\'{e} section or to take the maximum
norm. We consistently tracked both variants for all our experiments and it turned out that in all cases
they differed at most by one order of magnitude.

As an example we observed for the resonance $\lambda_0\approx -0.9998 + 9.12\mathrm{i}$ errors
of $3.38\cdot 10^{-8}$ and $1.94\cdot 10^{-8}$ after $n = 4$ iterations with symmetry reduction but
$3.2\cdot 10^{-4}$ as well as $9.82\cdot 10^{-5}$ after $n_\mathrm{max} = 6$ iterations without reduction. This
discrepancy becomes slightly smaller near the first resonances as $\lambda_0 = \delta - 1$ universally shows
the best convergence behavior.

\begin{figure}[h]
	\centering
	\makebox[\textwidth][c]{
		\includegraphics[width=1.5\textwidth,trim={0mm 20mm 0mm 0mm}]{%
			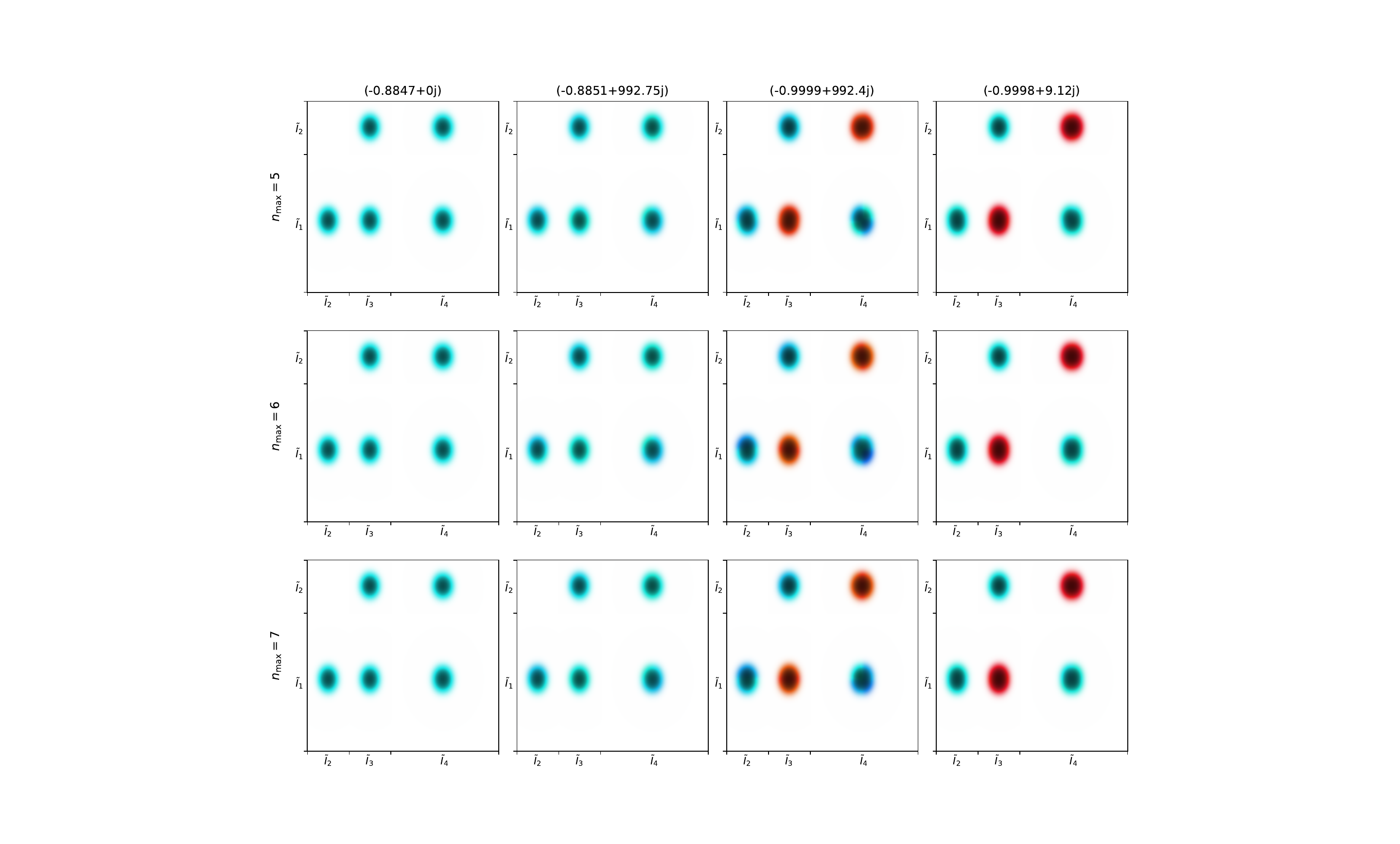}
	}
	\caption{%
		Analogous plots of $t^\Sigma_{\lambda_0, \sigma}$ for $\Sigma$ the canonical Poincar\'{e} section of
		the funneled torus $Y(10, 10, \pi/2)$ and $\sigma = 10^{-3}$ as in
		Figure~\ref{fig_torus_reduced_distributions} but now symmetry reduced by the trivial (one-element)
		group $\{e\}$. The invariant Ruelle distribution is itself invariant under the finite symmetry group
		of the surface so these plots should (and do) coincide with those of
		Figure~\ref{fig_torus_reduced_distributions} apart from the significantly higher cutoff $n_\mathrm{max} = 7$
		required for convergence of the cycle expansion (especially in the third column).
	}\label{fig_torus_reduced_distributions_trivial}
\end{figure}

\subsection*{Three-funneled Surface Experiments}

Next we conduct the analogous experiments for the three-funneled surface $X(12, 12, 12)$. The symmetry group of
the standard set of generators was identified as the Klein four-group in Section~\ref{sym_red.three_funnel}
but this is actually not
the full group of symmetries: In~\cite{Weich.2016} it was demonstrated how a \emph{flow-adapted} representation
of $X(\ell, \ell, \ell)$ yields a strictly larger symmetry group thereby unlocking the full power of symmetry
reduction for this class of surfaces. Without going into the details we state that this technique can be
adapted to the dynamical determinants considered here so we may calculate with straightforward adaptations of
the techniques described above
invariant Ruelle distributions for the flow-adapted three-funnel surfaces. For a comprehensive description of the
theoretical background refer to the first author's PhD thesis~\cite{Schuette.phd}.

\begin{remark}
	Geometrically the flow-adapted representation is defined by gluing two copies of hyperbolic space with three
	disjoint halfplanes removed, see Figure~\ref{fig_flow_adaptation}. This corresponds to a canonical Poincar\'{e}
	section which is far more symmetric compared to the Schottky representation of three-funneled surfaces.
	
	\begin{figure}[h]
		\centering
		\includegraphics[scale=0.55]{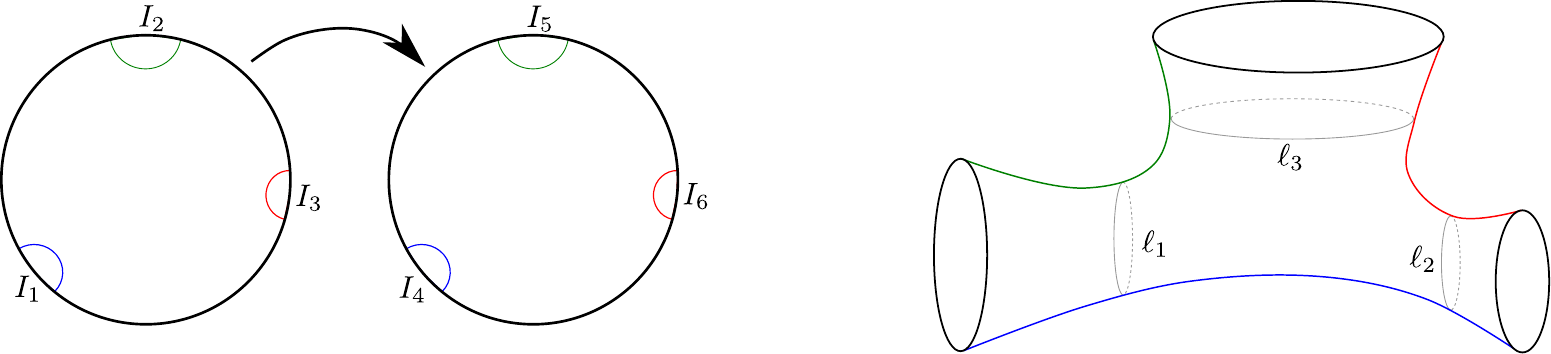}
		\caption{%
			Illustration of the gluing of two copies of the Poincar\'{e} disc $\mathbb{D}$ (left) along the
			colored circles resulting in a
			three-funnel surface (right). The Poincar\'{e} section resulting from this gluing treats the three
			seams along the funnels equally.
		}\label{fig_flow_adaptation}
	\end{figure}
\end{remark}

\begin{figure}[h]
	\centering
	\makebox[\textwidth][c]{
		\includegraphics[width=1.25\textwidth,trim={0mm 20mm 0mm 0mm}]{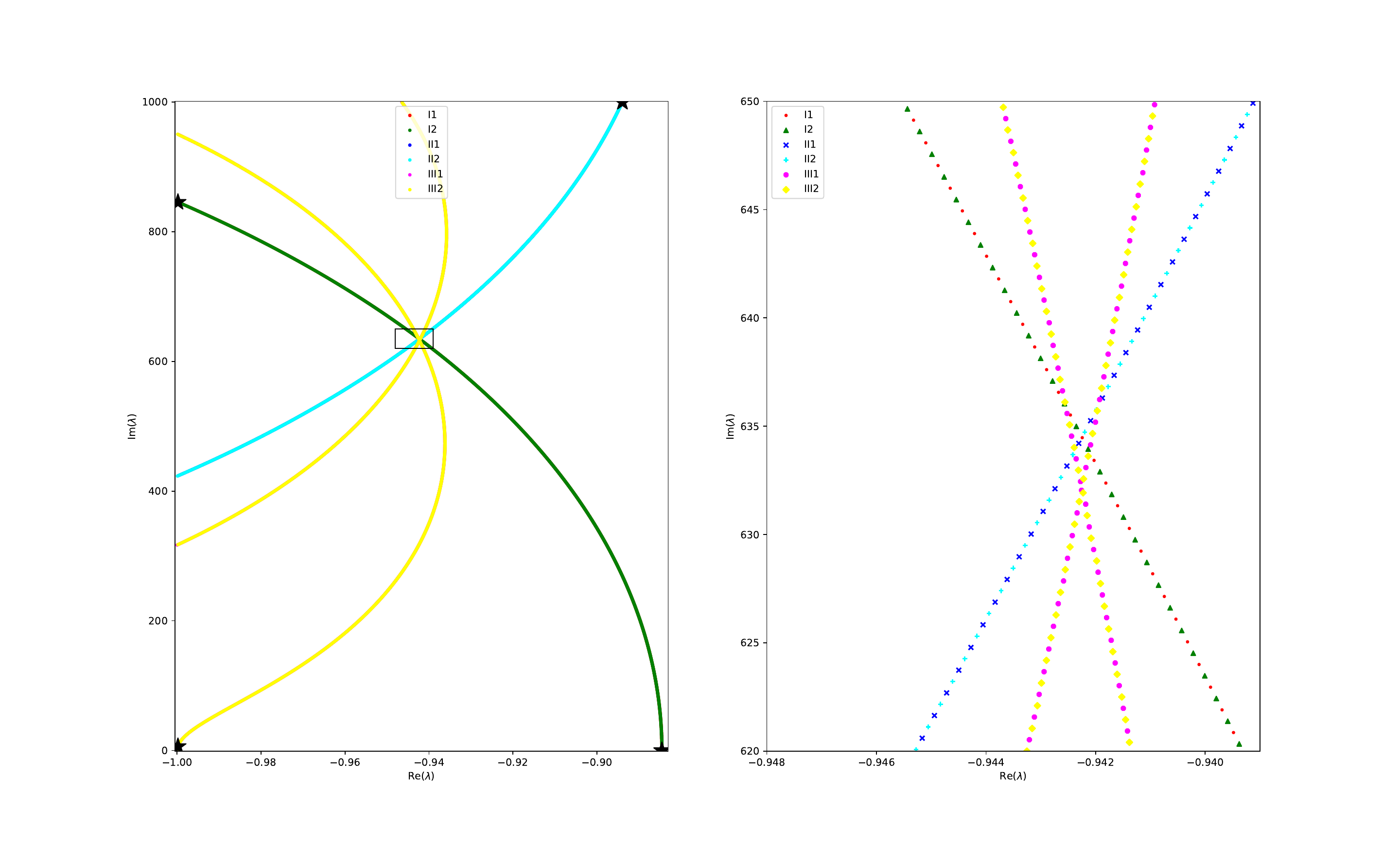}
	}
	\caption{%
		Resonances of the three-funnel surface $X(12, 12, 12)$ calculated in the full symmetry reduction of its
		flow-adapted representation. Due to the high degree of symmetry of the surface four very distinct resonance
		chains are visible. The right-hand side displays a zoom into the region where all four chains cross. It
		clearly shows how every chain belongs to a pair of representations which appear along the chain in
		alternating fashion.
		Again the four resonances considered below were marked with black stars.
	}\label{fig_funnel_resonances}
\end{figure}

As a first step we again
use the dynamical determinant to calculate the Pollicott-Ruelle resonances of this Schottky surface, see
Figure~\ref{fig_funnel_resonances}.
As expected from the quantum-classical correspondence this recovers the (shifted) quantum mechanical resonances
already calculated in~\cite[Figure~6]{Borthwick.2014} as well as the symmetry behavior of individual resonances as
determined in~\cite[Figure~8]{Weich.2016}. Note that the full symmetry group of the flow-adapted three-funneled
surfaces exhibits six irreducible unitary representations resulting in six different classes of resonances.
One also prominently observes the same resonant chains already studied in~\cite{Weich.2016}.

\begin{figure}[h]
	\centering
	\makebox[\textwidth][c]{
		\includegraphics[width=1.10\textwidth,trim={85mm 185mm 65mm 5mm}]{%
			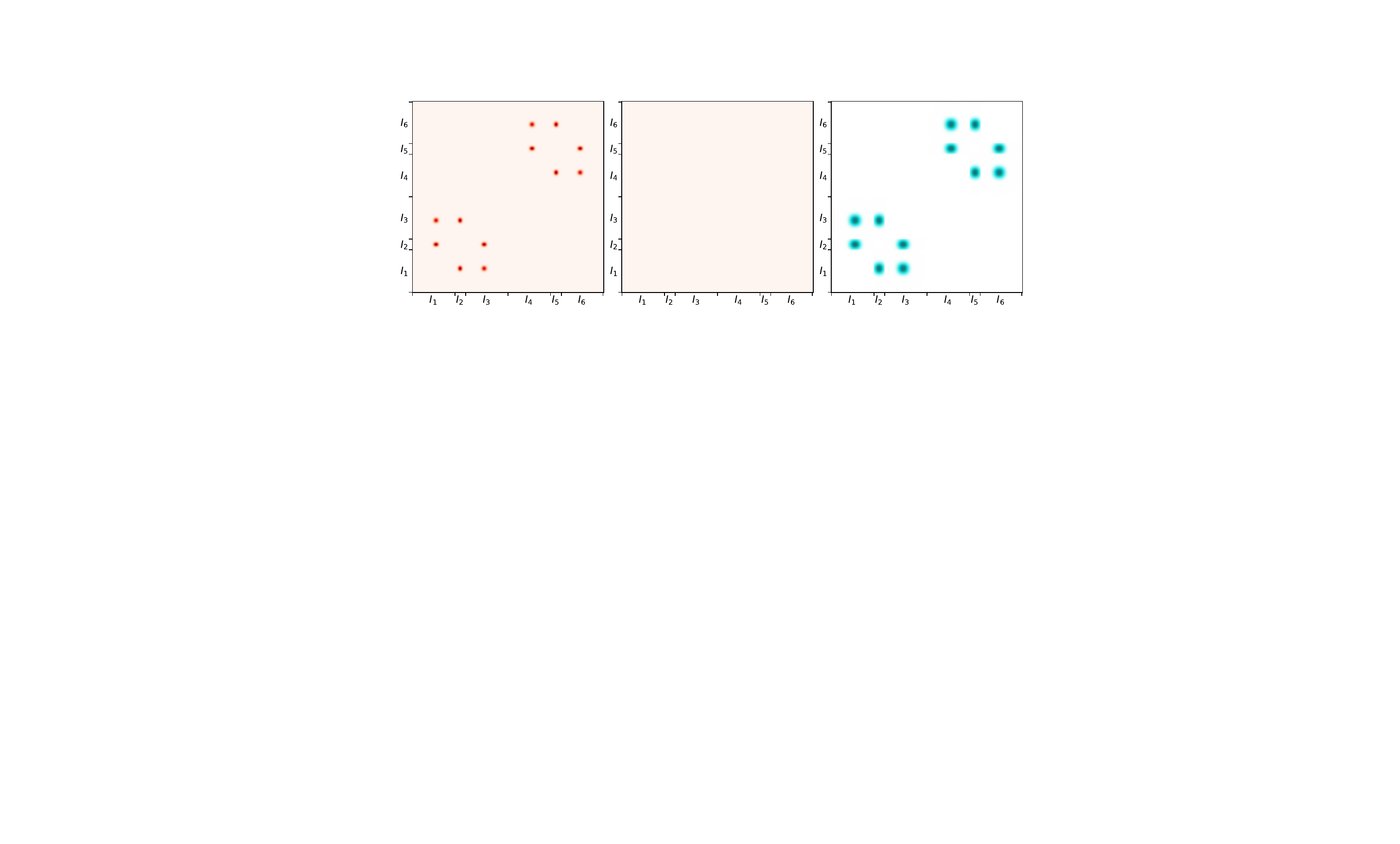}
	}
	\caption{%
		Invariant Ruelle distribution $t^\Sigma_{\lambda_0, \sigma}$ on the canonical Poincar\'{e} section
		$\Sigma$ of the three-funnel surface $X(12, 12, 12)$ associated with the first resonance
		$\lambda_0 \approx 0.8845$. This figure uses the width $\sigma = 10^{-2}$ and $n_\mathrm{max} = 5$ summands
		in the cycle expansion of the dynamical determinant $d_f$. Notice how the finite symmetry group of the
		function system and the surface itself still shows up in the distribution. The plots show (from left to
		right) real part, imaginary part, and argument/absolute value encoding of the distribution similar to
		Figure~\ref{fig_torus_first_distribution}.
	}\label{fig_funnel_first_distribution}
\end{figure}

With these first resonances available to us we can proceed very similarly to the funneled torus case. Again
we compute the smoothed invariant Ruelle distribution $t^\Sigma_{\lambda_0, \sigma}$ on the canonical
Poincar\'{e} section of the flow-adapted system associated with the first resonance
$\lambda_0 = \delta - 1 \approx -0.8845$. The resulting plot is shown
in Figure~\ref{fig_funnel_first_distribution}. Again the plot possesses the theoretically known properties of
being real-valued and positive, as well as showing a clear symmetry with respect to the finite symmetry group
of the underlying function system.

\begin{figure}[h]
	\centering
	\makebox[\textwidth][c]{
		\includegraphics[width=1.5\textwidth,trim={0mm 20mm 0mm 0mm}]{%
			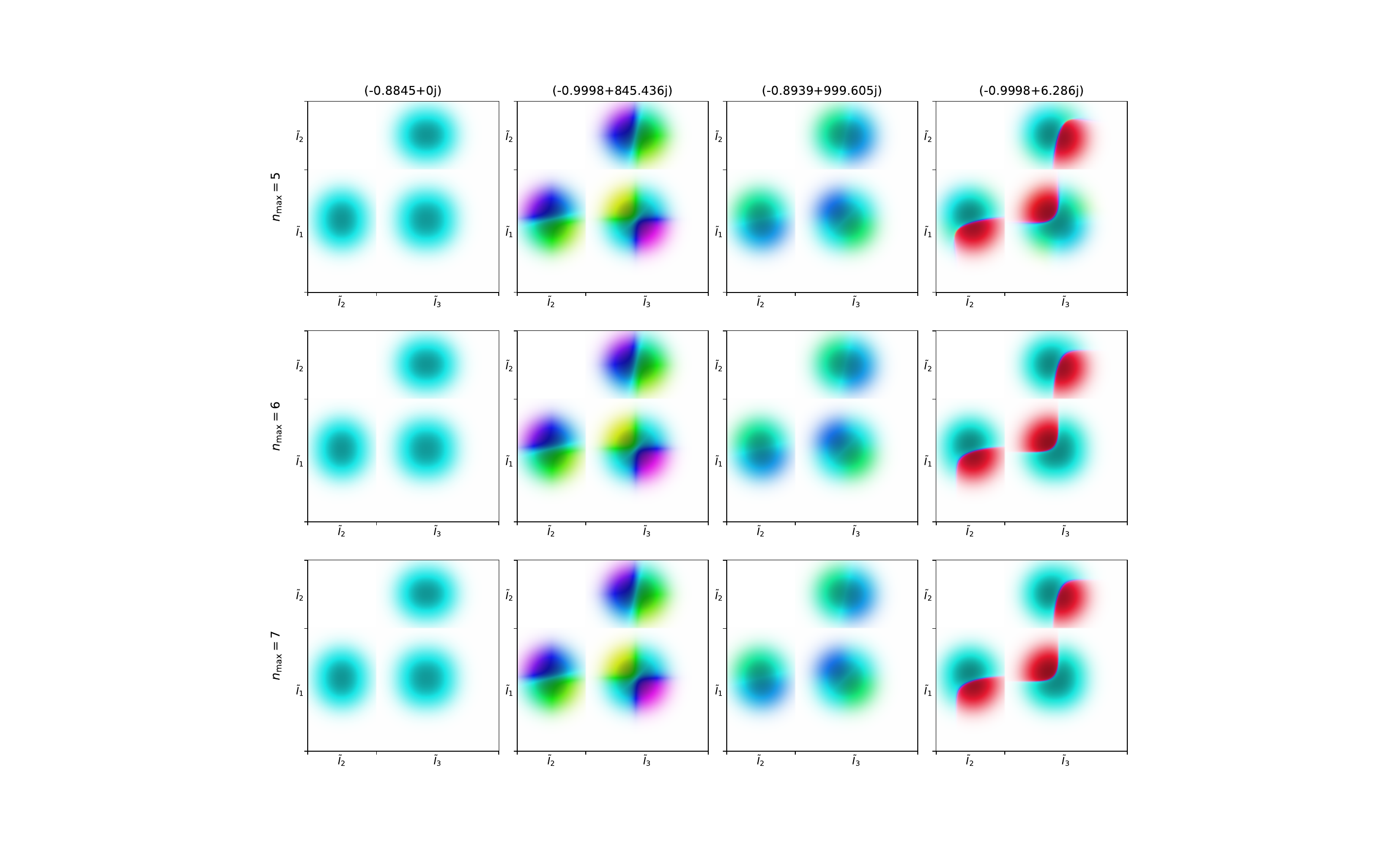}
	}
	\caption{%
		Fully symmetry reduced invariant Ruelle distributions $t^\Sigma_{\lambda_0, \sigma}$ on the canonical
		Poincar\'{e} section $\Sigma$ of the flow-adapted three-funnel surface $X(12, 12, 12)$ evaluated at four
		different choices of resonance $\lambda_0$ with $\sigma = 10^{-2}$ and plotted as phase-lightness.
		The three different rows
		show the numerical results of using (from top to bottom) $n_\mathrm{max}\in \{5, 6, 7\}$ summands in
		the cycle expansion. Following along any of the four columns shows that the presented plots have
		converged rather well already at $n_\mathrm{max} = 5$. Due to symmetries in the distributions as shown
		in Figure~\ref{fig_funnel_first_distribution} it again suffices
		to consider as the coordinate domain refined subsets
		$(\widetilde{I}_2\cup \widetilde{I}_3) \times (\widetilde{I}_1\cup \widetilde{I}_2)$ which are
		shown in Figure~\ref{fig_domain_refinement}.
	}\label{fig_funnel_reduced_distributions}
\end{figure}

As a next step we support the observations made above for the funneled torus with analogous experiments for $X(12, 12, 12)$:
Figure~\ref{fig_funnel_reduced_distributions} contains plots of invariant Ruelle distributions for a set of four different
resonances located roughly on the corners of the resonance plot calculated in full symmetry reduction and the comparison with
Figure~\ref{fig_funnel_reduced_distributions_trivial} which uses the trivial symmetry group $\{e\}$ shows again the great
benefit of symmetry reduction when it comes to the practical rate of convergence, i.e.~the required number of
summands in the cycle expansion. Comparing the funneled torus with the three-funneled surfaces also reveals that
the former requires the determination of more geodesic lengths and orbit integrals to achieve the same relative
errors as the latter even though this difference is not quite as significant as the difference between reduced and
non-reduced calculations.

\begin{figure}[h]
	\centering
	\makebox[\textwidth][c]{
		\includegraphics[width=1.5\textwidth,trim={0mm 20mm 0mm 0mm}]{%
			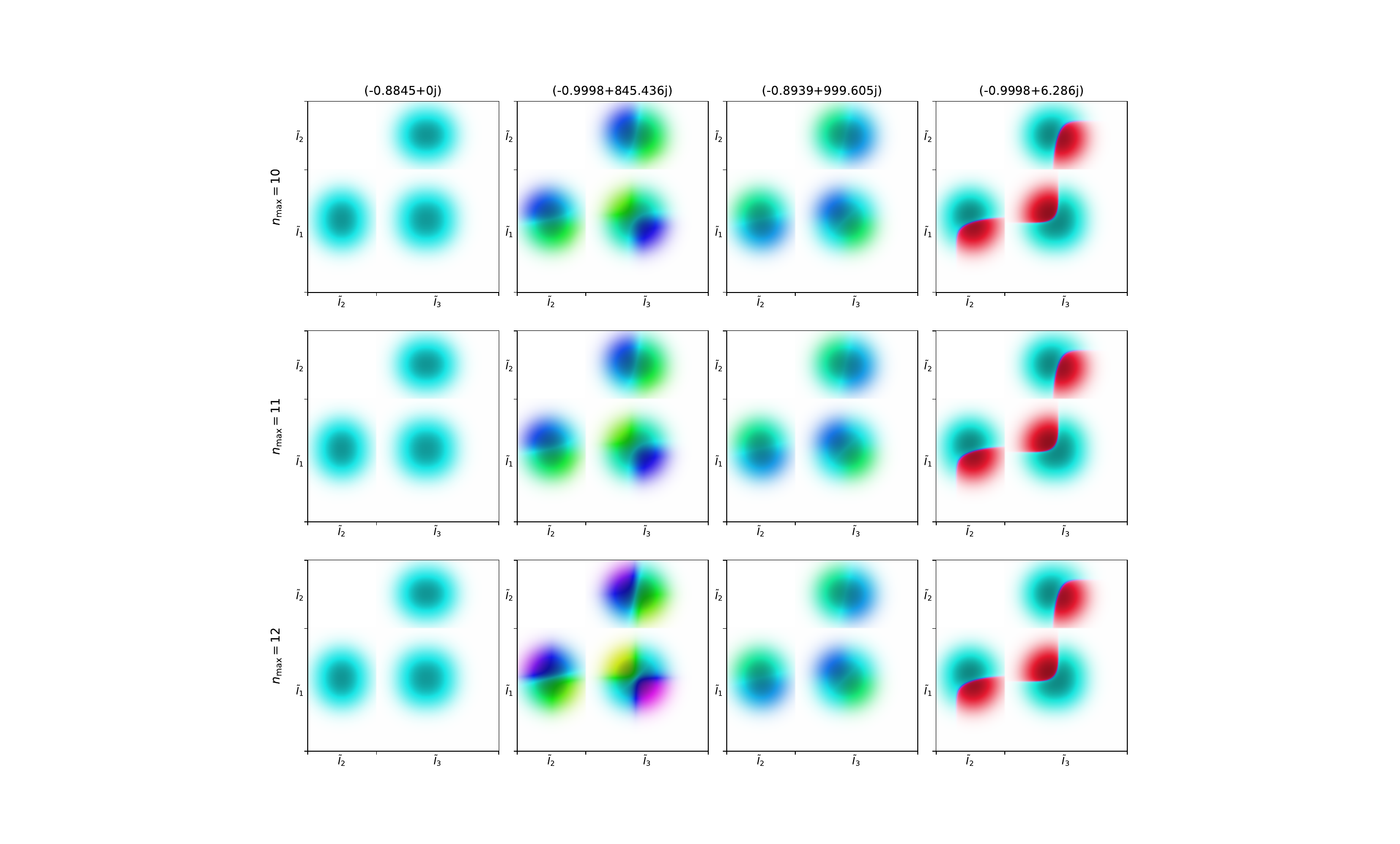}
	}
	\caption{%
		Analogous plots of $t^\Sigma_{\lambda_0, \sigma}$ for $\Sigma$ the canonical Poincar\'{e} section of
		the flow-adapted representation of $X(12, 12, 12)$ and $\sigma = 10^{-2}$ as in
		Figure~\ref{fig_funnel_reduced_distributions} but here the trivial symmetry group $\{e\}$ was used.
		Again the distributions are in close qualitative agreement with those in
		Figure~\ref{fig_funnel_reduced_distributions} but convergence required significantly more terms in
		the cycle expansion as illustrated by the obvious lack of convergence in the second column.
	}\label{fig_funnel_reduced_distributions_trivial}
\end{figure}

We note that the first and third columns of Figure~\ref{fig_funnel_reduced_distributions} which correspond to the first resonance
and the next resonance which is closest to the global spectral gap show far less similarity than observed for the funneled torus
in the first and second columns of Figure~\ref{fig_torus_reduced_distributions}. But this is simply due to the fact that we can
follow the resonance chain corresponding to the representations II1/II2 further to the right and the resonance on this chain
which is closest to the global gap actually turns out to be located quite a bit higher at
$\lambda_0 \approx -0.8845 + 1269.2\mathrm{i}$.\footnote{
	Resonances were calculated with an accuracy of $10^{-4}$ in the real part which is apparently not sufficient to resolve the
	fact that the real part of the chain maximum should be \emph{strictly} smaller than the real part of the first resonance.
}
In Figure~\ref{fig_funnel_resonances_extended} both this resonance chain as well as the invariant Ruelle distribution corresponding
to the chain maximum were plotted. The distribution clearly supports the hypothesis developed for the funneled torus namely that
Ruelle distributions associated with resonances near the global spectral gap should show high qualitative agreement with the one
attached to the first resonance.

The experiment in Figure~\ref{fig_funnel_resonances_extended} also supports another trend with respect to convergence rates: The
practical rates are much more sensitive to the real than the imaginary parts of the associated resonance. Even at the rather high
imaginary part of $\approx 1269.2$ the distributions convergence rapidly in full symmetry reduction with relative errors of magnitude
$\approx 10^{-6}$ after only $n_\mathrm{max} = 4$ iterations. This behavior is rather promising for future experiments involving
resonance chains near the global spectral gap!

\begin{figure}[!htb]
	\centering
	\begin{subfigure}[b]{0.98\textwidth}
		\centering
		\includegraphics[trim={30mm 15mm 30mm 5mm},width=\textwidth]{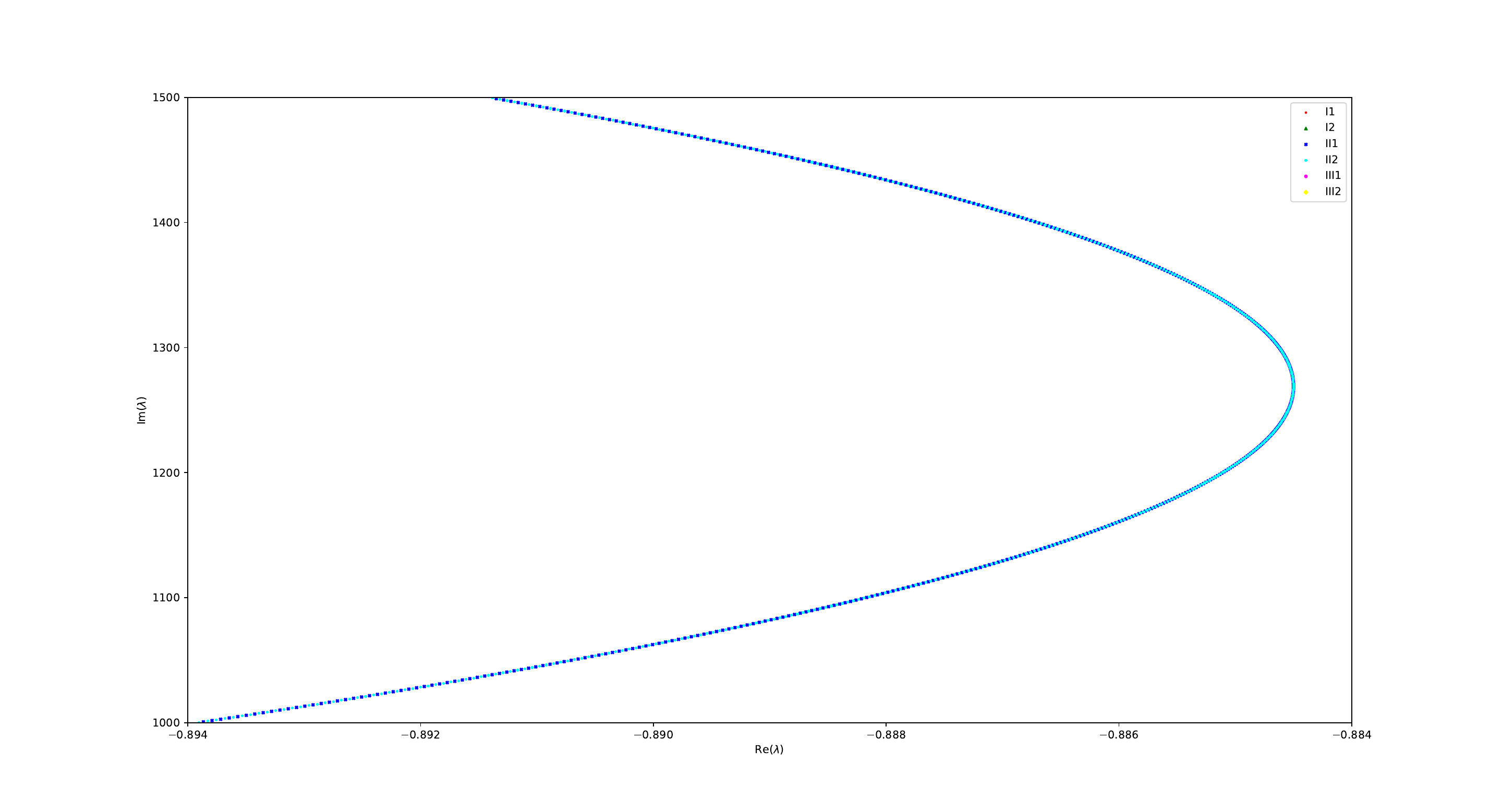}
	\end{subfigure}
	\hfill
	\begin{subfigure}[b]{0.98\textwidth}
		\centering
		\includegraphics[trim={10mm 20mm 10mm 5mm},width=\textwidth]{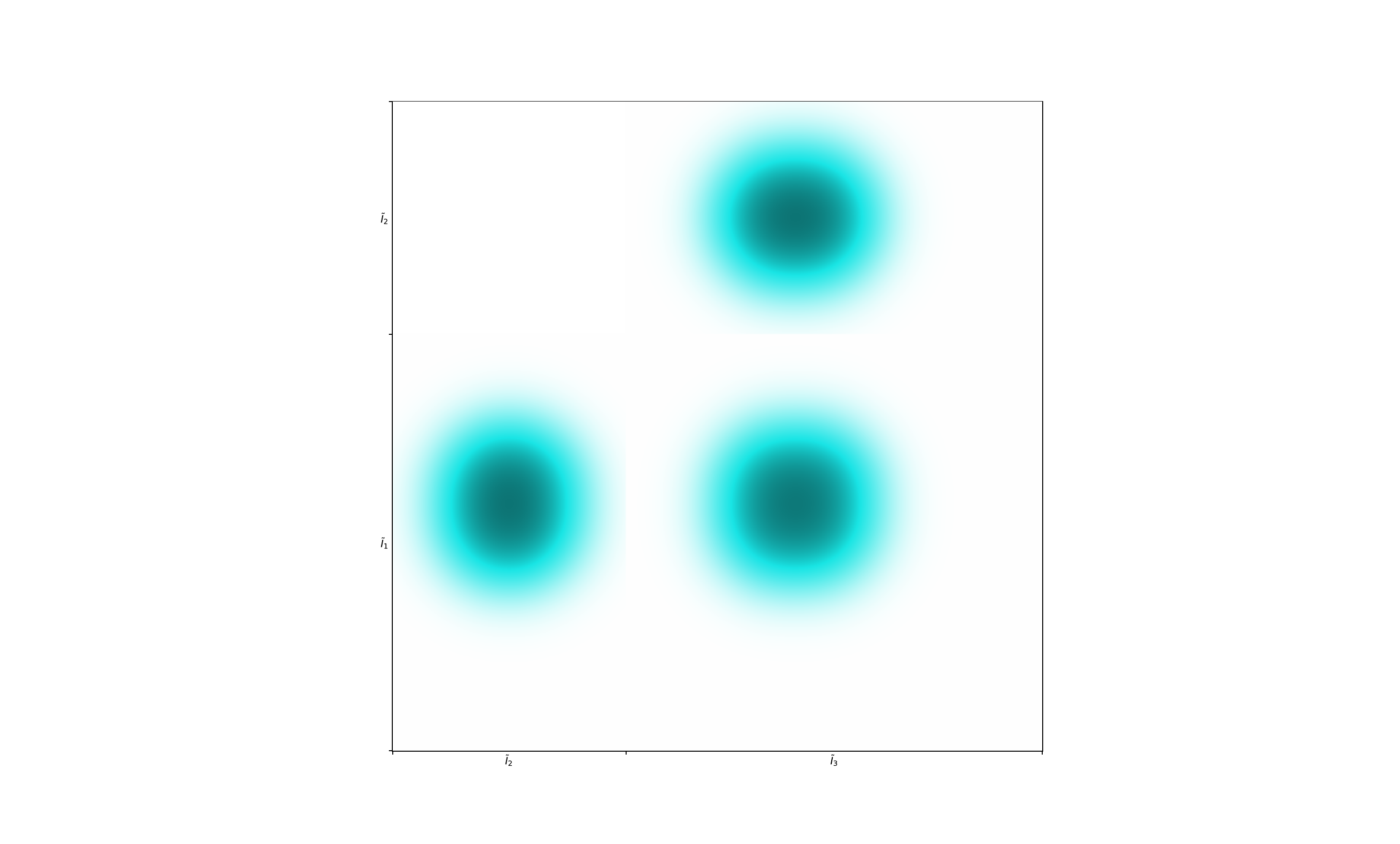}
	\end{subfigure}
	\caption{%
		Additional resonances along the chain of II1/II2 resonances already starting in Figure~\ref{fig_funnel_resonances}
		(top plot) together with the invariant Ruelle distribution at the chain resonance
		$\lambda_0 \approx -0.8845 + 1269.20\mathrm{i}$
		closest to the global spectral gap $\mathrm{Re}(\lambda) = \delta - 1 \approx -0.8845$ (bottom plot). The comparison with the
		first column in Figure~\ref{fig_funnel_reduced_distributions} which uses the same width $\sigma = 10^{-2}$
		shows near perfect qualitative agreement supporting the analogous observation made for the funneled torus above.
	}\label{fig_funnel_resonances_extended}
\end{figure}

The preceding examples should clearly justify the statements made in Section~\ref{sym_red} -- symmetry
reduction does indeed improve the convergence properties of invariant Ruelle distributions drastically.
Furthermore the example of three-funneled surfaces shows that one should use the largest available symmetry
group for a given Schottky surface whenever possible to take maximal advantage of the symmetries of the surface.
Finally refinement of the coordinate domain parameterizing the Poincar\'{e} section is vital if we want
to be able to distinguish
the relevant features of the calculated distributions with sufficient resolution. For systematic experiments in
a regime where the fractal limit set is resolved beyond the first level our ad-hoc procedure can easily be
augmented: After an application of group elements corresponding to a certain word length the initial fundamental
intervals have multiplied according to the given word length but the resulting intervals are also contracted.
One can now choose a subset of these contracted intervals which are representative for the distribution with
respect to the symmetry group of the surface. This effectively increases the resolution of the numerical
approximations without requiring a larger grid of support points on the coordinate domains.

\subsection*{Additional Three-funneled Surface Plots}

Even though already restricted to two-dimensional quantities our approximations $t^\Sigma_{\lambda_0, \sigma}$ of
invariant Ruelle distributions are still complex-valued functions on a two-dimensional domain effectively making
them four-dimensional quantities. Any visualization will
necessarily have to emphasize certain aspects of these distributions over others.
The particular kind of phase-lightness plots used for illustration above showed nicely the qualitative behavior in terms
of phase but the variations in absolute value are not as distinct. The following plots in Figure~\ref{fig_funnel_reduced_distributions_abs}
therefore emphasize this latter aspect by
showing the analogue of Figure~\ref{fig_funnel_reduced_distributions} but as a plot of the absolute value alone.

\begin{figure}[h]
	\centering
	\makebox[\textwidth][c]{
		\includegraphics[width=1.5\textwidth,trim={0mm 20mm 0mm 0mm}]{%
			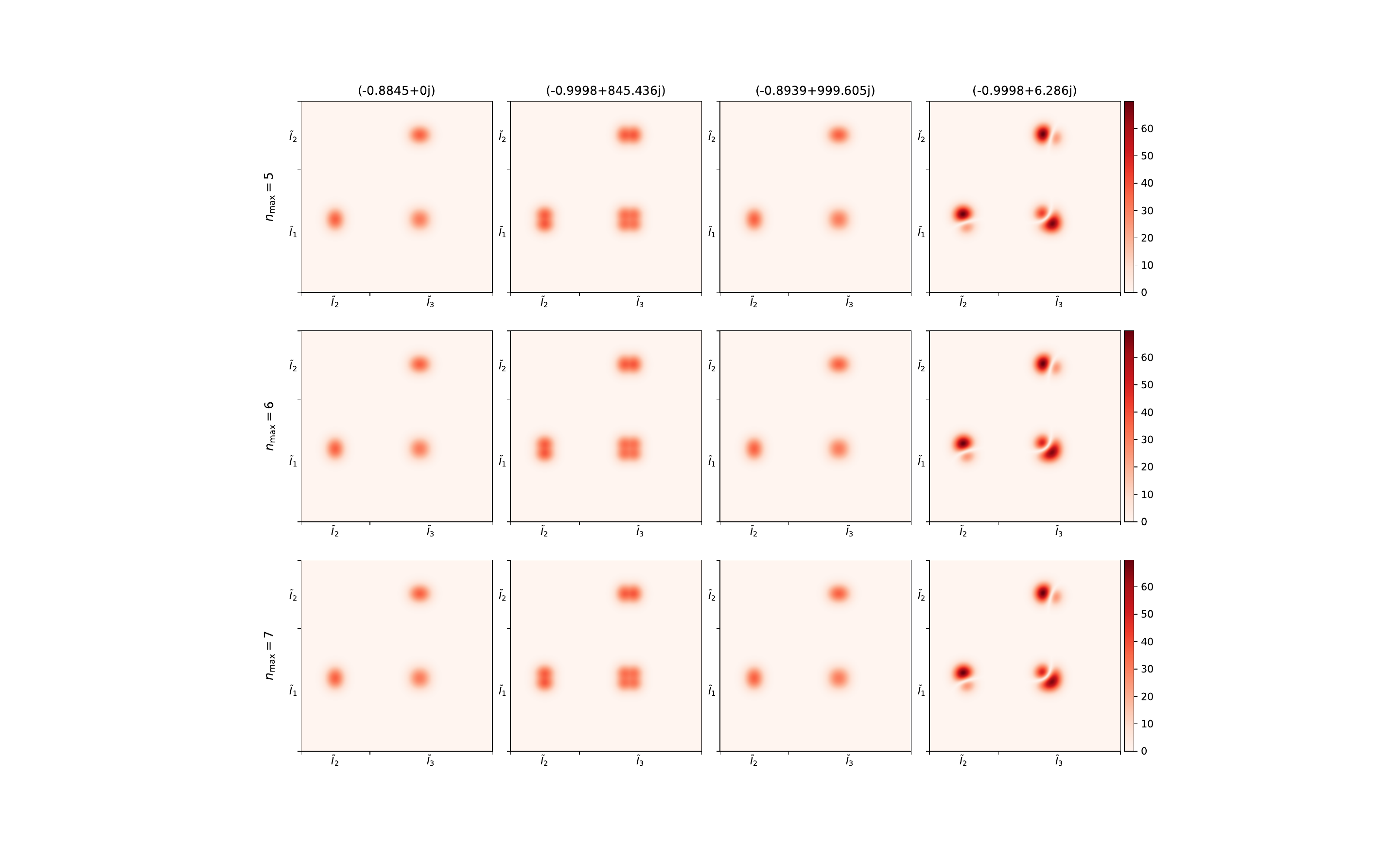}
	}
	\caption{%
		Absolute values of fully symmetry reduced invariant Ruelle distributions $t^\Sigma_{\lambda_0, \sigma}$ on the canonical
		Poincar\'{e} section $\Sigma$ of the flow-adapted three-funnel surface $X(12, 12, 12)$ evaluated at four
		different choices of resonance $\lambda_0$ with $\sigma = 10^{-2}$. Again the different rows
		show the numerical results of using (from top to bottom) $n_\mathrm{max}\in \{5, 6, 7\}$ summands in
		the cycle expansion. Following along any of the four columns shows that the presented plots have
		converged rather well already at $n_\mathrm{max} = 5$. To increase resolution the coordinates
		$(x_-, x_+)$ were again restricted to the refined rectangles
		$(\widetilde{I}_2\cup \widetilde{I}_3) \times (\widetilde{I}_1\cup \widetilde{I}_2)$ of
		Figure~\ref{fig_domain_refinement}.
	}\label{fig_funnel_reduced_distributions_abs}
\end{figure}

As second important facet not yet discussed is the dependence of the qualitative features as well as the convergence behavior of
$t^\Sigma_{\lambda_0, \sigma}$ on the Gaussian width $\sigma$. We therefore plotted the distributions using $\sigma = 10^{-2}$ in
Figure~\ref{fig_funnel_reduced_distributions_abs} again for different values of $n_\mathrm{max}$ to compare them with the
results for $\sigma = 7\cdot 10^{-3}$ in the subsequent Figure~\ref{fig_funnel_reduced_distributions_abs_small}.

\begin{figure}[h]
	\centering
	\makebox[\textwidth][c]{
		\includegraphics[width=1.5\textwidth,trim={0mm 20mm 0mm 0mm}]{%
			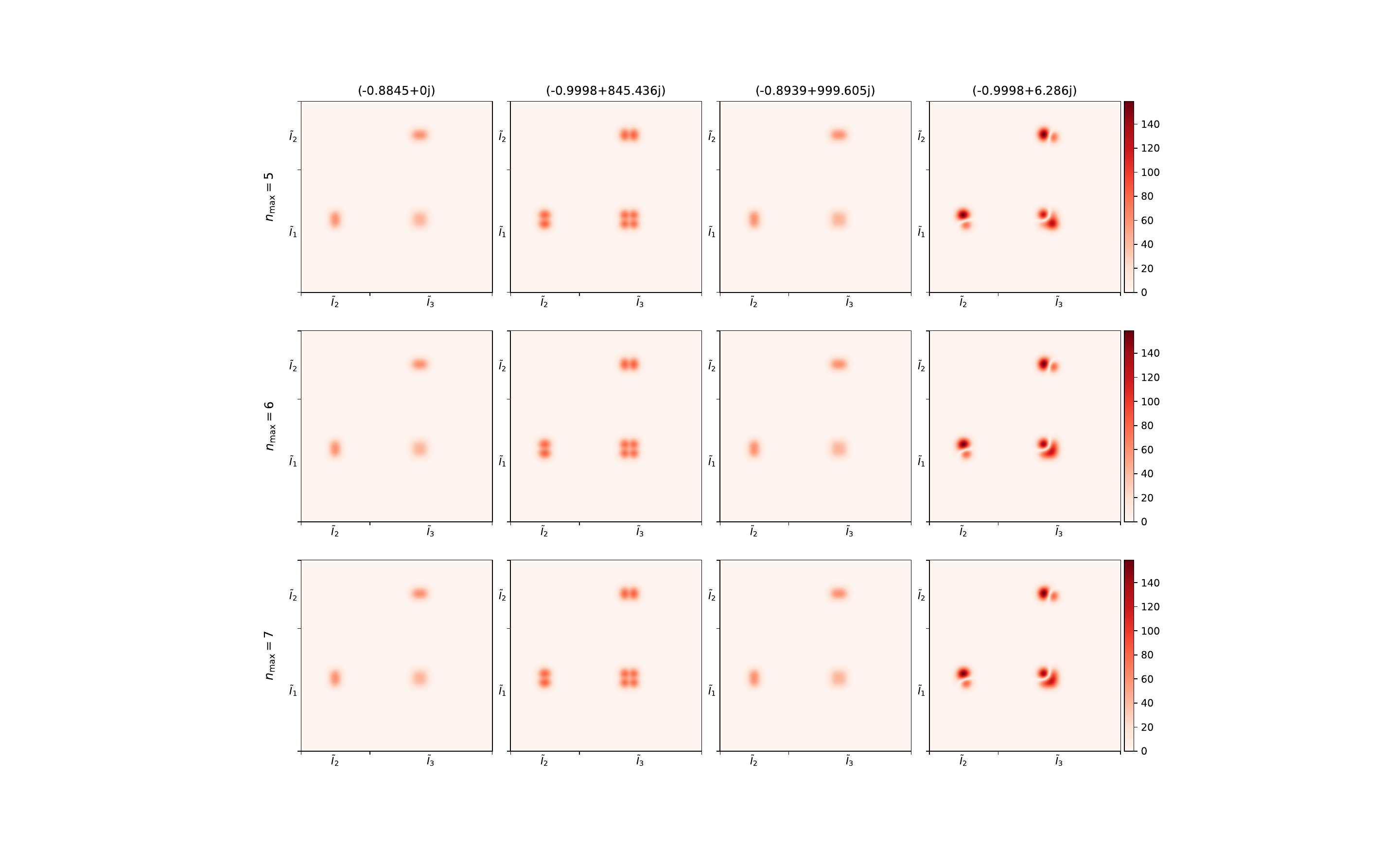}
	}
	\caption{%
		Similar to Figure~\ref{fig_funnel_reduced_distributions_abs} but now with $\sigma = 7\cdot 10^{-3}$.
		Due to the smaller value of $\sigma$ the resolution of the trapped set has become significantly better
		especially in columns two and four but at the expense of the rate of convergence: Now the right-most column
		changes noticeably between $n_\mathrm{max} = 5$ and $n_\mathrm{max} = 6$.
	}\label{fig_funnel_reduced_distributions_abs_small}
\end{figure}

The comparison shows how decreasing the width $\sigma$ brings out additional features in the distributions in particular the
convergence of their supports towards the fractal limit set of the underlying surface. This comes at the expense of decreased
convergence speed, though, as illustrated by the fourth column in Figure~\ref{fig_funnel_reduced_distributions_abs_small}.
The benefits of symmetry reduction should therefore be even more pronounced when combined with experiments of the regime
$\sigma\rightarrow 0$ for distributions associated with resonances far from the global spectral gap.

\begin{figure}[!htb]
	\centering
	\makebox[\textwidth][c]{
		\includegraphics[width=1.5\textwidth,trim={0mm 60mm 0mm 0mm}]{%
			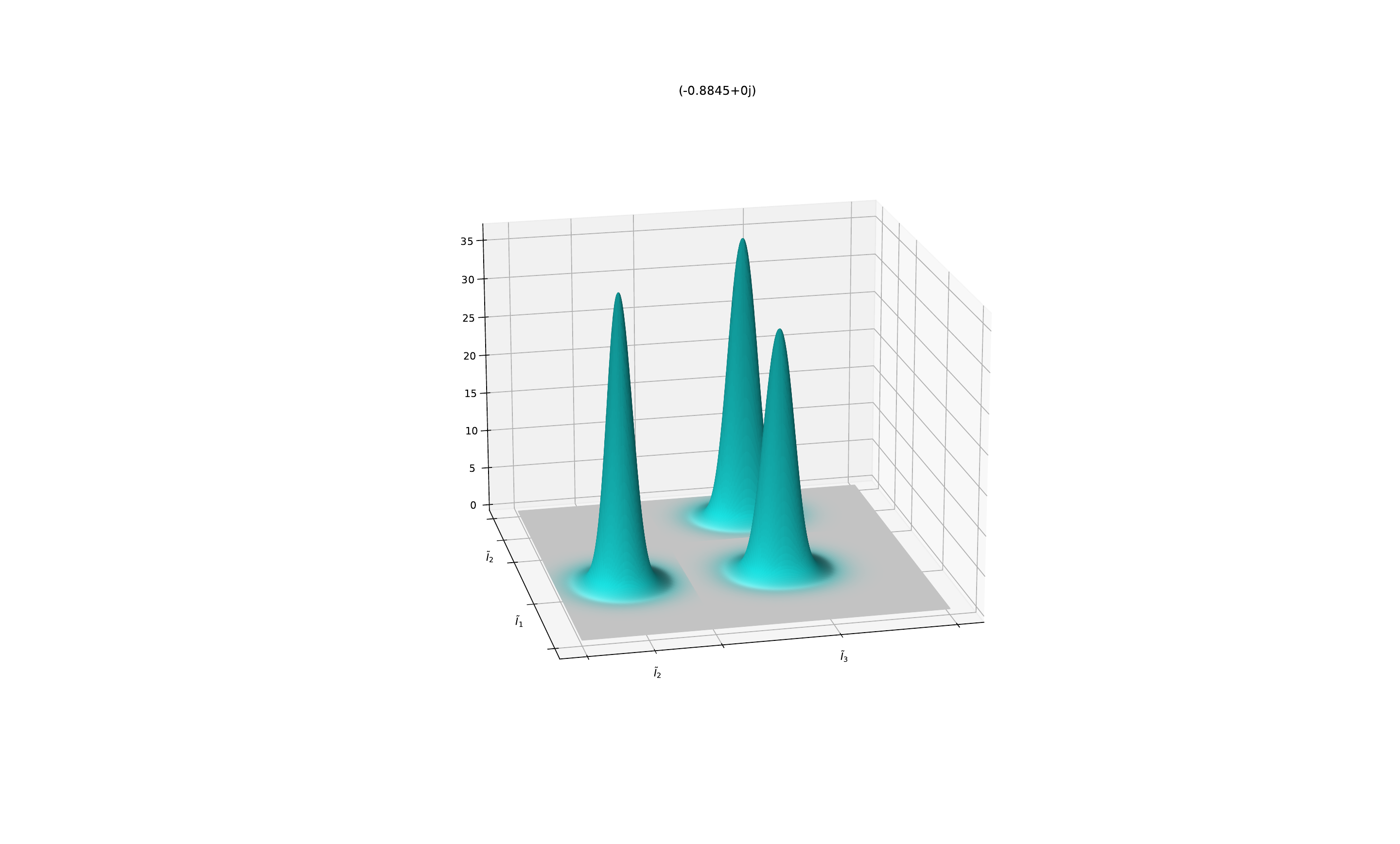}
	}
	\caption{%
		Three-dimensional representation of the invariant Ruelle distribution for the three-funneled surface
		$X(12, 12, 12)$ on the canonical Poincar\'{e} section of the flow-adapted system and with Gaussian width
		$\sigma = 10^{-2}$ associated to the first resonance $\lambda_0 \approx -0.8845$.
		The absolute value of the distribution determines the height of the surface whereas
		the colors indicate the position of the complex argument on the unit circle. This effectively combines the presentation
		in Figures~\ref{fig_funnel_reduced_distributions} and~\ref{fig_funnel_reduced_distributions_abs} emphasizing nicely
		the homogeneity of the first invariant Ruelle distribution.
	}\label{fig_funnel_reduced_distributions_3d_0}
\end{figure}

We finish this first numerical tour of invariant Ruelle distributions by considering a third technique to
illustrate these
that combines the phase-focused and absolute-valued focused approaches taken so far: By increasing the dimension of the plots
themselves we can map the absolute value to the height of a surface in three-dimensional space and simultaneously color this
surface according to the phase (argument) of the complex values. The resulting plot for the first resonance of $X(12, 12, 12)$
is shown in Figure~\ref{fig_funnel_reduced_distributions_3d_0}. While visually more complex due to its
three-dimensional nature it does illustrate the
distribution in quite an intuitive manner and the high homogeneity present in the first invariant Ruelle
distribution is immediately apparent.

\begin{figure}[!htb]
	\centering
	\makebox[\textwidth][c]{
		\includegraphics[width=1.5\textwidth,trim={0mm 60mm 0mm 0mm}]{%
			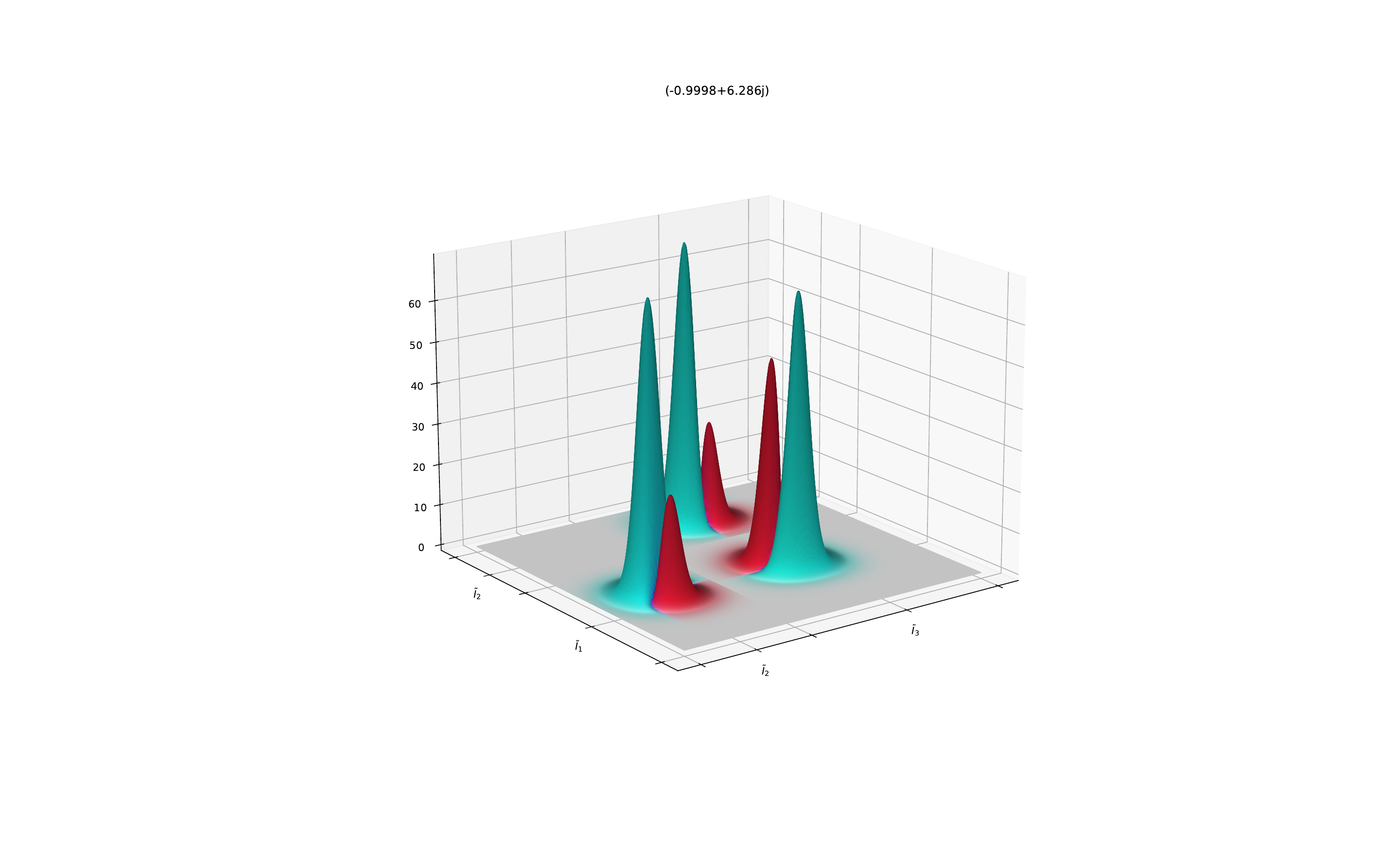}
	}
	\caption{%
		Analogous plot as in Figure~\ref{fig_funnel_reduced_distributions_3d_0} but for the resonance
		$\lambda_0 \approx -0.9998 + 6.286\mathrm{i}$ and again the same Gaussian width $\sigma = 10^{-2}$.
		The three-dimensional plot brings out the splitting of each peak into two sub-peaks of distinctly different
		phases better than the corresponding absolute value plot in the fourth column of
		Figure~\ref{fig_funnel_reduced_distributions_abs}.
	}\label{fig_funnel_reduced_distributions_3d_1}
\end{figure}

The plots contained in Figures~\ref{fig_funnel_reduced_distributions_3d_1} and~\ref{fig_funnel_reduced_distributions_3d_2}
show the three-dimensional representation of a different resonance already considered in
Figures~\ref{fig_funnel_reduced_distributions_abs} and~\ref{fig_funnel_reduced_distributions_abs_small} for the widths
$\sigma = 10^{-2}$ and $\sigma = 7\cdot 10^{-3}$. The splitting of larger peaks into sub-peaks supported on a neighborhood of the limit
set is brought out very clearly as well as the very distinct phases corresponding to different sub-peaks. Decreasing $\sigma$
emphasizes these features further but is less necessary than for the absolute value plots considered previously.

\begin{figure}[h]
	\centering
	\makebox[\textwidth][c]{
		\includegraphics[width=1.5\textwidth,trim={0mm 40mm 0mm 0mm}]{%
			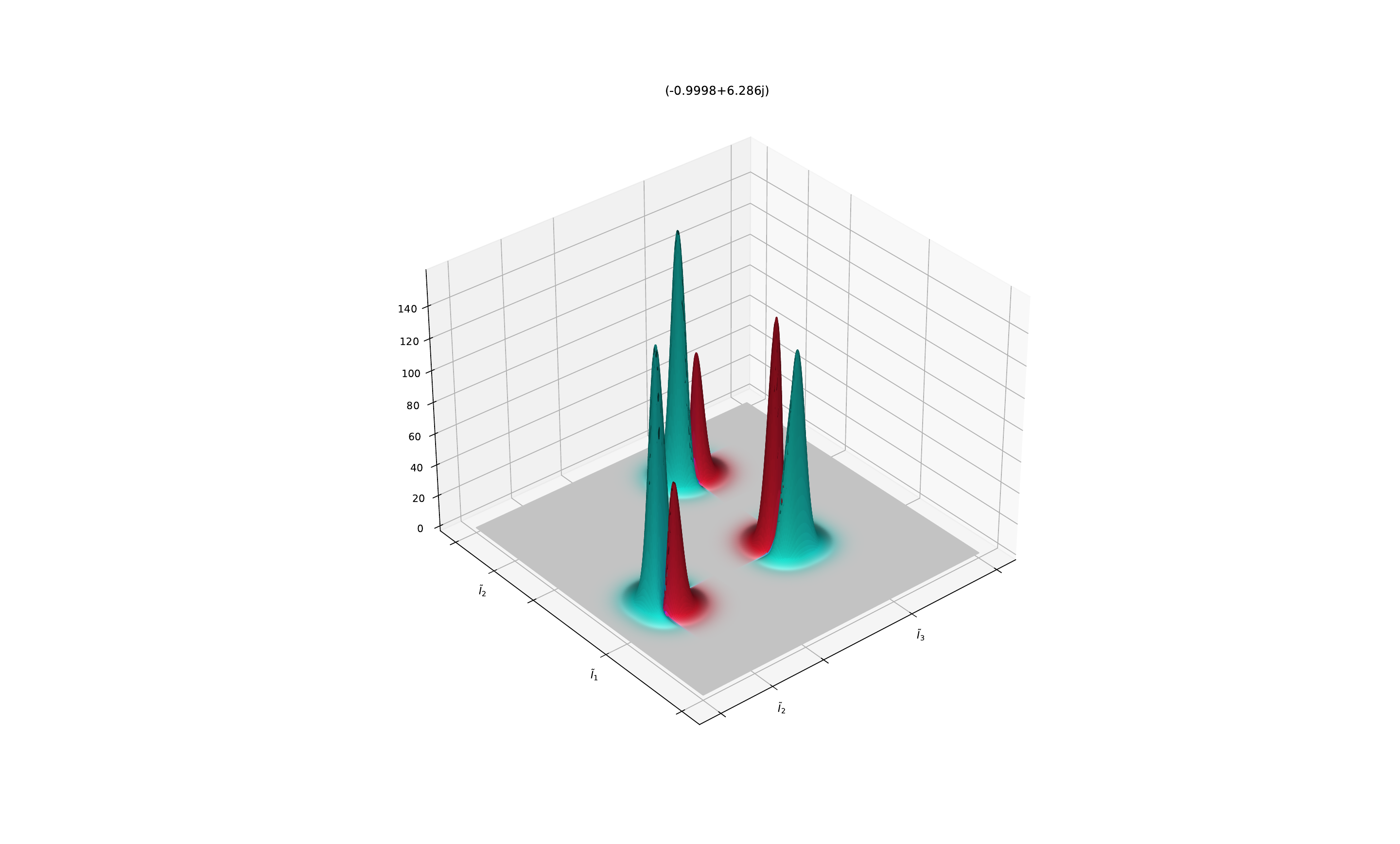}
	}
	\caption{%
		Analogous plot as in Figure~\ref{fig_funnel_reduced_distributions_3d_1} for the same resonance
		$\lambda_0 \approx -0.9998 + 6.286\mathrm{i}$ but a refined Gaussian width of $\sigma = 7\cdot 10^{-3}$. Even this
		slight decrease in $\sigma$ already increases the resolution of distinct sub-peaks with rather prominent differences in
		phase by quite a margin.
	}\label{fig_funnel_reduced_distributions_3d_2}
\end{figure}

\begin{figure}[h]
	\centering
	\makebox[\textwidth][c]{
		\includegraphics[width=1.5\textwidth,trim={0mm 40mm 0mm 0mm}]{%
			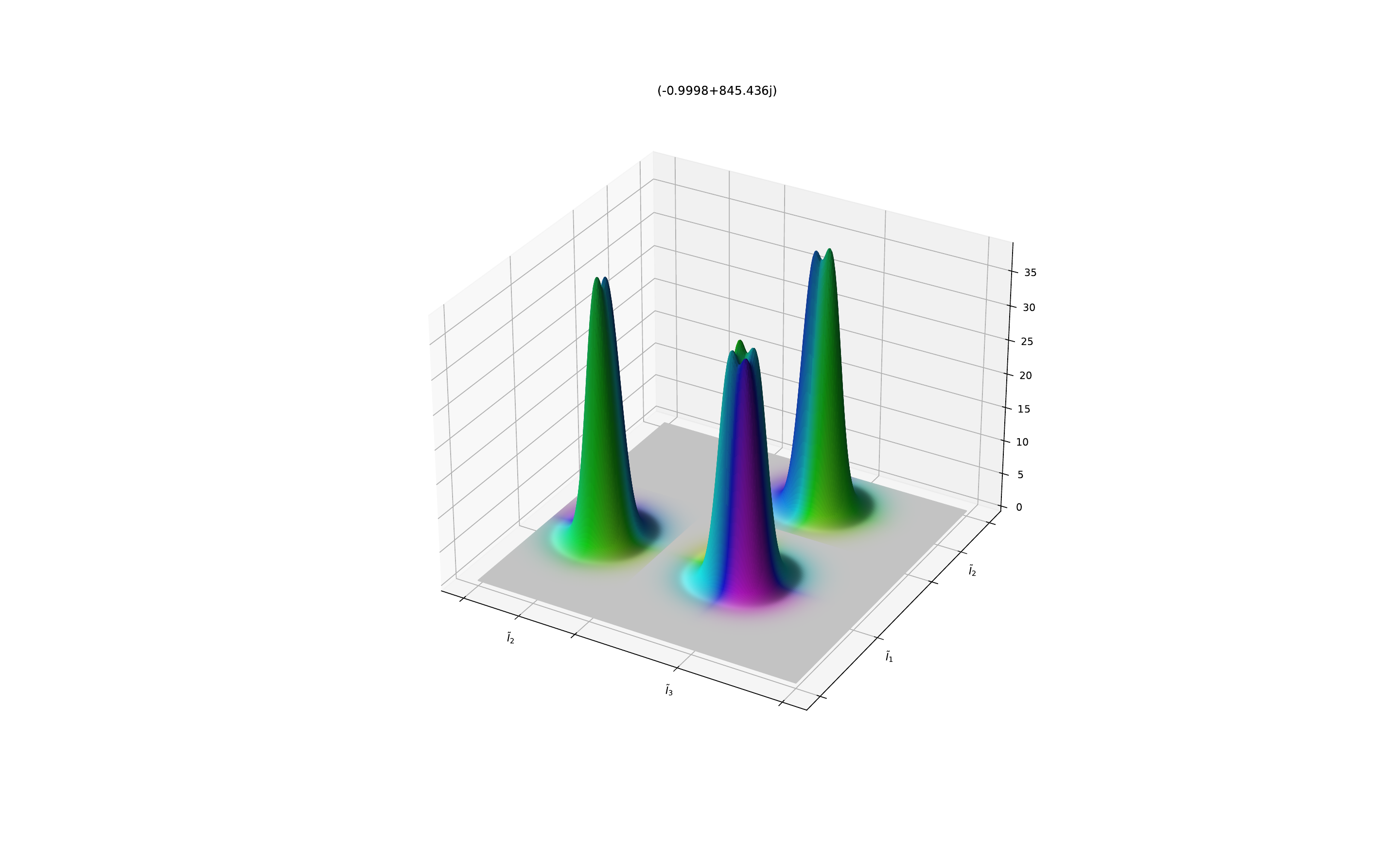}
	}
	\caption{%
		Analogous plot as in Figure~\ref{fig_funnel_reduced_distributions_3d_0} but for the resonance
		$\lambda_0 \approx -0.9998 + 845.436\mathrm{i}$ and again width $\sigma = 10^{-2}$. The splitting into pairs and
		quadruples of sub-peaks is significantly easier to distinguish than in the corresponding absolute value plot shown
		in the second column of Figure~\ref{fig_funnel_reduced_distributions_abs}.
	}\label{fig_funnel_reduced_distributions_3d_3}
\end{figure}

Finally the same experiment was repeated in Figures~\ref{fig_funnel_reduced_distributions_3d_3}
and~\ref{fig_funnel_reduced_distributions_3d_4} with a second resonance of $X(12, 12, 12)$ where one of the major
peaks even splits into four sub-peaks already at $\sigma = 10^{-2}$ and more clearly at $\sigma = 7\cdot 10^{-3}$.

\begin{figure}[h]
	\centering
	\makebox[\textwidth][c]{
		\includegraphics[width=1.5\textwidth,trim={0mm 40mm 0mm 0mm}]{%
			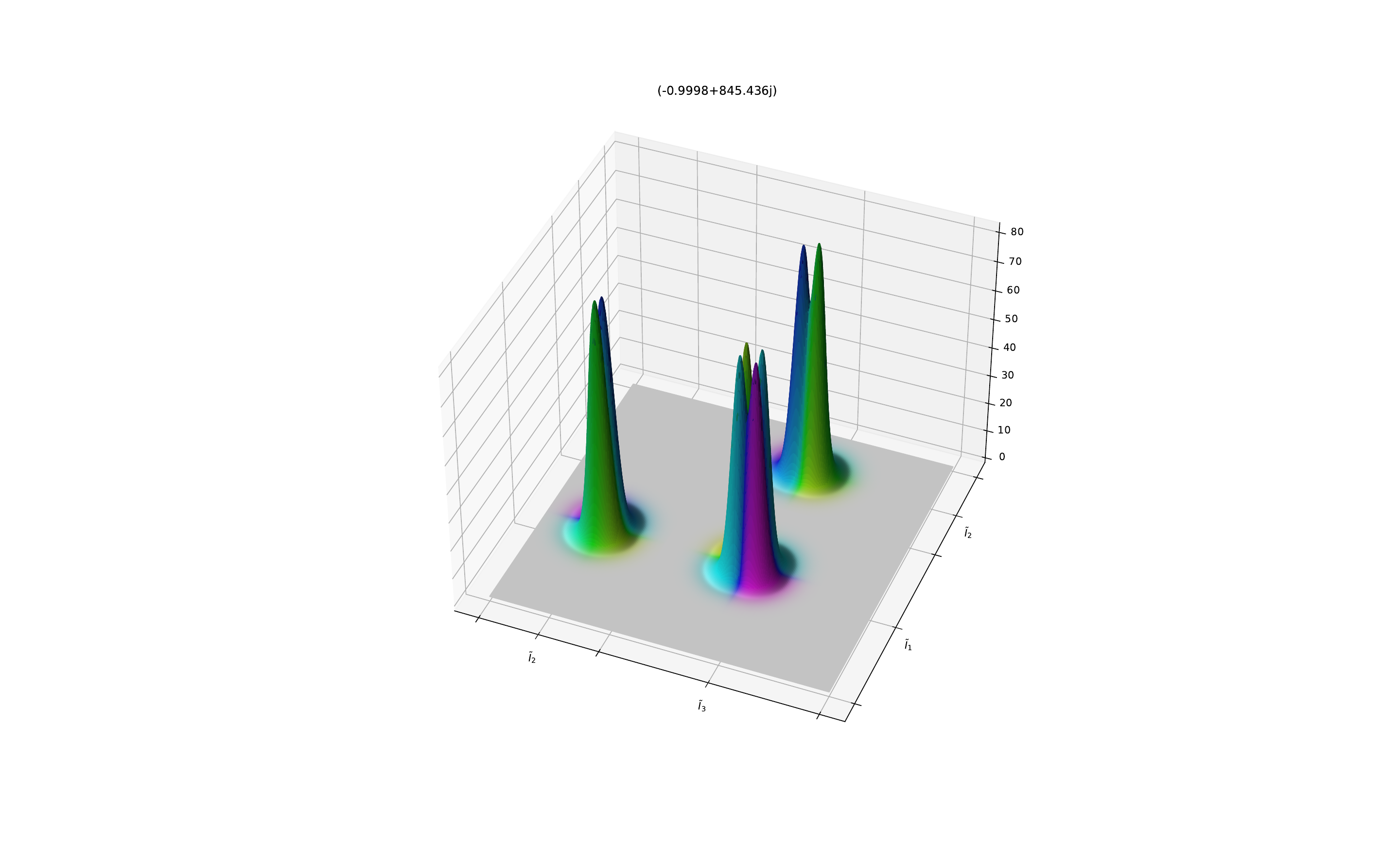}
	}
	\caption{%
		Analogous plot as in Figure~\ref{fig_funnel_reduced_distributions_3d_3} associated with the same resonance
		$\lambda_0 \approx -0.9998 + 845.436\mathrm{i}$ but for a refined Gaussian width of $\sigma = 7\cdot 10^{-3}$.
		Note how both the pairs and the quadruple of peaks are visually more distinct due to the (slight) decrease in $\sigma$.
	}\label{fig_funnel_reduced_distributions_3d_4}
\end{figure}

In summary we observe that the expected limiting behavior where for $\sigma\rightarrow 0$ the numerically
calculated distributions should converge to quantities supported on the (fractal) limit set can indeed be found
in our concrete experiments. We also observed that there are practical limits to the lower bound on $\sigma$
mostly dictated by an inverse relationship between the value of $\sigma$ and the required maximal number of
summands $n_\mathrm{max}$ in the cycle expansion of the dynamical determinant. Here symmetry reduction begins to show its
large potential despite the fact that only a few first experiments were conducted. But even with decently sized values of $\sigma$ the three-dimensional
representations of $t^\Sigma_{\lambda_0, \sigma}$ exhibit a number of interesting features that should provide further
insights into the structure of invariant Ruelle distributions, resonant states, and their associated resonances.

\begin{remark}
	Additional numerical experiments including in particular animations are being prepared and will be included in future
	versions of this article as well as the first author's PhD thesis~\cite{Schuette.phd}.
\end{remark}


\section{Outlook}\label{outlook}

In the present article we derived and implemented an algorithm for the practical numerical computation of invariant Ruelle
distributions. This opens the possibility to perform a large number of different numerical experiments on
which an outlook was already given in the final
Section~\ref{numerics}. More systematic investigations will be conducted in the future to fully leverage the
tools developed here.

Besides such experiments another interesting perspective would be to transfer the methods of~\cite{Pohl.2020} for the
numerical calculation of Fredholm determinants to our (weighted) dynamical determinants $d_f$. These techniques showed very
promising improvements over the computational cost and convergence of cycle expansions in the resonance case and it is expected that
similar improvements are possible with minor adaptations in the case of invariant Ruelle distributions.


\appendix

\section{Fredholm Determinants} \label{app_fredholm_det}

In this appendix we summarize some facts regarding functional analysis in general and Fredholm determinants in particular.
We include this appendix to make the present paper and specifically our proofs more self-contained.
All of the following material is quite standard and the proofs can be found in e.g.~\cite[Appendix~A.4]{Borthwick.2016} or the monograph~\cite[Chap.~1~\&~3]{Simon.2005}.

\begin{defn} \label{def_trace_class}
	Let $T: \mathcal{H}_1 \rightarrow \mathcal{H}_2$ be a compact operator between the Hilbert spaces $\mathcal{H}_1$ and $\mathcal{H}_2$.
	The square roots of the eigenvalues of the self-adjoint operator $T T^*$ are called \emph{singular values} of $T$ and denoted by\footnote{
		Where in this decreasing enumeration multiple eigenvalues are repeated according to their multiplicity.
	}
	\begin{equation*}
	\mu_1(T) \geq \mu_2(T) \geq \hdots \longrightarrow 0 ~.
	\end{equation*}
	If $\mathcal{H}_1 = \mathcal{H}_2$ then $T$ is called \emph{trace-class} in case $\sum_{i = 1}^\infty \mu_i(T) < \infty$ holds.
\end{defn}

For the next definition we denote, analogously to the singular values, the sequence of eigenvalues of $T$ ordered w.r.t. decreasing absolute value and repeated according to their multiplicity by $\lambda_1(T), \lambda_2(T), \hdots$.

\begin{defn} \label{def_abstract_fredholm_determinant}
	Let $T: \mathcal{H} \rightarrow \mathcal{H}$ be a trace-class operator on the Hilbert space $\mathcal{H}$.
	Then the \emph{Fredholm determinant} of $T$ is defined as
	\begin{equation}
	\mathrm{det}\left( \mathrm{id} - zT \right) \defgr \prod_{i = 1}^{\infty} \left( 1 + z\lambda_i(T) \right) ~,
	\end{equation}
	and converges for any $z\in \mathbb{C}$.
\end{defn}

\begin{theorem}
	Let $T: \mathcal{H} \rightarrow \mathcal{H}$ be a compact operator on the Hilbert space $\mathcal{H}$ and denote by $\bigwedge^i T$ the $i$-th exterior power of $T$.
	Then the Fredholm determinant can be calculated via
	\begin{equation*}
	\mathrm{det}\left( \mathrm{id} - zT \right) = \sum_{i = 0}^\infty z^i \cdot \mathrm{Tr}\left( \bigwedge^i T \right) ~,
	\end{equation*}
	where the \emph{trace} of $T$ may be defined as $\mathrm{Tr}(T) \defgr \sum_{i = 1}^\infty \lambda_j(T)$.
\end{theorem}

Combining this theorem with the inequality
\begin{equation*}
\sum_{j = 1}^\infty \vert \lambda_j(T)\vert \leq \sum_{j = 1}^\infty \mu_j(T)
\end{equation*}
one can now prove the analyticity of the Fredholm determinant in the $z$-variable:
The trace which is its $i$-th coefficient in the Taylor expansion about $z = 0$ can be bounded by
\begin{equation*}
\sum_{j_1 < \hdots < j_i} \mu_{j_1}(T) \cdot \hdots \cdot \mu_{j_i}(T) \leq \frac{1}{k!} \left( \sum_{j = 1}^\infty \mu_j(T) \right)^i ~,
\end{equation*}
because $\left( \bigwedge^i T \right) \left( \bigwedge^i T \right)^* = \bigwedge^i \left( T T^* \right)$.
But this already finishes the proof due to $T$ being trace-class.

We end this appendix by stating two well-known estimates regarding singular values:
\begin{theorem}[Min-Max-Theorem] \label{thm_min_max}
	Let $T: \mathcal{H} \rightarrow \mathcal{H}$ be a compact operator on the Hilbert space $\mathcal{H}$.
	Then the $n$-th singular value of $T$ satisfies
	\begin{equation*}
	\mu_n(T) = \min_{\substack{V\subseteq \mathcal{H} \\ \dim(V) = n - 1}} \max_{\varphi\in V^\bot} \frac{\Arrowvert T\varphi \Arrowvert}{\Arrowvert \varphi \Arrowvert} ~.
	\end{equation*}
\end{theorem}
\begin{theorem}[Fan Inequality]
	Let compact operators $S, T: \mathcal{H} \rightarrow \mathcal{H}$ on the Hilbert space $\mathcal{H}$ be given.
	Then the singular values of $S + T$ satisfy
	\begin{equation*}
	\mu_{i + j - 1}(S + T) \leq \mu_i(S) + \mu_j(T) ~.
	\end{equation*}
\end{theorem}


\bibliographystyle{amsalpha}
\bibliography{aux/bibo}

\bigskip


\end{document}

%% file: aux/macros.tex
\NewEnviron{solidbox}{
    \begin{center}
    \par
    \begin{tikzpicture}
    \node[rectangle,minimum width=0.85\textwidth] (m) {
        \begin{minipage}{0.8\textwidth}\BODY\end{minipage}
    };
    \draw (m.south west) rectangle (m.north east);
    \end{tikzpicture}
    \end{center}
}

\NewEnviron{dashedbox}{
    \begin{center}
    \par
    \begin{tikzpicture}
    \node[rectangle,minimum width=0.85\textwidth] (m) {
        \begin{minipage}{0.8\textwidth}\BODY\end{minipage}
    };
    \draw[dashed] (m.south west) rectangle (m.north east);
    \end{tikzpicture}
    \end{center}
}

\DeclarePairedDelimiterX\braket[2]{\langle}{\rangle}{#1 , #2}

\newcommand{\nnorm}[1]{
    {\left\vert\kern-0.25ex\left\vert\kern-0.25ex\left\vert #1
    \right\vert\kern-0.25ex\right\vert\kern-0.25ex\right\vert}
}

\newcommand{\defgr}{\mathrel{\mathop:\!\!=}}




%% file: paper.bbl
\providecommand{\bysame}{\leavevmode\hbox to3em{\hrulefill}\thinspace}
\providecommand{\MR}{\relax\ifhmode\unskip\space\fi MR }
\providecommand{\MRhref}[2]{%
  \href{http://www.ams.org/mathscinet-getitem?mr=#1}{#2}
}
\providecommand{\href}[2]{#2}
\begin{thebibliography}{BPSW20}

\bibitem[AZ07]{Anantharaman.2007}
N.~Anantharaman and S.~Zelditch, \emph{{Patterson–Sullivan distributions and
  quantum ergodicity}}, Annales Henri Poincar\'e \textbf{8} (2007), no.~2,
  361--426.

\bibitem[Bal18]{Baladi.2018}
V.~Baladi, \emph{{Dynamical Zeta Functions and Dynamical Determinants for
  Hyperbolic Maps}}, 1st ed., Ergebnisse der Mathematik und ihrer Grenzgebiete.
  3. Folge / A Series of Modern Surveys in Mathematics, Springer, 2018.

\bibitem[BKL02]{Liverani.2002}
M.~Blank, G.~Keller, and C.~Liverani, \emph{{Ruelle–Perron–Frobenius
  spectrum for Anosov maps}}, Nonlinearity \textbf{15} (2002), no.~6,
  1905--1973.

\bibitem[Bor14]{Borthwick.2014}
D.~Borthwick, \emph{Distribution of resonances for hyperbolic surfaces},
  Experimental Mathematics \textbf{23} (2014), no.~1, 25--45.

\bibitem[Bor16]{Borthwick.2016}
\bysame, \emph{{Spectral Theory of Infinite-Area Hyperbolic Surfaces}}, 2nd
  ed., Progress in Mathematics, Birkhäuser, 2016.

\bibitem[BPSW20]{Pohl.2020}
O.~F. Bandtlow, A.~Pohl, T.~Schick, and A.~Wei{\ss}e, \emph{{Numerical
  resonances for Schottky surfaces via Lagrange--Chebyshev approximation}},
  Stochastics and Dynamics (2020), 2140005.

\bibitem[BSW22]{Schuette.2022a}
S.~Barkhofen, P.~Schütte, and T.~Weich, \emph{{Semiclassical formulae for
  Wigner distributions}}, Journal of Physics A: Mathematical and Theoretical
  \textbf{55} (2022), no.~24, 20.

\bibitem[BW16]{Weich.2016}
D.~Borthwick and T.~Weich, \emph{{Symmetry reduction of holomorphic iterated
  function schemes and factorization of Selberg zeta functions}}, Journal of
  Spectral Theory \textbf{6} (2016), no.~2, 267--329.

\bibitem[CDDP22]{CDDP22}
Mihajlo Ceki{\'c}, Benjamin Delarue, Semyon Dyatlov, and Gabriel~P. Paternain,
  \emph{The {Ruelle} zeta function at zero for nearly hyperbolic 3-manifolds},
  Invent. Math. \textbf{229} (2022), no.~1, 303--394 (English).

\bibitem[CE89]{Eckhardt.1989}
P.~Cvitanovi{\'c} and B.~Eckhardt, \emph{{Periodic-orbit quantization of
  chaotic systems}}, Physical review letters \textbf{63} (1989), no.~8,
  823--826.

\bibitem[Com74]{Comtet.1974}
L.~Comtet, \emph{{Advanced Combinatorics}}, 1st revised and enlarged ed.,
  Springer Dordrecht, 1974.

\bibitem[Dal11]{Dalbo.2011}
F.~Dal'Bo, \emph{{Geodesic and Horocyclic Trajectories}}, 1st ed.,
  Universitext, Springer London, 2011.

\bibitem[DG16]{Dyatlov.2016b}
S.~Dyatlov and C.~Guillarmou, \emph{{Pollicott--Ruelle Resonances for Open
  Systems}}, Annales Henri Poincar{\'e} \textbf{17} (2016), no.~11, 3089--3146.

\bibitem[DG21]{Dang.2021}
Nguyen-Thi Dang and Olivier Glorieux, \emph{{Topological mixing of Weyl chamber
  flows}}, Ergodic Theory and Dynamical Systems \textbf{41} (2021), no.~5,
  1342–1368.

\bibitem[DGRS20]{DGRS20}
Nguyen~Viet Dang, Colin Guillarmou, Gabriel Rivi{\`e}re, and Shu Shen,
  \emph{The {Fried} conjecture in small dimensions}, Invent. Math. \textbf{220}
  (2020), no.~2, 525--579 (English).

\bibitem[DSW21]{Schuette.2021a}
B.~Delarue, P.~Schütte, and T.~Weich, \emph{{Resonances and Weighted Zeta
  Functions for Obstacle Scattering via Smooth Models}}, arXiv:2109.05907
  (2021).

\bibitem[Dya19]{Dya19}
Semyon Dyatlov, \emph{Improved fractal {Weyl} bounds for hyperbolic manifolds
  (with an appendix by {David} {Borthwick}, {Semyon} {Dyatlov} and {Tobias}
  {Weich})}, J. Eur. Math. Soc. (JEMS) \textbf{21} (2019), no.~6, 1595--1639
  (English).

\bibitem[DZ16]{Dyatlov.2016a}
S.~Dyatlov and M.~Zworski, \emph{{Dynamical zeta functions for Anosov flows via
  microlocal analysis}}, Annales de l'ENS \textbf{49} (2016), 543--577.

\bibitem[DZ17]{DZ17}
Semyon Dyatlov and Maciej Zworski, \emph{Ruelle zeta function at zero for
  surfaces}, Invent. Math. \textbf{210} (2017), no.~1, 211--229 (English).

\bibitem[DZ19]{Dyatlov.2019}
S.~Dyatlov and M.~Zworski, \emph{{Mathematical theory of scattering
  resonances}}, 1st ed., Graduate Studies in Mathematics, vol. 200, American
  Mathematical Soc., 2019.

\bibitem[FH04]{Fulton.2004}
W.~Fulton and J.~Harris, \emph{{Representation Theory -- A First Course}}, 1st
  ed., Graduate Texts in Mathematics, Springer New York, NY, 2004.

\bibitem[FT23]{FT23}
Fr{\'e}d{\'e}ric Faure and Masato Tsujii, \emph{Fractal weyl law for the ruelle
  spectrum of anosov flows}, Annales Henri Lebesgue \textbf{6} (2023),
  331--426.

\bibitem[GHW18]{Weich.2018}
Colin Guillarmou, Joachim Hilgert, and Tobias Weich, \emph{{Classical and
  quantum resonances for hyperbolic surfaces}}, Math. Ann. \textbf{370} (2018),
  1231--1275.

\bibitem[GHW21]{Weich.2021}
C.~Guillarmou, J.~Hilgert, and T.~Weich, \emph{{High frequency limits for
  invariant Ruelle densities}}, Annales Henri Lebesgue \textbf{4} (2021),
  81--119.

\bibitem[GR89]{GR89cl}
Pierre Gaspard and Stuart~A Rice, \emph{{Scattering from a classically chaotic
  repellor}}, The Journal of Chemical Physics \textbf{90} (1989), no.~4,
  2225--2241.

\bibitem[Hö03]{Hormander.2003}
L.~Hörmander, \emph{{The Analysis of Linear Partial Differential Operators
  I}}, 2nd ed., Classics in Mathematics, Springer Berlin, Heidelberg, 2003.

\bibitem[JP02]{Jenkinson.2002}
O.~Jenkinson and M.~Pollicott, \emph{Calculating {H}ausdorff dimension of
  {J}ulia sets and {K}leinian limit sets}, American Journal of Mathematics
  \textbf{124} (2002), no.~3, 495--545.

\bibitem[KW20]{KW20}
Benjamin K{\"u}ster and Tobias Weich, \emph{Pollicott-{Ruelle} resonant states
  and {Betti} numbers}, Commun. Math. Phys. \textbf{378} (2020), no.~2,
  917--941 (English).

\bibitem[Mas67]{Maskit.1967}
Bernard Maskit, \emph{{A characterization of Schottky groups}}, Journal
  d'Analyse Math\'{e}matique \textbf{19} (1967), 227--230.

\bibitem[NZ15]{Zworski.2015}
S.~Nonnenmacher and M.~Zworski, \emph{{Decay of correlations for normally
  hyperbolic trapping}}, Invent. math. \textbf{200} (2015), 345--438.

\bibitem[{Pol}85]{Pol81}
Mark {Pollicott}, \emph{{On the rate of mixing of Axiom A flows}}, {Invent.
  Math.} \textbf{81} (1985), 413--426.

\bibitem[PV19]{PV19}
Mark Pollicott and Polina Vytnova, \emph{Zeros of the {Selberg} zeta function
  for symmetric infinite area hyperbolic surfaces}, Geom. Dedicata \textbf{201}
  (2019), 155--186 (English).

\bibitem[Rob03]{Rob03}
Thomas Roblin, \emph{Ergodicity and equidistribution in negative curvature},
  M{\'e}m. Soc. Math. Fr., Nouv. S{\'e}r. \textbf{95} (2003), 96 (French).

\bibitem[{Rue}76]{Rue76}
David {Ruelle}, \emph{{Zeta-functions for expanding maps and Anosov flows}},
  {Invent. Math.} \textbf{34} (1976), 231--242.

\bibitem[Rug92]{Rugh.1992}
H.~H. Rugh, \emph{The correlation spectrum for hyperbolic analytic maps},
  Nonlinearity \textbf{5} (1992), no.~6, 1237--1263.

\bibitem[Rug96]{Rugh.1996}
\bysame, \emph{{Generalized Fredholm determinants and Selberg zeta functions
  for Axiom A dynamical systems}}, Ergodic Theory and Dynamical Systems
  \textbf{16} (1996), no.~4, 805--819.

\bibitem[Sch23]{Schuette.phd}
Philipp Schütte, \emph{{Invariant Ruelle Distributions on Open Hyperbolic
  Systems -- An Analytical and Numerical Investigation}}, PhD Thesis (2023).

\bibitem[Sim05]{Simon.2005}
Barry Simon, \emph{{Trace ideals and their applications}}, 2nd ed.,
  Mathematical surveys and monographs, American Mathematical Society,
  Providence, 2005.

\bibitem[SWB23]{Schuette.2021b}
Philipp Sch{\"u}tte, Tobias Weich, and Sonja Barkhofen, \emph{{Meromorphic
  Continuation of Weighted Zeta Functions on Open Hyperbolic Systems}}, Commun.
  Math. Phys. \textbf{398} (2023), no.~2, 655--678.

\bibitem[Thi07]{Thirion.2007}
Xavier Thirion, \emph{{Sous-groupes discrets de SL(d, R) et equidistribution
  dans les espaces symetriques}}, PhD Thesis (2007).

\bibitem[Tsu10]{Tsujii.2010}
M.~Tsujii, \emph{{Quasi-compactness of transfer operators for contact Anosov
  flows}}, Nonlinearity \textbf{23} (2010), 1495--1545.

\bibitem[Wei15]{Wei15}
Tobias Weich, \emph{Resonance chains and geometric limits on {Schottky}
  surfaces}, Commun. Math. Phys. \textbf{337} (2015), no.~2, 727--765
  (English).

\end{thebibliography}
